\documentclass[reqno,12pt]{amsart}

\usepackage[algosection, titlenumbered, longend, norelsize]{algorithm2e}
 \usepackage{amsmath}
 \usepackage{amsthm}
 \usepackage{amssymb}
 \usepackage{latexsym,longtable}
 \usepackage{graphicx}
 \usepackage{multicol}
 \usepackage{mathrsfs}
 \usepackage{young}
 \usepackage[vcentermath,enableskew,stdtext]{youngtab}
 \usepackage[left=1.08in,right=1.08in,top=1.06in,bottom=1.10in]{geometry}
 \usepackage[all]{xy}
    \SelectTips{cm}{10}     
    \everyxy={<2.5em,0em>:} 
 \usepackage{fancyhdr}      
 \usepackage{multicol}
 \usepackage{multirow}
 \usepackage{enumitem}
 \usepackage{bm}
 \usepackage{stmaryrd}
 \usepackage{etex}
 \usepackage{tikz}
 \usepackage{float}
 \usepackage{caption}
 \usepackage{array}
 \usepackage{colortbl}
 \usepackage{mathtools}
 \usepackage{url}
 \restylefloat{figure}
 \usetikzlibrary{shapes}
 \usetikzlibrary{arrows}
 \usetikzlibrary{positioning}
 \usetikzlibrary{snakes}
 \usetikzlibrary{decorations.pathmorphing}
 \usetikzlibrary{decorations.markings}

\usepackage{amsfonts,epsfig,color}
\makeatletter
\newcommand*{\centerfloat}{%
  \parindent \z@
  \leftskip \z@ \@plus 1fil \@minus \textwidth
  \rightskip\leftskip
  \parfillskip \z@skip}
\makeatother
\newcounter{ctr}
\theoremstyle{plain}
\newtheorem{theorem}{Theorem}[section]
\newtheorem{lemma}[theorem]{Lemma}
\newtheorem{corollary}[theorem]{Corollary}
\newtheorem{proposition}[theorem]{Proposition}

\newtheorem{propdef}[theorem]{Proposition-Definition}

\theoremstyle{definition}
\newtheorem{definition}[theorem]{Definition}

\newtheorem{remark}[theorem]{Remark}
\newtheorem{example}[theorem]{Example}
\newtheorem{myalgorithm}[theorem]{Algorithm}

\newcommand{\ignore}[1]{}

\newcommand{\G}{\ensuremath{\mathcal{G}}}

\newcolumntype{C}[1]{>{\centering\arraybackslash$}p{#1}<{$}}

\newcommand{\mat}[1]{\ensuremath{\left[ %
        \begin{array}{cccccccccccccccccccccccccccc} #1 \end{array}\right]}}
\newcommand{\matb}[1]{\ensuremath{\left[ %
        \begin{array}{@{}C{3mm}@{\ \,}C{3mm}@{}} #1 \end{array}\right]}}
\newcommand{\matd}[1]{\ensuremath{\left[ %
        \begin{array}{@{}C{3mm}@{\ \,}C{3mm}|C{3mm}@{\ \,}C{3mm}@{}} #1 \end{array}\right]}}

\newcommand{\matdnobrac}[1]{\ensuremath{ %
        \begin{array}{@{}C{3mm}@{\ \,}C{3mm}|C{3mm}@{\ \,}C{3mm}@{}} #1 \end{array}}}
\newcommand{\matrow}[1]{\ensuremath{\left[ %
        \begin{array}{@{}cccc@{}} #1 \end{array}\right]}}

\newcommand{\Q}{\ensuremath{\mathcal{Q}}}
\newcommand{\QQ}{\ensuremath{\mathbb{Q}}}

\newcommand{\RR}{\ensuremath{\mathbb{R}}}
\newcommand{\NN}{\ensuremath{\mathbb{N}}}

\newcommand{\sgn}{\text{\rm sgn}}

\newcommand{\U}{\mathcal{U}}
\newcommand{\W}{\ensuremath{\mathcal{W}}}

\newcommand{\ZZ}{\ensuremath{\mathbb{Z}}}

\newcommand{\be}{\begin{equation}}
\newcommand{\ee}{\end{equation}}
\newcommand{\myRes}{\text{\rm Res}}
\renewcommand{\S}{\ensuremath{\mathcal{S}}}
\newcommand{\tsr}{\ensuremath{\otimes}}

\newcommand{\y}{\ensuremath{Y}} 

\newcommand{\tto}{\ensuremath{\rightsquigarrow}}
\newcommand{\tleftright}{\ensuremath{\leftrightsquigarrow}}

\newcommand{\reading}{\text{\rm rowword}}
\newcommand{\creading}{\text{\rm colword}}

\def\Tiny{\fontsize{6pt}{6pt}\selectfont}

\newcommand{\crc}[1]{\ensuremath{\overline{#1}}\vphantom{\underline{\overline{#1}}}}

\newcommand{\stand}{\ensuremath{\text{\rm st}}} 

\newcommand{\rev}{\ensuremath{{\text{\rm rev}}}}

\newcommand{\Iplacst}{\ensuremath{{{I_\text{\rm plac}^\stand}}}}

\newcommand{\Ilamst}[1]{\ensuremath{{{I_{#1}^\text{\rm Lam, \stand}}}}}

\newcommand{\Jlamst}[1]{\ensuremath{{{J_{#1}^\text{\rm Lam, \stand}}}}}

\newcommand{\Irkst}{\ensuremath{{I_{\text{\rm KR}}^\stand}}}  

\newcommand{\Iassafst}[1]{\ensuremath{{I_{#1}^\text{\rm Assaf, \stand}}}}

\newcommand{\Des}{\ensuremath{{\text{\rm Des}}}}
\newcommand{\invi}[1]{\ensuremath{{\text{\rm inv$_{#1}'$}}}}

\newcommand{\Wi}[1]{\ensuremath{{\text{\rm W$_{#1}'$}}}}

\newcommand{\e}{\mathsf}

\newcommand{\SYT}{\text{\rm SYT}}

\newcommand{\lp}{\text{\rm LP}}
\newcommand{\lltlp}{\ensuremath{\text{\rm LLTLP}_8}}
\newcommand{\LSP}{\text{\rm LSP}}
\newcommand{\statb}{\text{\rm stat$4'b$}}
\newcommand{\stat}{\text{\rm stat}}

\newcommand{\stars}{\ensuremath{[\e{abc}]}}
\newcommand{\starsdef}{\ensuremath{[\e{def}]}}

\newcommand{\ver}{\text{\rm Vert}}

\newlength{\cellsize}
\newcommand\tableau[1]{
\vcenter{
\let\\=\cr
\baselineskip=-16000pt \lineskiplimit=16000pt \lineskip=0pt
\halign{&\tableaucell{##}\cr#1\crcr}}}

\newcommand{\tableaucell}[1]{{%
\def \arg{#1}\def \void{}%
\ifx \void \arg
\vbox to \cellsize{\vfil \hrule width \cellsize height 0pt}%
\else \unitlength=\cellsize
\begin{picture}(1,1)
\put(0,0){\makebox(1,1){$#1\vphantom{\crc{#1}}$}}
\put(0,0){\line(1,0){1}}
\put(0,1){\line(1,0){1}}
\put(0,0){\line(0,1){1}}
\put(1,0){\line(0,1){1}}
\end{picture}%
\fi}}

\newcommand\boldtableau[1]{
\vcenter{
\let\\=\cr
\baselineskip=-16000pt \lineskiplimit=16000pt \lineskip=0pt
\halign{&\boldtableaucell{##}\cr#1\crcr}}}

\newcommand{\boldtableaucell}[1]{{%
\def \arg{#1}\def \void{}%
\ifx \void \arg
\vbox to \cellsize{\vfil \hrule width \cellsize height 0pt}%
\else \unitlength=\cellsize
\begin{picture}(1,1)
\put(0,0){\makebox(1,1){$\mathbf{#1\vphantom{\crc{#1}}}$}}
\put(0,0){\line(1,0){1}}
\put(0,1){\line(1,0){1}}
\put(0,0){\line(0,1){1}}
\put(1,0){\line(0,1){1}}
\end{picture}%
\fi}}

\newcommand{\partition}[1]{{\setlength{\cellsize}{1ex} \tiny \tableau{#1}}}
\title{What makes a  D$_0$~graph Schur positive?}


\begin{document}

\author{Jonah Blasiak}
\email{jblasiak@gmail.com}
\address{Department of Mathematics, Drexel University, Philadelphia, PA 19104}
\thanks{This work was supported by NSF Grant DMS-14071174.}

\begin{abstract}
We define a \emph{D$_0$~graph} to be a graph whose vertex set is a subset of permutations of  $n$, with edges
of the form  $\cdots \e{bac} \cdots \tleftright \cdots \e{bca} \cdots$ or $\cdots \e{acb} \cdots \tleftright \cdots \e{cab} \cdots$ (\emph{Knuth transformations}),
or $\cdots \e{bac} \cdots \tleftright \cdots \e{acb} \cdots$ or $\cdots \e{bca} \cdots \tleftright \cdots \e{cab} \cdots$ (\emph{rotation transformations}), such that whenever
the Knuth and rotation transformations at positions $i-1, i, i+1$ are available at a vertex, exactly one of these is an edge.
The \emph{generating function} of such a graph is the sum of the quasisymmetric functions associated to the descent sets of its vertices.
Assaf studied D$_0$~graphs in \cite{Sami2} and showed that they provide a rich source of examples of the D~graphs of \cite{Sami}.
A key construction of \cite{Sami} expresses the coefficient of  $q^t$ in an LLT polynomial as the generating function of a certain D$_0$~graph.
It is known that LLT polynomials are Schur positive \cite{GH}, and experimentation shows that many
D$_0$~graphs have Schur positive generating functions, which suggests a vast generalization of LLT positivity in this setting.

As part of the series of papers \cite{BLamLLT, BF}, we study  D$_0$~graphs
using the Fomin-Greene theory of noncommutative Schur functions \cite{FG}.
We construct a
D$_0$~graph whose generating function is not Schur positive  by solving a linear program related to a certain noncommutative Schur function.
We go on to construct a D~graph on the same vertex set as this  D$_0$~graph.
\end{abstract}

\maketitle

\section{Introduction}
\label{s intro}
In the  90's,
Fomin and Greene  \cite{FG} developed a theory of noncommutative Schur functions and used it  to give
positive combinatorial formulae for  the Schur expansions of a large class of symmetric functions that includes the Stanley symmetric functions and stable Grothendieck polynomials.
Lam \cite{LamRibbon}  later showed how this theory can be adapted to study LLT polynomials.
This paper is part of a series of papers \cite{BLamLLT,BF} which further develop this theory and connect it to the D~graphs of Assaf \cite{Sami}.


LLT polynomials are  a family of symmetric functions defined by Lascoux, Leclerc, and Thibon \cite{LLT}. They play an important role in Macdonald theory and have intriguing connections to Kazhdan-Lusztig theory,  $k$-Schur functions, and plethysm.
LLT polynomials were proven to be Schur positive \cite{LT00, GH} using deep geometric results from Kazhdan-Lusztig theory, but
a fundamental problem remains open:
\[\parbox{13cm}{Find a positive combinatorial formula, in terms of simple tableau-like objects, for the coefficients in the Schur expansion of LLT polynomials.}\]
This problem is particularly important because the
Haglund-Haiman-Loehr formula \cite{HHL} expresses the transformed Macdonald polynomials $\tilde{H}_\mu(\mathbf{x};q,t)$ as a positive sum of LLT polynomials.
Hence a solution to this problem would yield an explicit formula for the Schur expansions of transformed Macdonald polynomials.

To attack this problem, Lam \cite{LamRibbon} showed that expanding LLT polynomials in terms of Schur functions is equivalent to writing certain noncommutative versions of  Schur functions $\mathfrak{J}_\lambda(\mathbf{u})$ as positive sums of monomials in an algebra he called the algebra of ribbon Schur operators.
Unfortunately, this mostly translates the difficulty of computing the Schur expansions to another language but does not make the problem much easier.
This setup does produce some new results, however.
Lam obtains the coefficient of $s_\lambda(\mathbf{x})$ in an LLT polynomial for $\lambda$ of the form  $(a,1^b)$,  $(a,2)$, or $(2,2,1^a)$.
In \cite{BLamLLT}, we obtain a formula for the Schur expansion of LLT polynomials indexed by a 3-tuple of skew shapes, proving and generalizing a conjecture of Haglund \cite{Haglund}.

This paper grew out of an attempt to push these results further.
The approach we pursue here and in \cite{BLamLLT,BF} combines ideas of Fomin-Greene \cite{FG}, Lam \cite{LamRibbon}, and Assaf \cite{Sami, Sami2}.
D$_0$~graphs, which we now define, are central objects in this approach.
These graphs were studied in \cite{Sami2} (with slightly different conventions), and the D~graphs of \cite{Sami} corresponding to LLT polynomials are D$_0$~graphs.

\begin{definition}\label{d intro D0 graph}
Let $\W_n$ denote the set of words of length $n$ with no repeated letter in the alphabet  $\{\e{1,2,\dots,N}\}$.
A \emph{D$_0$~graph} is a graph on a vertex set $W \subseteq \W_n$
with edges colored from the set  $\{2,3,\dots,n-1\}$ such that
\begin{list}{(\roman{ctr})}{\usecounter{ctr} \setlength{\itemsep}{2pt} \setlength{\topsep}{3pt}}
\item each $i$-edge is one of the four pairs
\begin{align*}
&\text{$\{\e{v \, bac \, w}, \e{v \, bca \, w}\}$ or  $\{\e{v \, acb \, w}, \e{v \, cab \, w}\}$ (\emph{Knuth edges}), or} \\[-1mm]
&\text{$\{\e{v \, bac \, w}, \e{v \, acb \, w}\}$ or  $\{\e{v \, bca \, w}, \e{v \, cab \, w}\}$ (\emph{rotation edges})}
\end{align*}
for some letters $a<b<c$ and words $\e{v},\e{w}$ such that $\e{v}$ has length  $i -2$,
\item
each element of $W \cap \{\e{v \, bac \, w}, \e{v \, bca \, w}, \e{v \, acb \, w}, \e{v \, cab \, w}\}$ belongs to exactly one \mbox{$i$-edge}
for all letters $a<b<c$ and words $\e{v},\e{w}$ such that $\e{v}$ has length  $i -2$.
\end{list}
\end{definition}
See Figure \ref{f intro} for an example.  Also, Figure \ref{f D graph k3} (\textsection \ref{ss D0 graphs}) depicts the D$_0$~graph on the vertex set $\S_4$ with only rotation edges.
Knuth and rotation edges correspond to the operators  $d_i$ and  $\widetilde{d_i}$ of \cite{Sami} (see \textsection\ref{ss D0 graphs}).

%
%

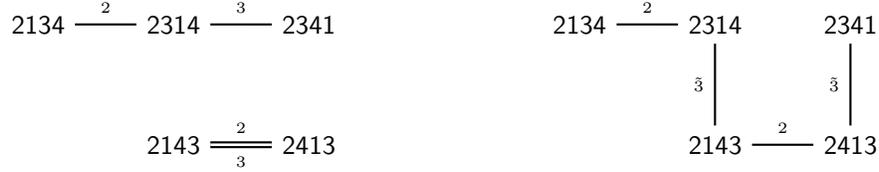
\begin{figure}
\begin{tikzpicture}[xscale = 1.8,yscale = 1.6]
\tikzstyle{vertex}=[inner sep=0pt, outer sep=4pt]
\tikzstyle{aedge} = [draw, thin, ->,black]
\tikzstyle{edge} = [draw, thick, -,black]
\tikzstyle{dashededge} = [draw, very thick, dashed, black]
\tikzstyle{LabelStyleH} = [text=black, anchor=south]
\tikzstyle{LabelStyleV} = [text=black, anchor=east]
\tikzstyle{LabelStyleH2} = [text=black, anchor=north]
\tikzstyle{doubleedge} = [draw, thick, double distance=1pt, -,black]
\tikzstyle{hiddenedge} = [draw=none, thick, double distance=1pt, -,black]

\begin{scope} [xshift = -2cm]
 \node[vertex] (v1) at (1,2){\footnotesize$\e{2134}$};
 \node[vertex] (v2) at (2,2){\footnotesize$\e{2314}$};
 \node[vertex] (v3) at (3,2){\footnotesize$\e{2341}$};
 \node[vertex] (v4) at (2,1){\footnotesize$\e{2143}$};
 \node[vertex] (v5) at (3,1){\footnotesize$\e{2413}$};
\end{scope}
 \draw[edge] (v1) to node[LabelStyleH]{{\Tiny 2}} (v2);
 \draw[edge] (v2) to node[LabelStyleH]{{\Tiny 3}} (v3);
 \draw[doubleedge] (v4) to node[LabelStyleH]{{\Tiny 2}} (v5);
 \draw[hiddenedge] (v4) to node[LabelStyleH2]{{\Tiny 3}} (v5);

\begin{scope} [xshift = 2cm]
 \node[vertex] (v1) at (1,2){\footnotesize$\e{2134}$};
 \node[vertex] (v2) at (2,2){\footnotesize$\e{2314}$};
 \node[vertex] (v3) at (3,2){\footnotesize$\e{2341}$};
 \node[vertex] (v4) at (2,1){\footnotesize$\e{2143}$};
 \node[vertex] (v5) at (3,1){\footnotesize$\e{2413}$};
\end{scope}
 \draw[edge] (v1) to node[LabelStyleH]{{\Tiny 2}} (v2);
 \draw[edge] (v2) to node[LabelStyleV]{{\Tiny $\tilde{3}$}} (v4);
 \draw[edge] (v4) to node[LabelStyleH]{{\Tiny 2}} (v5);
 \draw[edge] (v3) to node[LabelStyleV]{{\Tiny $\tilde{3}$}} (v5);
 \end{tikzpicture}
\caption{\label{f intro} On the left is a D$_0$~graph whose edges are all Knuth edges, and on the right is the D$_0$~graph obtained from this one by changing a pair of its edges to rotation edges.  Knuth  $i$-edges (resp. rotation  $i$-edges) are labeled  $i$ (resp.  $\tilde{i}$).}
\end{figure}

\begin{definition}\label{d gen func}
The \emph{generating function} of a D$_0$~graph $\G$ is
\[\sum_{\e{w} \in \ver(\G)} Q_{\Des(\e{w})}(\mathbf{x}),\]
where $Q_{\Des(\e{w})}(\mathbf{x})$ denotes Gessel's fundamental quasisymmetric function \cite{GesselPPartition}, $\Des(\e{w})$ is the descent set of $\e{w}$, and $\ver(\G)$ denotes the vertex set of $\G$.
\end{definition}

In this paper and the companion paper \cite{BF}, we study the generating functions of D$_0$~graphs using the theory of noncommutative Schur functions.
Extensive computer investigations and results of \cite{Sami2, Sami, BLamLLT} led us to speculate
that D$_0$~graphs always have Schur positive generating functions.
However, this is not true.
In this paper, we solve a linear program to find a D$_0$~graph on a vertex set  $\bar{W}^* \subseteq \S_8$ whose generating function is not Schur positive.
We go on to show that D~graphs in the sense of \cite{Sami} do not always have Schur positive generating functions by
constructing a D~graph on the same set  $\bar{W}^*$.

Given the abundance of Schur positivity in this setting, these examples were quite surprising and demand further investigation of the question
\begin{align}
\label{e question}
\text{What hypotheses must a  D$_0$~graph satisfy to guarantee Schur positivity?}
\end{align}
See Section~\ref{s Concluding remarks} for further discussion.

We now describe the theory that led us to these examples.

\subsection{Noncommutative Schur functions}
\label{ss noncomm sym intro}
The reader may find it helpful here to take a look at the paper \cite{FG} of Fomin and Greene, though we do not formally depend on it in any way.
Let
\[\U = \ZZ\langle u_1,u_2,\dots,u_N \rangle\]
 be the free associative algebra in the noncommuting variables $u_1, u_2, \dots, u_N$.
We frequently write  $\e{a}$ for the variable $u_a$ and think of the monomials of~\,$\U$ as words in the alphabet $\{\e{1,2,\dots,N}\}$.

The \emph{noncommutative elementary symmetric functions} are given by
\begin{equation*}
e_d(\mathbf{u})=\sum_{N \ge i_1 > i_2 > \cdots > i_d \ge 1}u_{i_1}u_{i_2}\cdots u_{i_d} \ \in \, \U
\end{equation*}
for any positive integer $d$; set $e_0(\mathbf{u})=1$ and $e_{d}(\mathbf{u}) = 0$ for $d<0$.

In this paper, we define the \emph{noncommutative Schur functions} by the following noncommutative analog of the Jacobi-Trudi formula:
\[ \mathfrak{J}_\lambda(\mathbf{u}) = \sum_{\pi\in \S_{t}}
\sgn(\pi) \, e_{\lambda'_1+\pi(1)-1}(\mathbf{u}) e_{\lambda'_2+\pi(2)-2}(\mathbf{u}) \cdots e_{\lambda'_{t}+\pi(t)-t}(\mathbf{u}),\]
where  $\lambda$ is any partition and $\lambda' = (\lambda'_1,\dots, \lambda'_t)$ is the conjugate partition of  $\lambda$.

The Fomin-Greene theory of noncommutative Schur functions is related to  D$_0$~graphs by the following algebra.
Define the algebra $\U/\Irkst$ to be the quotient of~\,$\U$ by the relations
\begin{alignat}{3}
&\e{b(ac-ca)} = \e{(ac-ca)b} \qquad &&\text{for $a<b<c$,} \label{e Irkst1} \\
&\e{w} = 0 \qquad &&\text{for words $\e{w}$ with a repeated letter.} \label{e Irkst2}
\end{alignat}
Here,  $\Irkst$ denotes the corresponding two-sided ideal of~\,$\U$; the K stands for Knuth and R for rotation, and ``st'' is short for ``standard word'' to
remind the reader of relation~\eqref{e Irkst2}.

We consider  $\U$ to be endowed with the symmetric bilinear form $\langle\cdot ,\cdot  \rangle$ for which monomials form an orthonormal basis.
We show that a zero-one vector (in the monomial basis) belonging to $(\Irkst)^\perp$ is the same as the sum of the vertices of a  D$_0$~graph (Proposition-Definition \ref{pd KR set}).
We also prove that the elementary symmetric functions $e_d(\mathbf{u})$ commute in $\U/\Irkst$ (Lemma \ref{l es commute}), which is the first step required to apply the Fomin-Greene approach.
Using these facts, we recast a result of Fomin-Greene to show that for any D$_0$~graph $\G$,
\begin{align}\label{e intro coef}
\Big(\text{the coefficient of $s_\lambda(\mathbf{x})$ in the generating function of $\G$}\Big) \, = \, \text{$\Big\langle \mathfrak{J}_\lambda(\mathbf{u}), \sum_{\e{w} \in \ver(\G)}\e{w} \Big\rangle$.}
\end{align}
This identity shows that if there is a positive monomial expression for $\mathfrak{J}_\lambda(\mathbf{u})$ in $\U/\Irkst$ then
the generating function of any  D$_0$~graph is Schur positive.
The main theoretical result of this paper is an approximate converse of this statement, whose precise version we give shortly.
Informally, it says that the Fomin-Greene approach to proving Schur positivity is a powerful one because it only fails
if the symmetric functions to which it is applied are not actually all Schur positive.

Another notable theoretical result of this paper, which is similar to a result of \cite{LamRibbon}, is
\begin{theorem}\label{t intro hook}
The noncommutative Schur function $\mathfrak{J}_\lambda(\mathbf{u})$ is a positive sum of monomials in  $\U/\Irkst$
when $\lambda$ is a hook shape or of the form  $(a,2)$ or  $(2,2,1^a)$.
\end{theorem}

\subsection{Linear programming duality}

We now give the precise statement of the ``approximate converse'' mentioned above.
Fix a positive integer  $d$.
Let  $\e{w^1,\dots,w^m}$ be the words of
$\U$ of length  $d$ (all the words, not just those in  $\W_d$).
Let  $\RR_{\ge 0}^m$ denote the polyhedron $\{\sum_{i} x_i \e{w^i} \mid x_i \ge 0\} \subseteq \RR\U$ and let
$\mathbf{c}_d = \sum_{\e{w} \in \W_d} \e{w} \in \RR_{\ge 0}^m$.

\begin{theorem} \label{t intro linear program 2}
Let  $\lambda$ be a partition  of  $d$.
Let $J^{(1)}, \ldots, J^{(l)}$ be two-sided ideals of~\,$\U$
such that for each  $j \in [l]$,
\begin{itemize}
\item $J^{(j)}$ is generated by monomials (i.e. words in  $\U$) and binomials of the form $\e{v}-\e{w}$, where  $\e{v}$ and $\e{w}$ are words of the same length,
\item  $J^{(j)} \supseteq \Irkst$,
\item $\e{w^i} \in J^{(j)}$ if and only if $\e{w^i} \notin \W_d$, for each word $\e{w^i}$ of length  $d$.
\end{itemize}
Set $J = \bigcap_j J^{(j)}$.
The following are equivalent:
\begin{list}{\emph{(\roman{ctr})}}{\usecounter{ctr} \setlength{\itemsep}{2pt} \setlength{\topsep}{3pt}}
\item[\emph{(i$'$)}] $\mathfrak{J}_\lambda(\mathbf{u})$ is not a real positive sum of monomials in $\RR \tsr_{\ZZ} \U/J$.
\item[\emph{(v$'$)}] There exists an $\mathbf{f} \in \RR_{\ge 0}^m$ such that $\mathbf{c}_d - \mathbf{f} \in \sum_j \big((\RR J^{(j)})^\perp \cap \RR_{\ge 0}^m \big)$ and
\newline $\langle \mathfrak{J}_\lambda(\mathbf{u}), \mathbf{f} \rangle < 0$.
\end{list}
\end{theorem}
These conditions are labelled (i$'$) and (v$'$) to match Theorem~\ref{t linear program 2}, the full version of Theorem~\ref{t intro linear program 2}.
Note that if (v$'$) holds, then $\mathbf{f}$ lies in $(\RR J)^\perp \cap \RR_{\geq 0}^m \subseteq (\RR \Irkst)^\perp \cap \RR_{\geq 0}^m$ and hence by \eqref{e intro coef} is (a  $\RR_{\ge 0}$-weighted version of) the sum of vertices of a  D$_0$~graph whose generating function is not Schur positive.
The proof of this theorem uses linear programming duality.

\subsection{D$_0$~graphs and D~graphs are not always Schur positive}

One of the main goals of this project \cite{BLamLLT,BF} was to find a positive expression for $\mathfrak{J}_\lambda(\mathbf{u})$ in $\U/\Irkst$.
By applying \eqref{e intro coef} to certain  D$_0$~graphs constructed in \cite{Sami}, this would give an elegant and uniform formula for the Schur expansion of LLT polynomials.
It is possible to carry this out for some $\lambda$ and in certain quotients of~\,$\U/\Irkst$ (e.g. Theorem~\ref{t intro hook} and \cite[Theorem 4.1]{BLamLLT}).
The main purpose of this paper, however, is a negative result which shows that for $\lambda = (2,2,2,2)$, both conditions (i$'$) and (v$'$) in Theorem~\ref{t intro linear program 2} hold.
We exhibit a set $\bar{W}^* \subseteq \S_8$ such that  $\mathbf{f} := \sum_{\e{w} \in \bar{W}^*} \e{w} \in (\Irkst)^\perp$ and
 $\langle \mathfrak{J}_{(2,2,2,2)}(\mathbf{u}), \mathbf{f} \rangle = -1$ (we do not know in general whether $\mathbf{f}$ can always be chosen to be a zero-one vector if (v$'$) holds, but it happens to be so for our choice of  $\lambda$ and  $J^{(j)}$ here).
Hence there is a D$_0$~graph $H^*$ on the vertex set $\bar{W}^*$ whose generating function is not Schur positive.


As already mentioned, this example was quite surprising given our speculation based on computer experimentation and \cite{Sami2,Sami,BLamLLT}.
It naturally led us to the question \eqref{e question} above, particularly focusing on whether axioms from \cite{Sami} are enough to guarantee Schur positivity.
Though we will not go into the details of these axioms until \textsection\ref{ss D graphs},
we briefly mention their relation to D$_0$~graphs.  A D~graph as defined in \cite{Sami} is a colored graph (with some extra data)
that satisfies the following properties from that paper: axioms 1, 2, 3, 4$'a$, 4$'b$, and 5 and  $\LSP_4$ and  $\LSP_5$.
Most of these properties are clearly satisfied for  D$_0$~graphs, but some require substantial proof and two of them do not always hold.
For instance, D$_0$~graphs  do not always satisfy axiom 5, which states that $i$-edges commute with $j$-edges for $|i-j| \geq 3$.
Using Theorem~\ref{t intro hook}, we prove that a D$_0$~graph is a D~graph if and only if it satisfies axioms  $4'b$ and 5.
A similar result is proved in \cite{Sami2} using different methods.

After  finding the example above, we investigated whether imposing axiom~5 is enough to guarantee Schur positivity.
In \textsection\ref{ss forcing axiom 5}, we describe an algorithm that
starts with the vertex set $\bar{W}^*$ and grows a  D$_0$~graph on this set by
adding edges that are forced by axiom~5.
 This produces a  D$_0$~graph satisfying axiom~5 whose generating function is not Schur positive, so this axiom is not enough.
The first graph we produced with this method was not a D~graph because it did not satisfy axiom~$4'b$.
With a considerably more intricate algorithm, we were able to produce a graph $\G^*$ on a subset of $\bar{W}^*$ satisfying axioms 5 and $4'b$.
It has 4950 vertices and is described in the accompanying data files.
Moreover, it satisfies a slightly stronger version of axiom~4$'b$, which we call axiom~4$''b$ (see \textsection\ref{ss axiom 4''b}).
\begin{theorem}
There exists a D$_0$~graph $\G^*$ satisfying axioms 4$''b$ and 5 whose generating function is not Schur positive.
Hence D~graphs do not always have Schur positive generating functions.
\end{theorem}

\subsection{Organization}
Section~\ref{s Noncommutative symmetric functions} gives a more thorough introduction to the theory of noncommutative Schur functions just discussed in \textsection\ref{ss noncomm sym intro}.
In Section~\ref{s D0 graphs and D graphs} we define D$_0$~graphs and the D~graphs of Assaf and relate D$_0$~graphs to the algebra $\U/\Irkst$.
Section~\ref{s Basic properties of D0 graphs} contains some important theoretical results: the full statement and proof of Theorem~\ref{t intro hook} and the corollary that  a D$_0$~graph is a D~graph if and only if it satisfies axioms $4'b$ and 5.
In Section~\ref{s LLT} we show how LLT polynomials can be expressed as the generating functions of certain D$_0$~graphs and relate our setup to Lam's \cite{LamRibbon}
and the graphs $\G^{(k)}_{c,D}$ defined by Assaf in \cite{Sami}.
In Section~\ref{s Linear programming and Schur positivity} we state and prove the full version of Theorem~\ref{t intro linear program 2}.
We next solve a linear program related to Theorem~\ref{t intro linear program 2} in \textsection\ref{ss linear program lltlp}--\ref{ss solution to lltlp}, and in the remainder of Section~\ref{s A D graph whose generating function is not Schur positive}, we use it to construct the D~graph  $\G^*$.
Some facts about this graph are given in \textsection\ref{ss accompanying data files} along with a guide to the accompanying data files giving its full description.
Finally, in Section~\ref{s Concluding remarks} we discuss the difficulties that must be overcome to answer question \eqref{e question} and suggest more precise versions of this question.

\section{The Fomin-Greene theory of noncommutative Schur functions}
\label{s Noncommutative symmetric functions}
D$_0$~graphs are closely tied to the algebra  $\U/\Irkst$.
Here we introduce the theory of noncommutative Schur functions in the algebra  $\U/\Irkst$, expanding on the short account  given in \textsection\ref{ss noncomm sym intro}.
This will be useful later for proving several facts about D$_0$~graphs.
The companion paper \cite{BF} contains a more thorough study of this theory.

\subsection{Noncommutative elementary symmetric and Schur functions}
\label{ss Noncommutative elementary symmetric functions}
The \emph{noncommutative elementary symmetric functions} in subsets of the variables are given by
\[
e_d(S)=\sum_{\substack{i_1 > i_2 > \cdots > i_d \\ i_1,\dots,i_d \in S}}u_{i_1}u_{i_2}\cdots u_{i_d} \, \in \, \U,
\]
for any subset $S$ of $[N]$ and positive integer $d$; set $e_0(S)=1$ and $e_{d}(S) = 0$ for $d<0$.
We also maintain the notation $e_d(\mathbf{u}) = e_d([N])$ from \textsection\ref{ss noncomm sym intro} throughout the paper.

Let $\lambda$ be a partition, let $\lambda'$ be the conjugate partition of  $\lambda$, and let  $t$ be the  number of parts of  $\lambda'$ (which is equal to $\lambda_1$).
Recall from \textsection\ref{ss noncomm sym intro} that the \emph{noncommutative Schur function}  $\mathfrak{J}_\lambda(\mathbf{u})$ is given by the following noncommutative version of the Jacobi-Trudi formula:
\[ \mathfrak{J}_\lambda(\mathbf{u}) = \sum_{\pi\in \S_{t}}
\sgn(\pi) \, e_{\lambda'_1+\pi(1)-1}(\mathbf{u}) e_{\lambda'_2+\pi(2)-2}(\mathbf{u}) \cdots e_{\lambda'_{t}+\pi(t)-t}(\mathbf{u}),\]
considered as an element of~\,$\U$.

\subsection{From words to generating functions}
Let  $\e{w = w_1 \cdots w_n}$ be a word.
We write $\Des(\e{w}) := \{i \in [n-1] \mid \e{w_i > w_{i+1}}\}$ for the \emph{descent set} of $\e{w}$.
Let
\[Q_{\Des(\e{w})}(\mathbf{x}) = \sum_{\substack{1 \le i_1 \le \, \cdots \, \le i_n\\j \in \Des(\e{w}) \implies i_j < i_{j+1} }} x_{i_1}\cdots x_{i_n}\]
be Gessel's \emph{fundamental quasisymmetric function} \cite{GesselPPartition} in the commuting variables $x_1,x_2,\ldots$.
Note that  $Q_{\Des(\e{w})}(\mathbf{x})$ depends on $n$ even though this does not appear in the notation.

Define the linear map
\[\Delta: \U\to \ZZ[x_1, x_2, \dots] \ \text{ by }  \ \e{w} \mapsto Q_{\Des(\e{w})}(\mathbf{x}).\]
The \emph{generating function} of an element $f$ of~\,$\U$ is defined to be $\Delta(f)$. This  is compatible with the definition of the generating function of a D$_0$~graph from the introduction in that the generating function of a D$_0$~graph $\G$ is equal to the generating function of $\sum_{\e{w} \in \ver(\G)} \e{w}$.

\subsection{The Fomin-Greene setup}
\label{ss FG setup}
Let $\langle\cdot ,\cdot  \rangle$ be the symmetric bilinear form on $\U$ for which the monomials form an orthonormal basis.
Note that any element of~\,$\U/I$ has a well-defined pairing with any element of $I^\perp$ for any two-sided ideal  $I$ of~\,$\U$.
Here,
\[I^\perp := \{f \in \U \mid \langle z, f \rangle = 0 \text{ for all }z \in I\}\]
denotes the orthogonal complement of $I$.

We state our variant of Theorem 1.2  of \cite{FG} and results in Section 6 of \cite{LamRibbon}.

\begin{theorem} \label{t basics}
Let $I$ be a two-sided ideal of~\,$\U$.
If $e_k(\mathbf{u})e_\ell(\mathbf{u})=e_\ell(\mathbf{u})e_k(\mathbf{u})$ in \,$\U/I$ for all $k$, $\ell$, then
for any $f\in I^\perp$,
the quasisymmetric function  $\Delta(f)$ is symmetric and its Schur expansion is given by
\vspace{-3pt}
\[\Delta(f) = \sum_{\lambda} s_\lambda(\mathbf x) \langle \mathfrak J_{\lambda}(\mathbf{u}) , f \rangle, \]
\vspace{-3pt}
where the sum is over all partitions $\lambda$.
\end{theorem}
The idea of the proof is that whenever the $e_d(\mathbf{u})$ commute in some quotient of
$\U$, the subalgebra they generate is the surjective image of the ring of symmetric functions in commuting variables and hence all the usual identities hold, including the Jacobi-Trudi formula. See \cite{BF} for details.

The following lemma is similar to \cite[Lemma 3.1]{FG}.  In \cite{BF}, we will prove a result that generalizes both the lemma below and that in \cite{FG}. Recall that  $\U/\Irkst$ is the quotient of~\,$\U$ by the relations \eqref{e Irkst1}--\eqref{e Irkst2}.
\begin{lemma}\label{l es commute}
There holds $e_k(\mathbf{u})e_\ell(\mathbf{u})=e_\ell(\mathbf{u})e_k(\mathbf{u})$ in the algebra  $\U/\Irkst$, for all  $k$,  $\ell$.
\end{lemma}
\vspace{-11pt}
\begin{proof}
We closely follow the proof of \cite[Lemma 3.1]{FG}, with only a few key modifications.
For $i\leq j$, let
\vspace{-3pt}
\[
E_{j,i}(x)=(1+xu_j)(1+xu_{j-1})\cdots(1+xu_i)=\sum_{d=0}^{j-i+1}x^d e_d(\{i, i+1, \ldots, j\}),
\]
\vspace{-3pt}
where $x$ commutes with all of the $u_i$. Define $E_{j,i}(x)=1$ if $i=j+1$.

The statement of the lemma is equivalent to saying that $E_{N,1}(x)$ and $E_{N,1}(y)$ commute,
where $x$ and $y$ are scalar variables. We will prove that $E_{j,i}(x)$ and $E_{j,i}(y)$ commute by induction on
$j-i$.

{\allowdisplaybreaks
The cases $j-i=-1$ and $j-i=0$ are clear. Now assume $i<j$. Let $[u_i,u_j]$ denote the commutator $u_iu_j-u_ju_i$.
Using the induction hypothesis, we compute
\begin{align*}
E_{j,i}&(x)E_{j,i}(y)\\
&=E_{j, i+1}(x)(1+x u_i)(1+yu_j)E_{j-1,i}(y)\\
&=E_{j, i+1}(x)(1+yu_j)(1+x u_i)E_{j-1, i}(y)+xyE_{j, i+1}(x)[u_i, u_j]E_{j-1, i}(y)\\
&=E_{j, i+1}(x)E_{j, i+1}(y)(E_{j-1, i+1}(y))^{-1}(E_{j-1, i+1}(x))^{-1}E_{j-1, i}(x)E_{j-1, i}(y)\\
&\quad+xy E_{j, i+1}(x)[u_i, u_j]E_{j-1, i}(y)\\
&= E_{j, i+1}(y)E_{j, i+1}(x)(E_{j-1, i+1}(x))^{-1}(E_{j-1, i+1}(y))^{-1}E_{j-1, i}(y)E_{j-1, i}(x)\\
&\quad+xy E_{j, i+1}(x)[u_i, u_j]E_{j-1, i}(y)\\
&=E_{j, i+1}(y)(1+xu_j)(1+yu_i)E_{j-1, i}(x)+xy E_{j, i+1}(x)[u_i, u_j]E_{j-1, i}(y)\\
&=E_{j, i+1}(y)(1+yu_i)(1+xu_j)E_{j-1, i}(x)\\
&\quad + xy E_{j, i+1}(y)[u_j, u_i]E_{j-1, i}(x)
+xy E_{j, i+1}(x)[u_i, u_j]E_{j-1, i}(y)\\
&=E_{j, i}(y)E_{j, i}(x) + xy E_{j-1, i+1}(y)[u_j, u_i]E_{j-1, i+1}(x)
+xy E_{j-1, i+1}(x)[u_i, u_j]E_{j-1, i+1}(y)\\
&=E_{j, i}(y)E_{j, i}(x) + xy [u_j, u_i] E_{j-1, i+1}(y)E_{j-1, i+1}(x)
+xy [u_i, u_j] E_{j-1, i+1}(x)E_{j-1, i+1}(y)\\
&=E_{j, i}(y)E_{j, i}(x).
\end{align*}
The third to last equality is by \eqref{e Irkst2} and the second to last equality is by \eqref{e Irkst1}.
}
\end{proof}
\pagebreak[3]

In \textsection\ref{ss D0 graphs}, we will discuss consequences of Theorem~\ref{t basics} and Lemma \ref{l es commute} for D$_0$~graphs.
The heart of the Fomin-Greene approach is to express  $\mathfrak{J}_\lambda(\mathbf{u})$ as a positive sum of monomials in $\U/I$ for  $I$ ``as small as possible''
because this yields (via Theorem~\ref{t basics}) a positive combinatorial formula for $\Delta(f)$ for any positive sum of monomials  $f$ lying in  $I^\perp$.

\subsection{Bijectivizations}
\label{ss bijectivizations}
For the ideal  $I = \Irkst$, the condition  $f \in I^\perp$ from Theorem~\ref{t basics} is unintuitive, and the problem of expressing $\mathfrak{J}_\lambda(\mathbf{u})$ as a positive sum of monomials in  $\U/I$ is computationally difficult.
It is therefore fruitful to first consider the easier condition and problem for
bijectivizations of~\,$\U/\Irkst$.

\begin{definition}
A \emph{bijectivization} of an algebra $\U/I$ is an algebra $\U/J$ such that $I \subseteq J$ and
$J$ is generated by monomials (i.e. words in  $\U$) and binomials of the form $\e{v}-\e{w}$, where  $\e{v}$ and $\e{w}$ are words of the same length.
To any bijectivization  $\U/J$, we associate the following equivalence relation:
two words $\e{v}$ and $\e{w}$ of~\,$\U$ are \emph{equivalent} if $\e{v}-\e{w} \in J$.
We also say that a word is \emph{nonzero} if it does not belong to $J$ and that an equivalence class of words is \emph{nonzero} if all of its elements are nonzero.
\end{definition}

If $J$ is as in the previous definition, then the space  $J^\perp$ has $\ZZ$-basis  $\{\sum_{\e{w} \in C}\e{w}\}$, where  $C$ ranges over the nonzero equivalence classes of~\,$\U/J$.
Moreover,  the problem of expressing $\mathfrak{J}_\lambda(\mathbf{u})$ as a positive sum of monomials in  $\U/J$ can be made explicit as follows:
let  $\e{w^1, \dots, w^m}$ denote the words of~\,$\U$ of length  $n$.  For each partition  $\lambda$ of  $n$, we can write
$\mathfrak{J}_\lambda(\mathbf{u}) = \sum_i c_i \e{w^i}, \ c_i \in \ZZ$ (an equality in  $\U$).
Then $\mathfrak{J}_\lambda(\mathbf{u})$ is a positive sum of monomials in  $\U/J$ if and only if $\sum_{\e{w^i} \in C}c_i \ge 0$ for each nonzero equivalence class $C$ of~\,$\U/J$.

\begin{example}
\label{ex plac}
Let $\U/\Iplacst$ be the quotient of the plactic algebra by the additional relation \eqref{e Irkst2}, i.e. the quotient of~\,$\U$ by \eqref{e Irkst2} and the Knuth relations
\begin{align*}
&\e{acb} = \e{cab},\\
&\e{bac} = \e{bca},
\end{align*}
for letters $a<b<c$. Then $\U/\Iplacst$ is a bijectivization of~\,$\U/\Irkst$.
The nonzero equivalence classes are the Knuth equivalence classes consisting of words with no repeated letter.

Another bijectivization of~\,$\U/\Irkst$ is the quotient of~\,$\U$ by \eqref{e Irkst2} and the relations
\begin{align*}
&\e{acb} = \e{bac},\\
&\e{cab} = \e{bca},
\end{align*}
corresponding to rotation edges.
This algebra was studied by Novelli and Schilling in \cite{NS}.
It follows easily from their results that two words  $\e{v}, \e{w} \in \W_n$
are equivalent if and only if they have the same number of inversions, they are permutations of the same set of letters, and
 $\min(\e{v})$ and  $\max(\e{v})$ occur in the same order in both words.
\end{example}

\begin{example}
\label{ex triple biject}
Here is a rich source of examples of bijectivizations of~\,$\U/\Irkst$.
For each subset  $S \subseteq [N]$ of size 3, declare it to be either a \emph{Knuth triple} or a  \emph{rotation triple}.
The  \emph{triples bijectivization} of~\,$\U/\Irkst$ associated to this data is the quotient of~\,$\U$ by the relations
\[\begin{array}{ll}
\e{acb} = \e{cab}& \text{ if $a < b < c$ is a Knuth triple}  \\
\e{bac} = \e{bca}& \text{ if $a < b < c$ is a Knuth triple} \\
\e{acb} = \e{bac}& \text{ if $a < b < c$ is a rotation triple}  \\
\e{cab} = \e{bca}& \text{ if $a < b < c$ is a rotation triple} \\
\e{w} = 0 & \text{ for words $\e{w}$ with a repeated letter.}
\end{array}\]
\end{example}

\begin{remark}
Extensive computer experimentation supports the following statement: 
the symmetric function $\Delta(\sum_{\e{w}\in C}\e{w})$ is Schur positive
for any nonzero equivalence class  $C$ of any triples bijectivization.
However, no proof of this is known.
See Section~\ref{s Concluding remarks} for further discussion.
\end{remark}

In Section~\ref{s LLT}, we define certain bijectivizations of~\,$\U/\Irkst$ to apply the Fomin-Greene approach to LLT polynomials and certain graphs related to LLT polynomials from \cite{Sami}.

\subsection{Automorphisms of~\,$\U/\Irkst$}
\label{ss automorphisms}
Certain natural maps on  $\U$ induce well-defined maps on the quotient $\U/\Irkst$.
Similar results are proven in \cite{LS} for the plactic monoid.
For a word $\e{w = w_1 w_2 \cdots w_n}$, let $\rev(\e{w})$ denote the word $\e{w_n w_{n-1} \cdots w_1}$.
\begin{proposition}\label{p order preserving Irk only}
\
\begin{list}{\emph{(\roman{ctr})}}{\usecounter{ctr} \setlength{\itemsep}{2pt} \setlength{\topsep}{3pt}}
\item Let $\theta: [N] \to \ZZ$ be an order-preserving injection and let $\theta_\U: \U \to \U$, $u_i \mapsto u_{\theta(i)}$ be the corresponding algebra homomorphism, where we set  $u_i = 0$ if  $i \notin [N]$.
Then $\theta_\U(\Irkst) \subseteq \Irkst$ and hence  $\theta_\U$ induces an algebra homomorphism $\U/\Irkst \to \U/\Irkst$.
\item The analog of statement (i) holds for $\theta$ an order-reversing injection.
\item Let $\rev : \U \to \U$ denote the algebra anti-automorphism given by $\e{w} \mapsto \rev(\e{w})$.
Then $\rev(\Irkst) = \Irkst$ and hence  $\rev$ induces an algebra anti-automorphism $\U/\Irkst \to \U/\Irkst$.
\end{list}
\end{proposition}
\begin{proof}
Statement (i) is clear since for an order-preserving injection $\theta$ and any  $1 \le a < b < c \le N$, there holds
\begin{align}
\theta_\U\big(\e{b(ac-ca)} - \e{(ac-ca)b}\big) = \e{b'(a'c'-c'a')} - \e{(a'c'-c'a')b'}  \in \Irkst,  \label{e order pr}
\end{align}
for some $a' < b' < c'$.  Note that there is no problem if $a', b', c'$ do not all lie in $[N]$ since if, say,
$c' \notin [N]$  then  $\e{c'} := u_{c'} = 0$, so the right side of \eqref{e order pr} is 0.

If  $\theta$ is an order-reversing injection, then for any $1 \le a < b < c \le N$,
\[\theta_\U\big(\e{b(ac-ca)} - \e{(ac-ca)b}\big) = \e{b'(c'a'-a'c')} - \e{(c'a'-a'c')b'} = -\big(\e{b'(a'c'-c'a')} - \e{(a'c'-c'a')b'}\big) \in \Irkst\]
for some $a' < b' < c'$.
To verify (iii), we compute
\[\rev\big(\e{b(ac-ca)} - \e{(ac-ca)b}\big) = \e{(ca-ac)b} - \e{b(ca-ac)} = \e{b(ac-ca)} - \e{(ac-ca)b} \in \Irkst. \qedhere\]
\end{proof}

\section{D$_0$~graphs and D~graphs}
\label{s D0 graphs and D graphs}
Here we define  D$_0$~graphs differently than in the introduction and relate them to the algebra
 $\U/\Irkst$.
We then recall the axioms defining a D~graph from \cite{Sami}.

\subsection{D$_0$~graphs}
\label{ss D0 graphs}
Let  $\W\subseteq \U$ denote the set of words with no repeated letter, let $\W_n \subseteq \W$ denote the set of words of length  $n$ with no repeated letter, and
let $\S_n$ denote the set of permutations of  $n$. We write  $|\e{v}|$ for the length of a word  $\e{v}$.

A \emph{KR square} is a 4-tuple of elements of  $\W$ of the form
\begin{align}
\e{v\stars w} =  \big(\e{v \, bac \, w}, \e{v \, bca \, w}, \e{v \, acb \, w}, \e{v \, cab \, w}\big),\label{e KR4tuple}
\end{align}
where  $a < b < c$ are letters and $\e{v}$ and $\e{w}$ are words. We also say \eqref{e KR4tuple} is a KR$_i$ square, for $i=|\e{v}|+2$.
We abuse notation and identify the tuple  $\e{v} \stars \e{w}$ with its set of elements when convenient.
We say that  a KR square $X$ is a \emph{KR square of  $\W_n$} or a \emph{KR square of  $\S_n$} if $X \subseteq \W_n$ or  $X \subseteq \S_n$, respectively.

\begin{definition}\label{d D0 graph}
A \emph{partial D$_0$~graph of degree  $n$} is a graph $\G$ on a vertex set $W \subseteq \W_n$
with edges colored from the set  $\{2,3,\dots,n-1\}$ such that
for each KR$_i$ square $X = \e{v} \stars \e{w}$,
exactly one of the following possibilities occurs:
\begin{list}{{(\roman{ctr})}}{\usecounter{ctr} \setlength{\itemsep}{2pt} \setlength{\topsep}{3pt}}
\item \ \ $X \cap W = \varnothing,$
\item \ $X \cap W$ is equal to $\{\e{v \, bac \, w}, \e{v \, bca \, w}\}$ or  $\{\e{v \, acb \, w}, \e{v \, cab \, w}\}$ and is an  $i$-edge,
\item $X \cap W$ is equal to $\{\e{v \, bac \, w}, \e{v \, acb \, w}\}$ or  $\{\e{v \, bca \, w}, \e{v \, cab \, w}\}$ and is an  $i$-edge,
\item $X \subseteq W$ and $\{\e{v \, bac \, w}, \e{v \, bca \, w}\}$ and $\{\e{v \, acb \, w}, \e{v \, cab \, w}\}$ are  $i$-edges,
\item  \ $X \subseteq W$ and $\{\e{v \, bac \, w}, \e{v \, acb \, w}\}$ and $\{\e{v \, bca \, w}, \e{v \, cab \, w}\}$ are  $i$-edges,
\item $X \subseteq W$ and no edge has both ends in $X$.
\end{list}
If (vi) never occurs, then we say that  $\G$ is a \emph{D$_0$~graph of degree  $n$}.

Recall from Definition \ref{d intro D0 graph} that an edge $\{\e{x,y}\}$ of  $\G$ is a \emph{Knuth edge} (resp. \emph{rotation edge})
if $\{\e{x,y}\}$ is as in (ii) and (iv) (resp. (iii) and (v)); we call this the \emph{type} of the edge so that each edge has two kinds of labels---color and type;
we can then speak of the type of the $i$-edge at  $\e{x}$, the color of a Knuth edge at  $\e{x}$, the rotation $j$-edges of  $\G$, etc.;
a Knuth $i$-edge (resp. rotation  $i$-edge) is denoted  $i$ (resp.  $\tilde{i}$) in figures.
Note that the type and color of an edge  $\{\e{x,y}\}$
can be recovered from the vertex labels $\e{x,y}$, with one exception: if $\e{x = v ac w, \ y = v ca w}$, then $\{\e{x,y}\}$ has type Knuth and color either  $|\e{v}|+1$ or  $|\e{v}|+2$.
\end{definition}
In agreement with standard terminology (see, e.g. \cite[A1]{St}), a \emph{Knuth transformation} is a transformation of the form  $\e{x} \tto \e{y}$ whenever  $\{\e{x,y}\}$ is as in (ii) and (iv) above.  We define a \emph{rotation transformation} to be a transformation of the form  $\e{x} \tto \e{y}$ whenever  $\{\e{x,y}\}$ is as in (iii) and (v) above.
These transformations are related to the  $d_i$ and  $\widetilde{d_i}$ of \cite{Sami} as follows:
for permutations  $\e{x}$ and $\e{y}$, the transformation
$\e{x} \tto \e{y}$ is a Knuth (resp. rotation) transformation at positions  $i-1, i, i+1$ if and only if $d_i(\e{x}^{-1}) = \e{y}^{-1}$  (resp. $\widetilde{d_i}(\e{x}^{-1}) = \e{y}^{-1}$).

Examples of  D$_0$~graphs are given in Figures \ref{f intro}, \ref{f D graph k3}, and \ref{f permutation D graph 32+311}.

Given a partial D$_0$~graph $G$ with vertex set  $W$, a \emph{KR square of  $G$} is just a KR square, but has the following extra data associated to it.
The \emph{type} of  a KR square $X = \e{v} \stars \e{w}$ of  $G$ is
\[
\left\{
\begin{array}{lll}
\text{\emph{Irrelevant}} & (\varnothing) &\text{ if $X \cap W = \varnothing$},\\
\text{\emph{Undetermined}} & (0) &\text{ if $X \cap W = X$ and no edge of  $G$ has both ends in $X$},\\
\text{\emph{Knuth}} & (\text{K})&\text{ if }\{\e{v \, bac \, w}, \e{v \, bca \, w}\}$ or $\{\e{v \, acb \, w}, \e{v \, cab \, w}\}\text{ (or both) are edges of }G,\\
\text{\emph{Rotation}} & (\text{R})&\text{ if }\{\e{v \, bac \, w}, \e{v \, acb \, w}\}$ or $\{\e{v \, bca \, w}, \e{v \, cab \, w}\}\text{ (or both) are edges of }G.
\end{array}
\right.
\]

\begin{figure}
 \begin{tikzpicture}[xscale = 2.2,yscale = 3.1]
\tikzstyle{vertex}=[inner sep=0pt, outer sep=4pt]
\tikzstyle{aedge} = [draw, thin, ->,black]
\tikzstyle{edge} = [draw, thick, -,black]
\tikzstyle{dashededge} = [draw, very thick, dashed, black]
\tikzstyle{LabelStyleH} = [text=black, anchor=south]
\tikzstyle{LabelStyleV} = [text=black, anchor=east]

 \node[vertex] (v1) at (1,1){\footnotesize$\e{1234 }$};

 \end{tikzpicture}
         \vspace{.3in}

 \begin{tikzpicture}[xscale = 2.2,yscale = 3.1]
\tikzstyle{vertex}=[inner sep=0pt, outer sep=4pt]
\tikzstyle{aedge} = [draw, thin, ->,black]
\tikzstyle{edge} = [draw, thick, -,black]
\tikzstyle{dashededge} = [draw, very thick, dashed, black]
\tikzstyle{LabelStyleH} = [text=black, anchor=south]
\tikzstyle{LabelStyleV} = [text=black, anchor=east]

 \node[vertex] (v1) at (3,1){\footnotesize$\e{1243 }$};
 \node[vertex] (v2) at (2,1){\footnotesize$\e{1324 }$};
 \node[vertex] (v3) at (1,1){\footnotesize$\e{2134 }$};
 \draw[edge] (v1) to node[LabelStyleH]{\Tiny$\tilde{3} $} (v2);
 \draw[edge] (v2) to node[LabelStyleH]{\Tiny$\tilde{2} $} (v3);

 \end{tikzpicture}
         \vspace{.3in}

 \begin{tikzpicture}[xscale = 2.2,yscale = 3.1]
\tikzstyle{vertex}=[inner sep=0pt, outer sep=4pt]
\tikzstyle{aedge} = [draw, thin, ->,black]
\tikzstyle{edge} = [draw, thick, -,black]
\tikzstyle{dashededge} = [draw, very thick, dashed, black]
\tikzstyle{LabelStyleH} = [text=black, anchor=south]
\tikzstyle{LabelStyleV} = [text=black, anchor=east]

 \node[vertex] (v1) at (5,1){\footnotesize$\e{1342 }$};
 \node[vertex] (v2) at (4,1){\footnotesize$\e{1423 }$};
 \node[vertex] (v3) at (3,1){\footnotesize$\e{2143 }$};
 \node[vertex] (v4) at (2,1){\footnotesize$\e{2314 }$};
 \node[vertex] (v5) at (1,1){\footnotesize$\e{3124 }$};
 \draw[edge] (v1) to node[LabelStyleH]{\Tiny$\tilde{3} $} (v2);
 \draw[edge] (v2) to node[LabelStyleH]{\Tiny$\tilde{2} $} (v3);
 \draw[edge] (v3) to node[LabelStyleH]{\Tiny$\tilde{3} $} (v4);
 \draw[edge] (v4) to node[LabelStyleH]{\Tiny$\tilde{2} $} (v5);

 \end{tikzpicture}
         \vspace{.3in}

 \begin{tikzpicture}[xscale = 2.2,yscale = 3.1]
\tikzstyle{vertex}=[inner sep=0pt, outer sep=4pt]
\tikzstyle{aedge} = [draw, thin, ->,black]
\tikzstyle{edge} = [draw, thick, -,black]
\tikzstyle{dashededge} = [draw, very thick, dashed, black]
\tikzstyle{LabelStyleH} = [text=black, anchor=south]
\tikzstyle{LabelStyleV} = [text=black, anchor=east]

 \node[vertex] (v1) at (1,1){\footnotesize$\e{1432 }$};
 \node[vertex] (v2) at (2,1){\footnotesize$\e{3142 }$};
 \node[vertex] (v3) at (3,1){\footnotesize$\e{3214 }$};
 \draw[edge] (v1) to node[LabelStyleH]{\Tiny$\tilde{2} $} (v2);
 \draw[edge] (v2) to node[LabelStyleH]{\Tiny$\tilde{3} $} (v3);

 \end{tikzpicture}
         \hspace{.3in}
 \begin{tikzpicture}[xscale = 2.2,yscale = 3.1]
\tikzstyle{vertex}=[inner sep=0pt, outer sep=4pt]
\tikzstyle{aedge} = [draw, thin, ->,black]
\tikzstyle{edge} = [draw, thick, -,black]
\tikzstyle{dashededge} = [draw, very thick, dashed, black]
\tikzstyle{LabelStyleH} = [text=black, anchor=south]
\tikzstyle{LabelStyleV} = [text=black, anchor=east]

 \node[vertex] (v1) at (3,1){\footnotesize$\e{2341 }$};
 \node[vertex] (v2) at (2,1){\footnotesize$\e{2413 }$};
 \node[vertex] (v3) at (1,1){\footnotesize$\e{4123 }$};
 \draw[edge] (v1) to node[LabelStyleH]{\Tiny$\tilde{3} $} (v2);
 \draw[edge] (v2) to node[LabelStyleH]{\Tiny$\tilde{2} $} (v3);
 \end{tikzpicture}
         \vspace{.3in}

 \begin{tikzpicture}[xscale = 2.2,yscale = 3.1]
\tikzstyle{vertex}=[inner sep=0pt, outer sep=4pt]
\tikzstyle{aedge} = [draw, thin, ->,black]
\tikzstyle{edge} = [draw, thick, -,black]
\tikzstyle{dashededge} = [draw, very thick, dashed, black]
\tikzstyle{LabelStyleH} = [text=black, anchor=south]
\tikzstyle{LabelStyleV} = [text=black, anchor=east]

 \node[vertex] (v1) at (1,1){\footnotesize$\e{2431 }$};
 \node[vertex] (v2) at (2,1){\footnotesize$\e{3241 }$};
 \node[vertex] (v3) at (3,1){\footnotesize$\e{3412 }$};
 \node[vertex] (v4) at (4,1){\footnotesize$\e{4132 }$};
 \node[vertex] (v5) at (5,1){\footnotesize$\e{4213 }$};
 \draw[edge] (v1) to node[LabelStyleH]{\Tiny$\tilde{2} $} (v2);
 \draw[edge] (v2) to node[LabelStyleH]{\Tiny$\tilde{3} $} (v3);
 \draw[edge] (v3) to node[LabelStyleH]{\Tiny$\tilde{2} $} (v4);
 \draw[edge] (v4) to node[LabelStyleH]{\Tiny$\tilde{3} $} (v5);

 \end{tikzpicture}
         \vspace{.3in}

 \begin{tikzpicture}[xscale = 2.2,yscale = 3.1]
\tikzstyle{vertex}=[inner sep=0pt, outer sep=4pt]
\tikzstyle{aedge} = [draw, thin, ->,black]
\tikzstyle{edge} = [draw, thick, -,black]
\tikzstyle{dashededge} = [draw, very thick, dashed, black]
\tikzstyle{LabelStyleH} = [text=black, anchor=south]
\tikzstyle{LabelStyleV} = [text=black, anchor=east]

 \node[vertex] (v1) at (1,1){\footnotesize$\e{3421 }$};
 \node[vertex] (v2) at (2,1){\footnotesize$\e{4231 }$};
 \node[vertex] (v3) at (3,1){\footnotesize$\e{4312 }$};
 \draw[edge] (v1) to node[LabelStyleH]{\Tiny$\tilde{2} $} (v2);
 \draw[edge] (v2) to node[LabelStyleH]{\Tiny$\tilde{3} $} (v3);

 \end{tikzpicture}
         \vspace{.3in}

 \begin{tikzpicture}[xscale = 2.2,yscale = 3.1]
\tikzstyle{vertex}=[inner sep=0pt, outer sep=4pt]
\tikzstyle{aedge} = [draw, thin, ->,black]
\tikzstyle{edge} = [draw, thick, -,black]
\tikzstyle{dashededge} = [draw, very thick, dashed, black]
\tikzstyle{LabelStyleH} = [text=black, anchor=south]
\tikzstyle{LabelStyleV} = [text=black, anchor=east]

 \node[vertex] (v1) at (1,1){\footnotesize$\e{4321 }$};

 \end{tikzpicture}
\caption{\label{f D graph k3} The  D$_0$~graph on  $\S_4$ with only rotation edges.  In the notation of \textsection\ref{ss Assaf's LLT bijectivizations}, this is the Assaf LLT$_3$ D~graph $\G_3^\stand$, restricted to  $\S_4$.}
\end{figure}

\subsection{KR sets}
\label{ss KR sets}
The vertex sets of D$_0$~graphs are the same as zero-one vectors of  $(\Irkst)^\perp$, as we now show.
In light of Theorem~\ref{t basics}, these sets are important objects in their own right and are in some ways more fundamental than D$_0$~graphs.

\begin{propdef}\label{pd KR set}
For a subset $W$ of $\W_n$, the following are equivalent:
\begin{list}{\emph{(\roman{ctr})}}{\usecounter{ctr} \setlength{\itemsep}{2pt} \setlength{\topsep}{3pt}}
\item $W$ is the vertex set of a  D$_0$~graph.
\item for every KR square $X$ of $\W_n$, the set $W \cap X$ is empty, is all of $X$, or consists of two words which differ by a Knuth transformation or rotation transformation.
\item $\sum_{\e{w} \in W} \e{w} \in (\Irkst)^\perp$.
\end{list}
If  $W$ satisfies these conditions, then we say that  $W$ is a \emph{KR set}.
\end{propdef}
\begin{proof}
By comparing the definition of a D$_0$~graph to (ii), it is clear that (i) implies (ii).
To see that (ii) implies (i), note that the KR$_i$ squares of  $\W_n$ are pairwise disjoint. Hence if  $W$ is as in (ii) and  $t_i$ is the number of KR$_i$ squares  $X$ such that  $X \subseteq W$, then there are  $2^{\sum_i t_i}$ D$_0$~graphs with vertex set   $W$ corresponding to the $\sum_i t_i$ independent choices for the types of these
KR squares.

We now prove the equivalence of (ii) and (iii).
Let $X = \e{v} \stars \e{w}$ be a KR square and let  $\ZZ X$ be the  $\ZZ$-module  with  $\ZZ$-basis
\begin{align*}
\e{v\stars w} =  \big(\e{v \, bac \, w}, \e{v \, bca \, w}, \e{v \, acb \, w}, \e{v \, cab \, w}\big).\label{e KR4tuple}
\end{align*}
Consider the vector
\[h_X :=  \e{v \, bac \, w} -\e{v \, bca \, w} -\e{v \, acb \, w} +\e{v \, cab \, w} = \matrow{1 & -1 & -1 & 1} \in \ZZ X.\]
The zero-one vectors orthogonal to  $h_X$ are
\begin{equation}
\begin{array}{l}\label{e 01 perp}
\matrow{0 & 0 & 0 & 0}, \
\matrow{1 & 1 & 1 & 1}, \
\matrow{1 & 1 & 0 & 0}, \
\matrow{0 & 0 & 1 & 1}, \
\matrow{1 & 0 & 1 & 0}, \
\matrow{0 & 1 & 0 & 1}.
\end{array}
\end{equation}

Therefore, since the  $\ZZ$-module $\Irkst \cap \ZZ \W_n$ is spanned by the elements  $h_X$ over all
KR squares  $X$ of  $\W_n$,
an element $\sum_{\e{w} \in W} \e{w} \in \ZZ \W_n$ belongs to $(\Irkst)^\perp$ if and only if  $\sum_{\e{w}\in W \cap X}\e{w}$ is one of the six vectors in \eqref{e 01 perp} for all KR squares  $X$.
The equivalence of (ii) and (iii) then follows because
those pairs of words of  $X$ which differ by a Knuth transformation or rotation transformation are exactly the supports of the last four vectors of \eqref{e 01 perp}.
\end{proof}

We define (a slight generalization of Definition \ref{d gen func}) the \emph{generating function} of a partial D$_0$~graph $H$ to be
\[\Delta\Big(\sum_{\e{w} \in \ver(H)}\e{w} \Big)  = \sum_{\e{w} \in \ver(H)}Q_{\Des(\e{w})}(\mathbf{x}).\]
Note that this only depends on the KR set  $\ver(H)$ and not on the edges of $H$.
Theorem~\ref{t basics} and Lemma \ref{l es commute} then have the following corollary.
\begin{corollary}\label{c symmetric bijectivizations}
The generating function of any (partial) D$_0$~graph $H$ is a symmetric function.
The coefficient of $s_\lambda(\mathbf{x})$ in this function is
$\langle \mathfrak{J}_{\lambda}(\mathbf{u}) , \sum_{\e{w} \in \ver(H)} \e{w} \rangle$.
\end{corollary}

\subsection{Preliminary definitions for D~graphs}
\label{ss Preliminary definitions for D graphs}
Before recalling the axioms of a D~graph from \cite{Sami}, we must first recall several definitions from that paper.

\begin{definition}[{\cite[Definition 3.3]{Sami}}]
A \emph{signed, colored graph\footnote{These are slightly less general than the signed, colored graphs considered in \cite{Sami}.} of degree  $n$} is an (undirected) graph  $\G$ with vertex set  $V$ whose edges are colored from the set  $\{2,3,\ldots, n-1\}$, together with a \emph{signature function}  $\sigma : V \to \{+1,-1\}^{n-1}$.

A signed, colored graph satisfies \emph{axiom 0} if the  $i$-degree of every vertex is at most 1.
All signed, colored graphs considered in this paper satisfy axiom 0.
Let  $\G$ be a signed, colored graph satisfying axiom 0.
The notation $E_i(v) = w$ means that  $\{v,w\}$ is an  $i$-edge of  $\G$.
A vertex  $v$ of  $\G$ \emph{admits an  $i$-neighbor} if there is an $i$-edge $\{v,w\}$ for some  $w$.
\end{definition}

\begin{remark}\label{r involution}
It is sometimes convenient to think of  $E_i$  in the definition above as an involution of  $V$, with  fixed points those vertices not  admitting an  $i$-neighbor.
However, to avoid confusion,  we consider the expression $E_i(v)$  to be undefined  if $v$  does not admit an  $i$-neighbor.
\end{remark}

Any (partial) D$_0$~graph  $\G$ of degree  $n$ may also be regarded as a signed, colored graph whose edge colors are the same as those of  $\G$ (but edge types are ignored) and
whose signature function on each word $\e{w} \in \ver(\G)$ is given by
\[\sigma(\e{w})_i =
\begin{cases}
+1 & \text{if }\e{w_i \le w_{i+1}},\\
-1 & \text{if }\e{w_i > w_{i+1}},
\end{cases}
\quad \text{for } i \in [n-1].
\]
For example, $\sigma(\e{72158346}) = - - + + - + +$.  The signature $\sigma(\e{w})$ of  $\e{w}$ is just another encoding of the descent set  $\Des(\e{w})$, the two being related by
 $\sigma(\e{w})_i = -$ if and only if  $i \in \Des(\e{w})$.

%

\begin{definition}
Given a D$_0$~graph $\G$ with vertex set $W \subseteq \W_n$ and an interval $K = \{i,i+1,\ldots, j\}$,  $1 \le i \le j \le n$, the \emph{$K$-restriction of $\G$}, denoted $\myRes_K \G$, is the graph with
\begin{list}{{(\roman{ctr})}}{\usecounter{ctr} \setlength{\itemsep}{2pt} \setlength{\topsep}{3pt}}
\item vertex set  $W$, with each vertex $\e{w} \in W$ labeled by the word $\e{w_iw_{i+1}\cdots w_j}$,
\item signature function $\hat{\sigma}$ defined by $\hat{\sigma}_s(\e{w}) = \sigma_{s+i-1}(\e{w}),$ for all $s \in [j-i]$, $\e{w} \in W$,
\item an  $(s-i+1)$-edge for each  $s$-edge of  $\G$ (with the same ends as in  $\G$), for all  $s \in \{i+1,\ldots, j-1\}$.
\end{list}

For a signed, colored graph $\G$, we define the \emph{$K$-restriction of $\G$}, $\myRes_K \G$, in the same way,
the only difference being that in place of (i) we simply define $\ver(\myRes_K \G)$ to be $\ver(\G)$.

If  $\G$ is a D$_0$~graph or a signed, colored graph and $w$ is a vertex of  $\G$,
we define $\myRes_K(\G, w)$ to be the component of $\myRes_K \G$ containing $w$.
\end{definition}
For example, here is a  D$_0$~graph (top line) and its $\{2,3,4\}$-restriction (bottom line):

\begin{center}
\begin{tikzpicture}[xscale = 2,yscale = 1]
\tikzstyle{vertex}=[inner sep=0pt, outer sep=4pt]
\tikzstyle{framedvertex}=[inner sep=3pt, outer sep=4pt, draw=gray]
\tikzstyle{aedge} = [draw, thin, ->,black]
\tikzstyle{edge} = [draw, thick, -,black]
\tikzstyle{doubleedge} = [draw, thick, double distance=1pt, -,black]
\tikzstyle{hiddenedge} = [draw=none, thick, double distance=1pt, -,black]
\tikzstyle{dashededge} = [draw, very thick, dashed, black]
\tikzstyle{LabelStyleH} = [text=black, anchor=south]
\tikzstyle{LabelStyleHn} = [text=black, anchor=north]
\tikzstyle{LabelStyleV} = [text=black, anchor=east]

\node[vertex] (v1) at (1,1){\footnotesize$\e{23451}$};
\node[vertex] (v2) at (2,1){\footnotesize$\e{23415}$};
\node[vertex] (v3) at (3,1){\footnotesize$\e{24135}$};
\node[vertex] (v4) at (4,1){\footnotesize$\e{41235}$};

\draw[edge] (v1) to node[LabelStyleH]{\Tiny$4 $} (v2);
\draw[edge] (v2) to node[LabelStyleH]{\Tiny$\tilde{3} $} (v3);
\draw[edge] (v3) to node[LabelStyleH]{\Tiny$\tilde{2} $} (v4);
\end{tikzpicture}

\begin{tikzpicture}[xscale = 2,yscale = 1]
\tikzstyle{vertex}=[inner sep=0pt, outer sep=4pt]
\tikzstyle{framedvertex}=[inner sep=3pt, outer sep=4pt, draw=gray]
\tikzstyle{aedge} = [draw, thin, ->,black]
\tikzstyle{edge} = [draw, thick, -,black]
\tikzstyle{doubleedge} = [draw, thick, double distance=1pt, -,black]
\tikzstyle{hiddenedge} = [draw=none, thick, double distance=1pt, -,black]
\tikzstyle{dashededge} = [draw, very thick, dashed, black]
\tikzstyle{LabelStyleH} = [text=black, anchor=south]
\tikzstyle{LabelStyleHn} = [text=black, anchor=north]
\tikzstyle{LabelStyleV} = [text=black, anchor=east]

\node[vertex] (v1) at (1,1){\footnotesize$\e{345}$};
\node[vertex] (v2) at (2,1){\footnotesize$\e{341}$};
\node[vertex] (v3) at (3,1){\footnotesize$\e{413}$};
\node[vertex] (v4) at (4,1){\footnotesize$\e{123}$};

\draw[edge] (v2) to node[LabelStyleH]{\Tiny$\tilde{2} $} (v3);
\end{tikzpicture}
\end{center}

\begin{remark}\label{r restriction}
The reason for the somewhat awkward definition of the vertex set of
$\myRes_K \G$ in (i) is because two distinct vertices $\e{w}, \e{w'}\in W$
can have the same label  $\e{w_iw_{i+1}\cdots w_j} = \e{w'_iw'_{i+1}\cdots w'_j}$.
This also means that $\myRes_K \G$ is
not a D$_0$~graph  in general.
But this is harmless because  every connected component of $\myRes_K \G$ is a D$_0$~graph.
(To view a component  $H$ of $\myRes_K \G$ as a D$_0$~graph, we should forget the original set $W$
and only retain the vertex labels $\{\e{w_iw_{i+1}\cdots w_j} \mid \e{w} \in \ver(H)\}$ of  $H$.)
\end{remark}

\begin{definition}
Let $\G$ be a signed, colored graph satisfying axiom 0. Any component  $H$ of $\myRes_{[i-2,i+1]} \G$ is a path or an even cycle.
For $2 < i < n$, an \emph{$(i-1)$-$i$-string for  $H$} is a sequence $(x^1, x^2,\ldots, x^{t})$ consisting of all the vertices of  $H$ (without repetition)
such that the pairs $\{x^1, x^{2}\}$, $\{x^2, x^{3}\}$, $\dots$ alternate between $i$-edges and $i-1$-edges (starting with either an  $i$-edge or an $i-1$-edge).
We also say that $(x^1, x^2,\ldots, x^{t})$ is an \emph{$(i-1)$-$i$-string} if there exists an  $H$ such that it is an $(i-1)$-$i$-string for  $H$.
\end{definition}
%

\begin{definition}[{\cite[Definition 5.3]{Sami}}]
A vertex $v$ of a signed, colored graph $\G$ satisfying axiom 0 \emph{has $i$-type W} if  $v$ admits an  $i$-neighbor and an  $i-1$-neighbor and
$\sigma(v)_i = -\sigma(E_{i-1}(v))_i$.
\end{definition}
For example, the vertices in Figure \ref{f D graph k3} that have 3-type W are $\e{2143, 1423, 3412, 4132}$.

\begin{definition}[{\cite[Definition 5.7]{Sami}}]
\label{d flat chain}
Let $\G$ be a signed, colored graph of degree  $n$ satisfying axiom 0.  For
$4 \le i < n$, a \emph{flat $i$-chain} of  $\G$ is a sequence $(x^1, x^2,\ldots, x^{2h-1}, x^{2h})$ of distinct vertices admitting $i-2$-neighbors
such that
\[x^{2j-1} = E_i(x^{2j})\, \text{ and }\, x^{2j+1} = E_{i-2}(E_{i-1}E_{i-2})^{m_j}(x^{2j}),\]
for nonnegative integers $m_j$ such that $(E_{i-1}E_{i-2})^{m_j}(x^{2j})$ does not have $(i-1)$-type W.
\end{definition}
In accordance with Remark \ref{r involution}, we assume that  the applications of $E_j$  in this definition correspond to  $j$-edges.
A flat 4-chain of a  D$_0$~graph of degree  $8$ is shown in Figure \ref{f flat chain}.

\begin{remark}
To understand this definition, it is helpful to first consider the case in which the  $m_j$ are all 0.
For instance, if $\G$ is a  D$_0$~graph
in which all  $i-2$-edges and  $i-1$-edges are Knuth, then the flat  $i$-chains of  $\G$ have this property.
A flat  $i$-chain in the case all the $m_j$ are 0 is a sequence $(x^1, x^2,\ldots, x^{2h})$ of distinct vertices such that
\begin{itemize}
\item the pairs $\{x^1, x^{2}\}$, $\{x^2, x^{3}\}$, $\dots$,  $\{x^{2h-1}, x^{2h}\}$ alternate between $i$-edges and \newline \mbox{$i-2$-edges}, starting and ending with an $i$-edge, and
\item $x^1$ and  $x^{2h}$ admit  $i-2$-neighbors.
\end{itemize}
\end{remark}

\begin{remark}
\label{r flat chain}
Let $(x^1, x^2,\ldots, x^{2h})$ be a flat $i$-chain of a D$_0$~graph $\G$. The vertices $x^{2j}$ and $x^{2j+1}$ are related as follows:
let $H$ be the component of $\myRes_{[i-3,i]}\G$ containing  $x^{2j}$ and $x^{2j+1}$. By Proposition \ref{p deg4restrictions} (below), $H$ is either a 2-cycle,
a path with two edges, or a path with four edges.
The component $H$ cannot be a 2-cycle because in this case both $x^{2j}$ and $x^{2j+1}$ have $(i-1)$-type W. If $H$ is a path with two edges, then we must have $x^{2j+1} = E_{i-2}(x^{2j})$.

Finally, suppose $H$ is a path with four edges and let $(y^1,\ldots, y^5)$ be the $(i-2)$-$(i-1)$-string for  $H$ such that  $\{y^1,y^2\}$ is an  $i-1$-edge.
Then the vertices of  $H$ having $(i-1)$-type W are  $y^2$ and  $y^3$. Hence by the definition of a flat  $i$-chain, $x^{2j+1} \in \{y^4,y^5\}$.
We also must have  $m_j \in \{0,1\}$.  Checking each possibility for $x^{2j+1}$ and  $m_j$ shows that exactly one of the following must hold:
\begin{itemize}
\item $x^{2j} = y^2$, $m_j = 1$,  $x^{2j+1} = y^5$,
\item $x^{2j} = y^4$, $m_j = 0$,  $x^{2j+1} = y^5$,
\item $x^{2j} = y^5$, $m_j = 0$,  $x^{2j+1} = y^4$.
\end{itemize}
In particular, the $m_j$ and $x^{2j+1}$ in the definition of a flat  $i$-chain are determined uniquely by  $x^{2j}$.
\end{remark}


\subsection{D~graphs}
\label{ss D graphs}
Let  $\G$ be a signed, colored graph with vertex set $V$.
The following axioms from \cite{Sami} are needed to define D~graphs:
\begin{list}{}{\usecounter{ctr} \setlength{\itemsep}{2pt} \setlength{\topsep}{3pt}}
\item[(axiom 1)] $\G$ satisfies axiom 0
and $w \in V$ admits an  $i$-neighbor if and only if $\sigma(w)_{i-1} = -\sigma(w)_i$;
\item[(axiom 2)] for each $i$-edge $\{w, x \}$, $\sigma(w)_j = -\sigma(x)_j$ for $j = i-1, i,$ and $\sigma(w)_h = \sigma(x)_h$ for $h < i-2$ and $h > i+1$;
\item[(axiom 3)] for each $i$-edge $\{w, x \}$, if  $\sigma(w)_{i-2} = -\sigma(x)_{i-2}$, then $\sigma(w)_{i-2} = -\sigma(w)_{i-1}$, and if $\sigma(w)_{i+1} = -\sigma(x)_{i+1}$, then $\sigma(w)_{i+1} = -\sigma(w)_i$;
\item[(axiom $4'a$)] 
for every  $w \in V$ such that\footnote{In the statement of axiom  $4'a$ in \cite{Sami}, the conditions on $w$ are $w \in W_i(\G)$ has a non-flat $i-1$-edge, which are equivalent to the conditions in \eqref{e 4'a} assuming axioms 1 and 2.}
\begin{equation} \label{e 4'a}
\parbox{10cm}{$w$ admits an $i-1$-neighbor  $v$ and $\{w,v\}$ is not an $i$-edge, and $\sigma_{i-3}(w) \neq \sigma_{i-3}(v)$ and $\sigma_i(w) \neq \sigma_i(v)$,}
\end{equation}
the generating function of $\myRes_{[i-3,i]}(\G, w)$ is equal to that of $\myRes_{[i-2,i+1]} (\G,w)$;
\item[(axiom $4'b$)] if $(x^1,x^2,\ldots,x^{2h-1},x^{2h})$ is a flat $i$-chain and $x^j$ has $(i+1)$-type W for some $2 < j < 2h-1$, then either $x^1, \dots, x^j$
have $(i+1)$-type W or  $x^j, \dots, x^{2h}$ have  $(i+1)$-type W (or both);
\item[(axiom 5)] if $\{w,x\}$ is an $i$-edge and $\{x,y\}$ is a $j$-edge for $|i-j| \geq 3$, then $\{w,v\}$ is a $j$-edge and $\{v,y\}$ is an $i$-edge for some $v \in V$.
\end{list}
\begin{remark}\label{r axiom 4'b}
Note that we must assume  $\G$ satisfies axiom 0 in the definition of axiom  $4'b$ since this is assumed in the definition of a flat  $i$-chain.
The version of axiom $4'b$ stated here is from an earlier version of \cite{Sami} posted on 04-2014.
We use the 04-2014 version because it implies the version in \cite{Sami} (posted 08-2014) and is more straightforward to check.
For reference, the 08-2014 version is:
\[\parbox{14cm}{
if $x \in C_i(\G)$ has $(i+1)$-type W,
then there exists a maximal length flat $i$-chain $(x^1,x^2,\ldots,x^{2h-1},x^{2h})$ with $x=x^j$ such that either $x^1, \dots, x^j$ have $(i+1)$-type W or  $x^j, \dots, x^{2h}$ have  $(i+1)$-type W (or both).}\]
Here,
\[C_i(\G) := \{x \in \ver(\G) \mid x = x_j \text{ on a flat $i$-chain of length $2h$ with $2 < j < 2h - 1$}\}.\]

For comparison, here is the version of axiom $4'b$ from \cite{SamiOct13} (posted 10-2013):
\[\parbox{14cm}{
if $x \in C^\text{simple}_i(\G)$ has $(i+1)$-type $W$ and $E_{i-2}(x)$ does not, then $(E_{i-2}E_i)^m (x)$ has $(i+1)$-type $W$ for
all $m \geq 1$ for which $(E_{i-2}E_i)^{m-1}(x)$ admits an $i$-neighbor.}
\]
Here,
\[C^\text{simple}_i(\G) := \left\{x \in \ver(\G) \middle| \, \parbox{9.8cm}{$E_{i-2}E_iE_{i-2}(x), E_iE_{i-2}(x),E_{i-2}(x), x, E_i(x), E_{i-2}E_i(x)$ are well defined and distinct}\right\},\]
where we have assumed  $\G$ satisfies axioms 1 and 2 to translate from the notation of \cite{SamiOct13}.
\end{remark}

The \emph{generating function} of a signed, colored graph  $\G$ with signature function  $\sigma$
is defined to be  $\sum_{v \in \ver(\G)} Q_{\Des(v)}(\mathbf{x})$, where  $\Des(v)$ is defined to be the descent set encoded by  $\sigma(v)$, i.e.
$i \in \Des(v)$ if and only if $\sigma(v)_i = -$.
When  $\G$ is also a  D$_0$~graph, this generating function is the same as the generating function of  $\G$ previously defined (Definition \ref{d gen func}).

\begin{definition}[{\cite[Definition 4.5]{Sami}}]
A signed, colored graph  $\G$ is \emph{locally Schur positive for degree $d$} ($\LSP_d$) if the generating function of every
component of every
$K$-restriction  of $\G$ is Schur positive, for all intervals  $K$ of size $d$.
\end{definition}
Note that  $\G$ being $\LSP_d$ depends only on the signature function of  $\G$ and not on the edges of $\G$.

\begin{definition}
A \emph{D$_5$ graph} is a
D$_0$~graph that satisfies axiom 5.

Recall Definition 5.37 from \cite{Sami}:
a signed, colored graph is a
\emph{D~graph} if it satisfies axioms 1, 2, 3, $4'a$, $4'b$, 5, and  $\LSP_4$ and  $\LSP_5$.
\end{definition}

In the  next section
we will prove that a  D$_0$~graph satisfying axioms $4'b$ and 5 is a D~graph (Theorem~\ref{t D0 satisfy 123etc}).

\section{Basic properties of  D$_0$~graphs}
\label{s Basic properties of D0 graphs}
Here we show that  D$_0$~graphs satisfy  $\LSP_4$ and  $\LSP_5$ using noncommutative Schur functions.
We also show that they satisfy axioms 1, 2, 3, and $4'a$ as well as some other basic properties.

\subsection{Components, unions, and complements}
D$_0$~graphs are well behaved under some basic set operations.
\begin{proposition} \label{p component KR set}
~
\begin{list}{\emph{(\alph{ctr})}}{\usecounter{ctr} \setlength{\itemsep}{2pt} \setlength{\topsep}{3pt}}
\item The connected components of a D$_0$~graph are D$_0$~graphs.
\item The union of two D$_0$~graphs is a D$_0$~graph if their vertex sets are disjoint subsets of  $\W_n$.
\item If  $\G$ is a D$_0$~graph on a vertex set $W \subseteq \W_n$, then there is a  D$_0$~graph on the vertex set $\W_n\setminus W$.
\end{list}
\end{proposition}
\begin{proof}
It is easy to see that any connected component of a  D$_0$~graph satisfies (i)--(v) in Definition \ref{d D0 graph}, hence (a).
Statement (b) also follows easily from the definition of  D$_0$~graphs.
For (c), note that  $\W_n$ and  $W$ are KR sets; 
hence  $\sum_{\e{w} \in \W_n \setminus W} \e{w} \in (\Irkst)^\perp$ by Proposition-Definition \ref{pd KR set}.
Therefore, there is a D$_0$~graph on the vertex set  $\W_n \setminus W$, again by Proposition-Definition \ref{pd KR set}.
\end{proof}

\subsection{Local structure of D$_0$~graphs}

For a word  $\e{w} \in \W$, the \emph{standardization} of  $\e{w}$, denoted $\e{w}^{\stand}$,
is the permutation obtained from $\e{w}$ by relabeling its smallest letter by 1, its next smallest letter by 2, etc.
\begin{proposition}\label{p order preserving}
If  $\G$ is a connected D$_0$~graph of degree  $n$ on a vertex set  $W$, then there is a unique  D$_0$~graph  $\G^\stand$ on the set  $W^\stand := \{\e{w}^\stand \mid \e{w} \in W\} \subseteq \S_n$ such that the standardization map
induces a graph isomorphism $\G \to \G^\stand$ which preserves edge color and type and vertex signatures.
\end{proposition}
\begin{proof}
This is straightforward from the fact that Knuth and rotation transformations only depend on the relative order of three distinct letters, not on their exact values.
The assumption that $\G$ is connected is necessary to ensure that the standardization map restricted to  the vertices of  $\G$ is injective.
\end{proof}
\begin{remark}
For most of this paper we work with $\W_n$ instead of  $\S_n$ because the former has the convenient property that any subword of some $\e{w} \in \W$ also belongs to $\W$.
However, by the previous  proposition, we could phrase all our results in terms of subsets of $\S_n$ and lose no generality.
\end{remark}

\begin{proposition}\label{p deg4restrictions}
If  $\G$ is a D$_0$~graph and  $K$ is an interval of size $4$, then the connected components of (the underlying colored graph of) $\myRes_{K} \G$ are of the form:

\vspace{-1mm}
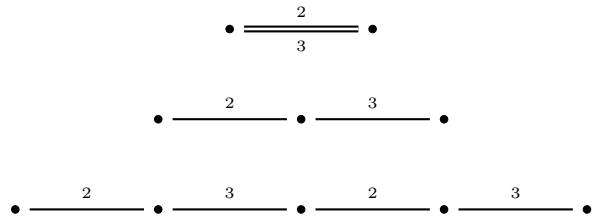
\begin{center}         \centerfloat
\begin{tikzpicture}[xscale = 1.9,yscale = 1.2]
\tikzstyle{vertex}=[inner sep=1pt, outer sep=4pt, draw, circle, fill]
\tikzstyle{aedge} = [draw, thin, ->,black]
\tikzstyle{edge} = [draw, thick, -,black]
\tikzstyle{dashededge} = [draw, very thick, dashed, black]
\tikzstyle{LabelStyleH} = [text=black, anchor=south]
\tikzstyle{LabelStyleH2} = [text=black, anchor=north]
\tikzstyle{LabelStyleV} = [text=black, anchor=east]
\tikzstyle{doubleedge} = [draw, thick, double distance=1pt, -,black]
\tikzstyle{hiddenedge} = [draw=none, thick, double distance=1pt, -,black]

\node[vertex] (v1) at (2,2){};
\node[vertex] (v2) at (3,2){};
\node[vertex] (v3) at (4,2){};
\node[vertex] (v4) at (1,1){};
\node[vertex] (v5) at (2,1){};
\node[vertex] (v6) at (3,1){};
\node[vertex] (v7) at (4,1){};
\node[vertex] (v8) at (5,1){};
\node[vertex] (v9) at (2.5,3){};
\node[vertex] (v10) at (3.5,3){};

\draw[edge] (v1) to node[LabelStyleH]{\Tiny$2$} (v2);
\draw[edge] (v2) to node[LabelStyleH]{\Tiny$3$} (v3);

\draw[edge] (v4) to node[LabelStyleH]{\Tiny$2$} (v5);
\draw[edge] (v5) to node[LabelStyleH]{\Tiny$3$} (v6);
\draw[edge] (v6) to node[LabelStyleH]{\Tiny$2$} (v7);
\draw[edge] (v7) to node[LabelStyleH]{\Tiny$3$} (v8);

\draw[doubleedge] (v9) to node[LabelStyleH]{\Tiny$2$} (v10);
\draw[hiddenedge] (v9) to node[LabelStyleH2]{\Tiny$3$} (v10);
\end{tikzpicture}
\end{center}
\end{proposition}
\begin{proof}
Let  $H$ be a connected component of $\myRes_{K} \G$.
By Remark \ref{r restriction} and Proposition \ref{p order preserving}, we can assume that $H$ is a D$_0$~graph on a subset of $\S_4$.
Since every vertex of  $H$ has $i$-degree at most 1,
$H$ is a path or an even cycle. Moreover,
$H$ cannot be a path with an even number of vertices since
this would imply that the generating function of  $H$ is either
\begin{align*}
& Q_{\{1\}}(\mathbf{x}) + Q_{\{2,3\}}(\mathbf{x})+  (h-1) \big( Q_{\{1,3\}}(\mathbf{x}) + Q_{\{2\}}(\mathbf{x}) \big)
= Q_{\{1\}}(\mathbf{x}) + Q_{\{2,3\}}(\mathbf{x}) + (h-1) s_{(2,2)}(\mathbf{x}), \text{ or } \\
& Q_{\{1,2\}}(\mathbf{x}) + Q_{\{3\}}(\mathbf{x})+  (h-1) \big( Q_{\{1,3\}}(\mathbf{x}) +  Q_{\{2\}}(\mathbf{x}) \big)
=  Q_{\{1,2\}}(\mathbf{x}) + Q_{\{3\}}(\mathbf{x}) + (h-1) s_{(2,2)}(\mathbf{x}),
\end{align*}
where  $2h = |\ver(H)|$.
But this is impossible since these functions are not symmetric while the generating function of $H$ is Schur positive (by Corollary \ref{c LSP45}, below).

Now let  $(\e{x^1,\dots, x^t})$ be a 2-3-string for $H$. 
 Suppose for a contradiction that  $H$  is not one of the three graphs above.
Then either $t \ge 7$ or  $H$ is a cycle of length  $> 2$.
Set  $\e{w} = \e{x^{(t+1)/2}}$.
The  desired contradiction follows from the claim
\begin{equation} \label{e rotation 3 edge}
\parbox{10cm}{at least one of $\e{w},E_2(\e{w}), E_3(\e{w}), E_3E_2(\e{w}),E_2E_3(\e{w})$ does not admit both a $2$-neighbor and a $3$-neighbor.}
\end{equation}

Before proving \eqref{e rotation 3 edge}, we show the following for  any vertex  $\e{v}$ of  $H$:
\begin{align}
\parbox{13cm}{Suppose $\e{v_4} \in \{\e{1,4}\}$.  If $\e{v}$ admits both a 2-neighbor and a 3-neighbor, then $E_2(\e{v})$ does not.} \label{e 14 on end 1}\\[1mm]
\parbox{13cm}{Suppose $\e{v_1} \in \{\e{1,4}\}$.  If $\e{v}$ admits both a 2-neighbor and a 3-neighbor, then $E_3(\e{v})$ does not.} \label{e 14 on end 2}
\end{align}
To prove \eqref{e 14 on end 1}, suppose $\e{v_4 = 4}$, the case  $\e{v_4 = 1}$ being similar.
If  $\e{v}$ admits a 2-neighbor, then exactly one of $\sigma_2(\e{v}), \sigma_2(E_2(\e{v}))$ is $+$.
Since  $\e{v_4=4}$, there holds $\sigma_3(\e{v})=  \sigma_3(E_2(\e{v})) = +$, so exactly one of  $\e{v}, E_2(\e{v})$ admits a 3-neighbor.
Statement \eqref{e 14 on end 1} follows.  The proof of \eqref{e 14 on end 2} is similar.

We now prove \eqref{e rotation 3 edge} by showing that at least one of  $\e{w}, E_2(\e{w}), E_3(\e{w})$ satisfies
the first hypothesis of \eqref{e 14 on end 1} or \eqref{e 14 on end 2}.
If $\{\e{w_2, w_3}\} \ne \{\e{1,4}\}$, then  $\e{w}$ satisfies the first hypothesis of \eqref{e 14 on end 1} or \eqref{e 14 on end 2}.
If $\{\e{w_2, w_3}\} = \{\e{1,4}\}$, then $E_2(\e{w})$ (resp.  $E_3(\e{w})$) satisfies the first hypothesis of \eqref{e 14 on end 1} or \eqref{e 14 on end 2}
if  the 2-edge (resp. 3-edge) at $\e{w}$  is a  rotation edge.
The only  remaining possibility is that the 2-edge and 3-edge at  $\e{w}$ are Knuth,
but  this implies that  $H$ is a cycle of length 2,  contrary to our supposition.
This completes the proof of \eqref{e rotation 3 edge}.
\end{proof}

We can now show that axiom  $4'a$ in the definition of a D~graph always holds for D$_0$~graphs.  This was also proved as part of \cite[Theorem 6.1]{Sami2} by a computer check.
A similar result is proved in \cite[Theorem 5.38]{Sami} using similar methods.

\begin{proposition}\label{p D0 axiom 4'a}
D$_0$~graphs satisfy axiom $4'a$.
\end{proposition}
\begin{proof}
Let  $\G$ be a D$_0$~graph.
Since axiom $4'a$ for  $\G$ depends only on $\myRes_K \G$ for  $|K| = 5$, by Proposition \ref{p order preserving} we may assume that the vertices of  $\G$ lie in $\S_5$ and the $i$ in the definition of axiom $4'a$ is $4$.
Let $\{\e{w},\e{v}\}$ be an  $i-1$-edge of  $\G$ as in \eqref{e 4'a}.
Since $\G$ satisfies axioms 1 and 2 (see Theorem~\ref{t D0 satisfy 123etc}, below),
the conditions \eqref{e 4'a} on $\e{w}$ imply that $\myRes_{[1, 4]}(\G, \e{w})$ and $\myRes_{[2, 5]}(\G, \e{w})$ each contain at least 3 edges.
Then by  Proposition \ref{p deg4restrictions}, $\myRes_{[1, 4]}(\G, \e{w})$ is a path with 4 edges and $\myRes_{[2, 5]}(\G, \e{w})$ is a path with  4 edges.

 Since a Knuth transformation only changes two adjacent letters, the conditions $\sigma_{i-3}(\e{w}) \neq \sigma_{i-3}(\e{v})$ and $\sigma_i(\e{w}) \neq \sigma_i(\e{v})$ imply that the $3$-edge  $\{\e{w},\e{v}\}$ must be a rotation edge. In fact, one can check that these conditions imply that $\{\e{w}, \e{v}\}$ is equal to $\{\e{21534}, \e{23154}\}$ or $\{\e{43512}, \e{45132}\}$.
Both cases are similar, so assume the former.
It is easy to see that no matter the type of the 4-edge at $\e{21534}$, the vertex $E_4(\e{21534})$ does not admit a 3-neighbor.
Hence the  3-4-string containing  $\e{w}$ is $\big( E_4(\e{21534}), \e{21534}, \e{23154}, E_4(\e{23154}), E_3 E_4(\e{23154}) \big)$; see Figure \ref{f axiom 4'a} (the vertices of $\G$ and types of the edges need not be as in the figure, but the numbers on the edges and the signatures of the words must be as in the figure).
It follows that the signatures of the vertices of $\myRes_{[2,5]} (\G,\e{w})$ are  $(++-,+-+,-+-,+-+,-++)$ and therefore its generating function is
$s_{(2,2)}(\mathbf{x}) + s_{(3,1)}(\mathbf{x})$.
A similar argument shows that this is also the generating function of $\myRes_{[1,4]} (\G,\e{w})$.
\begin{figure}
\begin{tikzpicture}[xscale = 2.2,yscale = 2]
\tikzstyle{vertex}=[inner sep=0pt, outer sep=4pt]
\tikzstyle{aedge} = [draw, thin, ->,black]
\tikzstyle{edge} = [draw, thick, -,black]
\tikzstyle{dashededge} = [draw, very thick, dashed, black]
\tikzstyle{LabelStyleH} = [text=black, anchor=south]
\tikzstyle{LabelStyleV} = [text=black, anchor=east]

 \node[vertex] (v1) at (1,3){\footnotesize$\e{13524}$};
 \node[vertex] (v2) at (2,3){\footnotesize$\e{15234}$};
 \node[vertex] (v3) at (2,2){\footnotesize$\e{21534}$};
 \node[vertex] (v4) at (2,1){\footnotesize$\e{21453}$};
 \node[vertex] (v5) at (3,3){\footnotesize$\e{31254}$};
 \node[vertex] (v6) at (3,2){\footnotesize$\e{23154}$};
 \node[vertex] (v7) at (3,1){\footnotesize$\e{23415}$};
 \node[vertex] (v8) at (4,1){\footnotesize$\e{24135}$};

 \draw[edge] (v1) to node[LabelStyleH]{\Tiny$\tilde{3} $} (v2);
 \draw[edge] (v2) to node[LabelStyleV]{\Tiny$\tilde{2} $} (v3);
 \draw[edge] (v3) to node[LabelStyleV]{\Tiny$\tilde{4} $} (v4);
 \draw[edge] (v3) to node[LabelStyleH]{\Tiny$\tilde{3} $} (v6);
 \draw[edge] (v5) to node[LabelStyleV]{\Tiny$\tilde{2} $} (v6);
 \draw[edge] (v6) to node[LabelStyleV]{\Tiny$\tilde{4} $} (v7);
 \draw[edge] (v7) to node[LabelStyleH]{\Tiny$\tilde{3} $} (v8);
\end{tikzpicture}
\caption{\label{f axiom 4'a}}
\end{figure}
\end{proof}

\subsection{Noncommutative flagged Schur functions}
\label{ss Noncommutative flagged Schur functions}
For the inductive computation of the noncommutative Schur functions $\mathfrak{J}_\lambda(\mathbf{u})$ in the next subsection, we need their flagged generalization.

For any positive integer $d$ and $m \in \{0,1,\dots,N\}$, recall that
\[
e_d([m])= \sum_{m \geq i_1 > i_2 > \cdots > i_d \geq 1} u_{i_1}u_{i_2}\cdots u_{i_d} \ \in \, \U,
\]
where  $[0] := \{\}$.  Also recall $e_0([m])=1$ and $e_{d}([m]) = 0$ for $d<0$.
Note that
\begin{align*}
e_d([0]) =&
\begin{cases}
  1 & \text{ if } d = 0, \\
  0 & \text{otherwise.}
\end{cases} \notag
\end{align*}
We will use the following fact several times in the proof of Theorem~\ref{t fat hook} below:
\begin{align}\label{e ek induction}
e_d([m]) =&~\e{m} e_{d-1}([{m-1}])+e_d([{m-1}]) & \text{if  $m > 0$ and  $d$ is any integer}.
\end{align}

Given a weak composition
$\alpha=(\alpha_1,\dots,\alpha_l)$ and  $\mathbf{n} = (n_1, n_2, \dots, n_l) \in \{0,1,\ldots,N\}^l$, define the \emph{noncommutative column flagged Schur function} $J_\alpha(\mathbf{n}) = J_\alpha(\mathbf{n})(\mathbf{u})$ by
\begin{align*}
J_{\alpha}(\mathbf{n})&
=\sum_{\pi\in \S_{l}}
\sgn(\pi) \, e_{\alpha_1+\pi(1)-1}([{n_1}]) e_{\alpha_2+\pi(2)-2}([{n_2}]) \cdots e_{\alpha_{l}+\pi(l)-l}([{n_l}]).
\end{align*}
Note that if $\lambda$ is a partition then $\mathfrak{J}_\lambda(\mathbf{u}) = J_{\lambda'}(N, N, \ldots, N)$,
so the noncommutative column flagged Schur functions generalize the noncommutative Schur functions defined earlier.

It follows from Lemma \ref{l es commute} that
\begin{align}\label{e elem sym Jswap}
J_\alpha(\mathbf{n}) \equiv -J_{\alpha_1, \ldots, \alpha_{j-1},\alpha_{j+1}-1,\alpha_j+1,\ldots,\alpha_l}(\mathbf{n}) \ (\bmod{\ \Irkst}) \quad \text{whenever  $n_j = n_{j+1}$}.
\end{align}
In particular,
\begin{align}\label{e elem sym J0}
J_\alpha(\mathbf{n}) \equiv 0 \ (\bmod{\ \Irkst}) \quad \text{whenever  $\alpha_j = \alpha_{j+1}-1$ and  $n_j=n_{j+1}$.}
\end{align}

\subsection{Local Schur positivity of D$_0$~graphs}
\label{ss local schur positivity}
We now  prove
that D$_0$~graphs satisfy  $\LSP_4$ and  $\LSP_5$
by showing that certain noncommutative Schur functions are positive sums of monomials in $\U/\Irkst$.

We first review  some terminology and fix conventions  related to  diagrams and tableaux.
Partition diagrams are drawn with the English (matrix-style) convention so that row (resp. column) labels start with 1 and increase from north to south (resp. west to east).
Let $\lambda$ be a partition.
A \emph{semistandard Young tableau} $T$ of shape $\lambda$ is the diagram of $\lambda$ together with a letter occupying each of its cells
such that entries strictly increase from north to south in each column and weakly increase from west to east in each row.
Write  $T_{r,c}$ for the entry of  $T$ in the  $r$-th row and  $c$-th column.
The \emph{row reading word} of  $T$, denoted $\reading(T)$, is the word obtained by concatenating the rows of  $T$ (reading each row left to right), starting with the bottom row.
The \emph{column reading word} of  $T$, denoted $\creading(T)$, is the word obtained by concatenating the columns of  $T$ (reading each column bottom to top), starting with the leftmost column.

For example, $\reading\Big({\tiny\tableau{1&2\\3&4\\5} }\Big)= \e{53412}$ and $\creading\Big({\tiny\tableau{1&3\\2&4\\5}}\Big) = \e{52143}$.

For a partition  $\lambda$ and a tuple $\mathbf{n} = (n_1,\ldots,n_l)$ of nonnegative integers,
let  $\SYT(\lambda)^\mathbf{n}$ denote the set of semistandard Young tableaux of shape $\lambda$ having no repeated letter
such that the entries in column  $c$ lie in $[n_c]$.

We now state and prove the full version of Theorem~\ref{t intro hook} from the introduction.  It is similar to Theorems 17 and 19 of \cite{LamRibbon}.

\begin{theorem}\label{t fat hook}
Let  $\mathbf{n} = (n_1,n_2, \dots, n_l)$ with $0\le n_1\le n_2 \le \cdots \le n_l$.
\begin{list}{\emph{(\roman{ctr})}}{\usecounter{ctr} \setlength{\itemsep}{2pt} \setlength{\topsep}{3pt}}
\item Let $\lambda$ be the hook shape $(a,1^b)$ ($a \ge 1$, $b = l-1 \ge 0$).  Then in $\U/\Irkst$,
\begin{align} \label{e fat hook1}
J_{\lambda}(\mathbf{n}) = \sum_{T\in \SYT(\lambda')^\mathbf{n}} \creading(T).
\end{align}
\item If  $a \ge 2$, $\lambda = (a,2)$, and  $l=2$, then in $\U/\Irkst$,
\begin{align} \label{e fat hook2}
J_{\lambda}(\mathbf{n}) = \sum_{T\in \SYT(\lambda')^\mathbf{n}, \ T_{1,2} > T_{2,1}} \creading(T) +
\sum_{T\in \SYT(\lambda')^\mathbf{n}, \ T_{1,2} < T_{2,1}} \reading(T).
\end{align}
\end{list}
In particular, $\mathfrak{J}_\lambda(\mathbf{u})$ is a positive sum of monomials in  $\U/\Irkst$ if  $\lambda$ is a hook shape or of the form $(a,2)'$.
\end{theorem}

\begin{proof}
Throughout the proof, we write $\e{x \equiv y}$ to mean that  $\e{x}$ and  $\e{y}$ are congruent~$\bmod{\Irkst}$.
We first prove (i). The proof is by induction on $a$, $b$, and the $n_i$.
If  $n_1 = 0$, then the first row of the matrix $\mat{e_{\lambda_r+c-r}([{n_r}]) }_{r,c \in [l]}$ is zero, hence its determinant $J_{\lambda}(\mathbf{n})$ is 0.
The right side of \eqref{e fat hook1} is certainly also 0 in this case because  $\SYT(\lambda')^\mathbf{n}$ is empty.

We now may assume  $n_1 > 0$. Set  $\mathbf{n}^- = (n_1-1,n_2,\ldots, n_l)$.
First suppose $a\ge 2$. Then \eqref{e ek induction} yields
\begin{align*}
J_{(a,1^b)}(\mathbf{n})
=&~ \e{n_1}J_{(a-1,1^b)}(\mathbf{n}^-) + J_{(a,1^b)}(\mathbf{n}^-)\\
\equiv&~ \sum_{T\in \SYT(\lambda')^\mathbf{n}, \ T_{a,1} = n_1} \creading(T) +
\sum_{T\in \SYT(\lambda')^{\mathbf{n}^-} } \creading(T) \\
=&~  \sum_{T\in \SYT(\lambda')^\mathbf{n}} \creading(T).
\end{align*}
The congruence is by induction on $n_1$. The last equality is by the fact that
$\{T \in \SYT(\lambda')^\mathbf{n} \mid T_{a,1} = n_1\} \sqcup \SYT(\lambda')^{\mathbf{n}^-}$
 is a set partition of $\SYT(\lambda')^\mathbf{n}$.

Now suppose $a=1$.
Applying \eqref{e ek induction}, we obtain
\begin{align*}
J_{(1,1^b)}(\mathbf{n})
=&~ \e{n_1}J_{(0,1^b)}(\mathbf{n}^-)+J_{(1,1^b)}(\mathbf{n}^-) \\
=&~ \e{n_1}\big( J_{(1^b)}(n_2, n_3,\dots)- J_{(1^b)}(n_1-1,n_3, \dots) \big)+J_{(1,1^b)}(\mathbf{n}^-) \\
\equiv&~ \sum_{T\in \SYT(\lambda')^\mathbf{n}, \ T_{1,1} = n_1} \creading(T) +
\sum_{T\in \SYT(\lambda')^{\mathbf{n}^-} } \creading(T) \\
=&~  \sum_{T\in \SYT(\lambda')^\mathbf{n}} \creading(T).
\end{align*}
Note that for $\alpha = (0,1^b)$,
$J_{\alpha}(\mathbf{n}^-)$ is a (noncommutative version of) the determinant of the matrix $\mat{e_{\alpha_r+c-r}([{n_r^-}]) }_{r,c \in [l]}$, which has first column $\mat{1 & 1 & 0 & \cdots & 0}^T$.
The second equality then follows from computing this determinant by expanding along this column.
The congruence is by induction on  $b$, and the last equality is by the fact that
$\{T \in \SYT(\lambda')^\mathbf{n} \mid T_{1,1} = n_1\} \sqcup \SYT(\lambda')^{\mathbf{n}^-}$
is a set partition of $\SYT(\lambda')^\mathbf{n}$.

We now prove (ii).
The statement for  $a > 2$ reduces to the  $a = 2$ case by a similar argument to the  $a \ge 2$ case in the proof of (i).
So we may assume  $\lambda = (2,2)$.

For each subset  $S$ of $\{u_1,\dots,u_N\}$, let $(\U/\Irkst)_S$ denote the subspace of~\,$\U/\Irkst$ spanned by the permutations of  $S$.
By Proposition \ref{p order preserving Irk only} (i) and the fact that $\U/\Irkst$ is the direct sum of the subspaces $(\U/\Irkst)_S$, the result for general  $\mathbf{n}$ reduces to the case $n_l \le |\lambda| = 4$.
It therefore suffices to show
\begin{align}
J_{\lambda}(1, 4) &\equiv0,        \label{e J22 1}\\
J_{\lambda}(2, 4) &\equiv \e{2143},  \label{e J22 2}\\
J_{\lambda}(3, 4) &\equiv\e{2143} + \e{3412},  \label{e J22 3}\\
J_{\lambda}(4, 4) &\equiv \e{2143} + \e{3412},  \label{e J22 4}
\end{align}
since these right hand sides are equal to
\[
\sum_{T\in \SYT(\lambda')^{n_1\, 4}, \ T_{1,2} > T_{2,1}} \creading(T) +
\sum_{T\in \SYT(\lambda')^{n_1 \, 4}, \ T_{1,2} < T_{2,1}} \reading(T),
\]
for  $n_1 = 1,2,3,4$, respectively.

The verification of \eqref{e J22 1} and \eqref{e J22 2} is immediate from the definition of  $J_{\lambda}(\mathbf{n})$.
For \eqref{e J22 3}, we compute directly
\begin{align*}
J_{22}(3, 4) &\equiv
\e{2143} + \e{3142} + \e{3241}  -\e{3214}  \\
& = \e{2143} + \e{3412 + 3(14-41)2} + \e{32(41-14)} \\
& \equiv \e{2143} + \e{3412 + 32(14-41)} + \e{32(41-14)}\\
& = \e{2143} + \e{3412}.
\end{align*}
Finally, \eqref{e J22 4} follows from \eqref{e J22 3} and
\begin{align*}
J_{22}(4, 4)
=\e{4}J_{12}(3, 4)+J_{22}(3, 4)
\equiv \e{4}J_{12}(3, 3)+J_{22}(3, 4)
\equiv J_{22}(3, 4).
\end{align*}
Here, the first equality is by \eqref{e ek induction}, the first congruence is because words with repeated letters are 0 in  $\U/\Irkst$, and the second congruence is by \eqref{e elem sym J0}.
\end{proof}

The  final ingredient to prove that   D$_0$~graphs satisfy $\LSP_4$ and  $\LSP_5$ is
a result about how the noncommutative Schur functions behave under the  anti-automorphism  $\rev : \U/\Irkst \to \U/\Irkst$ of Proposition \ref{p order preserving Irk only} (iii).

The \emph{noncommutative homogeneous symmetric functions} are given by
\[
h_d(\mathbf{u})=\sum_{1 \le i_1 \le i_2 \le \cdots \le i_d \le N}u_{i_1}u_{i_2}\cdots u_{i_d} \ \in \, \U.
\]
Let $\Lambda(\mathbf{u})$ be the subalgebra of~\,$\U/\Irkst$ generated by $e_1(\mathbf{u}), e_2(\mathbf{u}), \dots$,
and  let $\Lambda(\mathbf{x})$  denote the ring of symmetric functions in the commuting variables $x_1,x_2,\ldots$.
By Lemma \ref{l es commute}, there is an algebra homomorphism
\[\Psi: \Lambda(\mathbf{x}) \to \Lambda(\mathbf{u}), \ \ e_d(\mathbf{x}) \mapsto e_d(\mathbf{u}).\]

\begin{lemma} \label{l subalgebra h}
There holds $\Psi(h_d(\mathbf{x})) = h_d(\mathbf{u})$. Hence the $h_d(\mathbf{u})$ belong to  $\Lambda(\mathbf{u})$ and pairwise commute in $\U/\Irkst$.
\end{lemma}
\begin{proof}
By Theorem~\ref{t fat hook} (i), $\mathfrak{J}_{1^d}(\mathbf{u}) = h_d(\mathbf{u})$ in  $\U/\Irkst$.
The first statement then follows from
\[\mathfrak{J}_{1^d}(\mathbf{u}) = \det(e_{1+j-i}(\mathbf{u})) = \Psi(\det(e_{1+j-i}(\mathbf{x}))) = \Psi(h_d(\mathbf{x})),\]
where  the last equality is by the Jacobi-Trudi formula in the commuting variables $x_1, x_2, \ldots$.
The concluding statement follows from the fact that the algebra  $\Lambda(\mathbf{u})$ is commutative.
\end{proof}

\begin{proposition}\label{p rev J}
In  $\U/\Irkst$, there holds $\rev(\mathfrak{J}_\lambda(\mathbf{u})) = \mathfrak{J}_{\lambda'}(\mathbf{u})$.
\end{proposition}
\begin{proof}
By the version of the Jacobi-Trudi formula expressing $s_\lambda(\mathbf{x})$ in terms of the  $h_d(\mathbf{x})$ and the version expressing $s_\lambda(\mathbf{x})$ in terms of the $e_d(\mathbf{x})$,
\begin{equation}\label{e hse}
\det(h_{\lambda_i+j-i}(\mathbf{x})) = s_{\lambda}(\mathbf{x}) = \det(e_{\lambda'_i+j-i}(\mathbf{x})).
\end{equation}
The identity $\rev(e_d(\mathbf{u})) = h_d(\mathbf{u})$ holds in  $\U/\Irkst$ because of relation \eqref{e Irkst2}.
Hence in  $\U/\Irkst$,
\[\mathfrak{J}_{\lambda'}(\mathbf{u}) = \det(e_{\lambda_i+j-i}(\mathbf{u})) = \rev(\det(h_{\lambda_i+j-i}(\mathbf{u}))) = \rev(\mathfrak{J}_{\lambda}(\mathbf{u})). \]
The last equality is obtained by applying  $\Psi$ to \eqref{e hse} and using Lemma \ref{l subalgebra h}.
\end{proof}

\begin{corollary} \label{c LSP45}
D$_0$~graphs satisfy  $\LSP_4$ and $\LSP_5$.
\end{corollary}
\begin{proof}
Let  $H$ be the  $K$-restriction of a D$_0$~graph for $K$ of size 4 or 5.
Theorem~\ref{t fat hook} and  Proposition \ref{p rev J} imply that  $\mathfrak{J}_\lambda(\mathbf{u})$  is a positive sum of monomials in  $\U/\Irkst$
for all $\lambda$ of size  $\leq 5$.
Then by Corollary \ref{c symmetric bijectivizations} and Remark \ref{r restriction},
the generating function of each connected component of  $H$ is Schur positive.
\end{proof}

The following theorem summarizes the results of this section.
The similar result \cite[Theorem 6.1]{Sami2} is proved by a computer check.
\begin{theorem}\label{t D0 satisfy 123etc}
D$_0$~graphs satisfy axioms 1, 2, 3, $4'a$ and  $\LSP_4$ and  $\LSP_5$.  Hence a  D$_0$~graph is a  D~graph if and only if it satisfies axioms $4'b$ and 5.
\end{theorem}
\begin{proof}
Axiom 1 is built into the definition of a D$_0$~graph because if  $\e{w}$ is a vertex of a D$_0$~graph, then
\[
\text{$\e{w}$ admits an  $i$-neighbor  $\iff$ $\e{w}$ belongs to a KR$_i$ square  $\iff$ $\sigma_{i-1}(\e{w}) \neq \sigma_i(\e{w})$.}
\]
Axiom 2 is also clear because Knuth and rotation transformations only involve three consecutive letters.

Axiom $4'a$ was proved in Proposition \ref{p D0 axiom 4'a} and $\LSP_4$ and  $\LSP_5$ were just proved in  Corollary \ref{c LSP45}.
Assuming axiom 1, axiom 3 is equivalent to the condition that any component of any $K$-restriction with $|K| = 4$ has more than 1 edge.
This is implied by $\LSP_4$ (for more about this, see \cite[Remark 5.2]{Sami}).
\end{proof}

\section{LLT polynomials}
\label{s LLT}

LLT polynomials are certain  $q$-analogs of products of skew Schur functions, first defined by Lascoux, Leclerc, and Thibon in \cite{LLT}.
There are two versions of LLT polynomials (which we distinguish following the notation of \cite{GH}): the combinatorial LLT polynomials of \cite{LLT} defined using spin, and the new variant combinatorial LLT polynomials of \cite{HHLRU} defined using inversion numbers. 
Although the theory of noncommutative Schur functions is well suited to studying the former (see \cite{LamRibbon}), we prefer to work with the latter to follow \cite{Sami,BLamLLT} and because inversion numbers are easier to calculate than spin.

In this section, we first recall an expression for LLT polynomials in terms of quasisymmetric functions from  \cite{Sami}.  We then introduce a simplified version  $\U_q/\Jlamst{k}$ of Lam's algebra of ribbon Schur operators from \cite{LamRibbon} and show how LLT polynomials and their Schur expansions can be interpreted in terms of this algebra.
We then connect $\U_q/\Jlamst{k}$ to the algebra $\U/\Irkst$ and the graphs  $\G^{(k)}_{c,D}$ defined by Assaf \cite{Sami}.

\subsection{LLT polynomials in terms of quasisymmetric functions}
Recall our conventions for partition diagrams from \textsection\ref{ss local schur positivity}.
A \emph{skew shape}  $\lambda/\mu$, where $\mu \subseteq \lambda$ are partitions, is the diagram obtained by removing the cells of $\mu$ from the diagram of $\lambda$.
The \emph{content} of a cell $(i,j)$ of a skew shape is $j-i$.
A \emph{skew shape with contents} is an equivalence class of skew shapes, where two skew shapes are equivalent if there is a content and
order preserving bijection between their diagrams.
A \emph{$k$-tuple of skew shapes with contents} is a $k$-tuple $\bm{\beta} =(\beta^{(0)},\dots,\beta^{(k-1)})$ such that each  $\beta^{(i)}$ is a skew shape with contents.
The \emph{shifted content} of a cell  $z$ of  $\bm{\beta}$ is
\[\tilde{c}(z) = k\cdot c(z)+i,\]
when $z \in \beta^{(i)}$ and where $c(z)$ is the usual content of $z$ regarded as a cell of $\beta^{(i)}$.

For a $k$-tuple $\bm{\beta} =(\beta^{(0)},\dots,\beta^{(k-1)})$ of skew shapes with contents such that
\begin{align}
\text{each $\beta^{(i)}$ contains no $2 \times 2$ square and its shifted contents lie in  $[N]$,}\label{e beta condition}
\end{align}
define  $\Wi{k}(\bm{\beta}) \subseteq \W$ to be the set of words $\e{v}$ such that
\begin{itemize}
\item $\e{v}$ is a rearrangement of the shifted contents of $\bm{\beta}$,
\item  for each  $i$ and each pair  $z, z'$ of cells of  $\beta^{(i)}$ such that $z'$
lies immediately east or south of  $z$, the letter $\tilde{c}(z)$ occurs before  $\tilde{c}(z')$ in  $\e{v}$.
\end{itemize}
Also define the following statistic on words  $\e{v}\in \W$:
\begin{align*}
\invi{k}(\e{v}) &= \big| \big\{ \, (i,j)\mid \text{$i<j$ and $0<\e{v_i}-\e{v_j}<k$} \, \big\} \big|.
\end{align*}

\begin{example}
\setlength{\cellsize}{1.9ex}
\label{ex LLT}
Let $\bm{\beta}$ be the 3-tuple $\left(\partition{&~\\&},\partition{&~&~\\&~&},\partition{&&~\\&~&~}\right) =
(2,32,33)/(1,11,21)$
of skew shapes with contents.
Its shifted contents are
\[{\tiny\left(\tableau{&3\\&},\tableau{&4&7\\&1&},\tableau{&&8\\&2&5}\right)}.\]
Some elements of  $\Wi{3}(\bm{\beta})$, together with their  $3$-inversion numbers, are
\[
\begin{array}{l|ccccc}
\e{v} &
\e{2413857} & \e{3417285}& \e{4871235} &\e{8341275}&  \e{2853417}\\[1.4mm]
\invi{3}(\e{v}) & 3 & 4 & 4&5 & 5
\end{array}
\]
\end{example}

The next proposition is an adaptation of \cite[Corollary 4.3]{Sami}, which expresses the new variant combinatorial LLT polynomials in terms of quasisymmetric functions.  We take this as the definition of these polynomials; they were originally defined in  \cite{HHLRU}  as a sum over  $k$-tuples of semistandard tableaux with an inversion statistic.
Be aware that the words used here are inverses of those used in \cite{Sami}; see \cite[\textsection2.5--2.6]{BLamLLT} for details.
\begin{proposition}[{\cite[Proposition 2.8]{BLamLLT}}]\label{p assafi}
Let  $\bm{\beta}$ be a  $k$-tuple of skew shapes with contents satisfying \eqref{e beta condition}.  The \emph{new variant combinatorial LLT polynomial} indexed by  $\bm{\beta}$ is
\[\mathcal{G}_{\bm{\beta}}(\mathbf{x};q) = \sum_{\e{v}\in\Wi{k}(\bm{\beta})}q^{\invi{k}(\e{v})}Q_{\Des(\e{v})}(\mathbf{x}). \]
\end{proposition}

\begin{remark}
There is a statement similar to Proposition \ref{p assafi} without the restriction \eqref{e beta condition} on  $\bm{\beta}$, however in order to be compatible with
the convention in this paper of only focusing on words with no repeated letter, we chose to simplify the discussion in this section
to this special case.
For the same reason, we work with a quotient of Lam's algebra by the relation \eqref{e Irkst2} instead of the original algebra defined by Lam.
See \cite{BF} for the full treatment without this simplification.
\end{remark}

\subsection{Lam's algebra of ribbon Schur operators}
\label{ss Lams algebra}

Lam \cite{LamRibbon} defines  an algebra of ribbon Schur operators, which gives an elegant algebraic framework for LLT polynomials.  It  sets the stage for  applying the theory of noncommutative Schur functions from \cite{FG}  to the problem of computing Schur expansions of LLT polynomials.
We now introduce a simplified version of this algebra and review how the Schur expansions of certain LLT polynomials can be interpreted using this algebra and
Theorem~\ref{t basics}.

Let $\QQ(q)$ be the field of rational functions in the indeterminate $q$ with coefficients in $\QQ$ and let $\U_q = \QQ(q) \tsr_{\ZZ} \U$.
Let  $\U_q/\Jlamst{k}$ be the quotient of~\,$\U_q$ by the following relations (let $\Jlamst{k}$ denote the corresponding two-sided ideal of~\,$\U_q$):
\begin{alignat}{3}
&\e{ac} = \e{ca} \qquad &&\text{for $c-a > k$,} \label{e far commute} \\
&\e{ab} = q^{-1} \e{ba} \qquad &&\text{for $0<b-a<k$,} \label{e q commute}\\
&\e{w} = 0 \qquad &&\text{for words $\e{w}$ with a repeated letter.} \label{e no repeat Lam}
\end{alignat}

The \emph{new variant $q$-Littlewood-Richardson coefficients}, denoted $\mathfrak{c}_{\bm{\beta}}^\lambda(q)$, are the coefficients of the Schur expansion of the new variant
combinatorial LLT polynomials, i.e.
\[\mathcal{G}_{\bm{\beta}}(\mathbf{x};q) = \sum_{\lambda}\mathfrak{c}_{\bm{\beta}}^\lambda(q)s_\lambda(\mathbf{x}).\]
The new variant $q$-Littlewood-Richardson coefficients are known to be polynomials in $q$ with nonnegative integer coefficients \cite{LT00, GH}.
We now express these coefficients (or rather, the subset of them that fit with our simplified setup) using the theory of noncommutative Schur functions as was done in \cite{LamRibbon}.

\begin{proposition}
\label{p llt formula}
Let $\bm{\beta}$ be a  $k$-tuple of skew shapes with contents satisfying \eqref{e beta condition}.
Then
\[ f := \sum_{\e{v}\in\Wi{k}(\bm{\beta})} q^{\invi{k}(\e{v})}\e{v}\in (\Jlamst{k})^\perp, \]
and the new variant combinatorial LLT polynomial can be written as
\[ \mathcal{G}_{\bm{\beta}}(\mathbf{x};q) = \Delta(f). \]
Hence the new variant $q$-Littlewood-Richardson coefficient is given by
\[\mathfrak{c}_{\bm{\beta}}^\lambda(q) = \langle \mathfrak{J}_{\lambda}(\mathbf{u}), f \rangle.  \]
\end{proposition}
\begin{proof}
For the first statement, we must show that  $\langle \e{v (ca - ac) w}, f \rangle = 0$, for any $c-a > k$ (as in \eqref{e far commute}) and words $\e{v}$, $\e{w}$.
This is clear since  $\invi{k}(\e{vcaw}) = \invi{k}(\e{vacw})$, and $\e{vcaw} \in \Wi{k}(\bm{\beta})$ if and only if $\e{vacw} \in \Wi{k}(\bm{\beta})$.
We must also show that $\langle \e{v}(q\,\e{ab}-\e{ba}) \e{w}, f \rangle = 0$, for any $0< b-a < k$ (as in \eqref{e q commute}) and words $\e{v}$, $\e{w}$.
This holds because
\[\text{$\e{vabw} \in \Wi{k}(\bm{\beta})$ \ \ if and only if \ \ $\e{vbaw} \in \Wi{k}(\bm{\beta})$, and} \]
\[q^{-\invi{k}(\e{vbaw})}\big\langle  \e{v}(q\,\e{ab}-\e{ba}) \e{w}, f \big\rangle = \big\langle q^{-\invi{k}(\e{vabw})}\e{vabw}- q^{-\invi{k}(\e{vbaw})}\e{vbaw}, f \big\rangle = 0,\]
where the first equality is by $\invi{k}(\e{vabw}) = \invi{k}(\e{vbaw})-1$.

The second statement is immediate from the definition of  $\mathcal{G}_{\bm{\beta}}(\mathbf{x};q)$ (Proposition \ref{p assafi}) and  $\Delta$.

For the last statement, it is not hard to show that the relations \eqref{e far commute}--\eqref{e no repeat Lam} imply the relations \eqref{e Irkst1}--\eqref{e Irkst2}, hence $\U_q/\Jlamst{k}$ is a quotient of $\QQ(q) \tsr_\ZZ (\U/\Irkst)$ (see the proof of Proposition \ref{p lam and assaf}).  Hence the last statement follows from Lemma \ref{l es commute} and an easy modification of Theorem~\ref{t basics} to the coefficient field  $\QQ(q)$.
\end{proof}

Though working with the algebra  $\U_q/\Jlamst{k}$ over $\QQ(q)$ is particularly convenient for expressing $q$-Littlewood Richardson coefficients as in Proposition \ref{p llt formula},
to relate to Assaf's work it is convenient to have a version of  this algebra without the $q$.
Accordingly, define $\Ilamst{k} := \U\cap \Jlamst{k} \subseteq \U_q$, where  $\U \subseteq \U_q$ is the $\ZZ$-subalgebra of~\,$\U_q$ generated by 1 and the  $u_i$.
\begin{proposition}\label{p Ilamst equiv classes}
The nonzero equivalence classes of~\,$\U/\Ilamst{k}$ are
\[\{\e{v} \in \Wi{k}(\bm{\beta}) \mid \invi{k}(\e{v}) = t\} \]
for $\bm{\beta}$ as in \eqref{e beta condition} and $t \in \NN$ such that this set is nonempty.
\end{proposition}
\begin{proof}
Since applying a relation \eqref{e far commute} or \eqref{e q commute} to a word does not change the relative position of
two letters that differ by $k$, it follows that if  $\e{v} \equiv \e{w} \mod \Ilamst{k}$, then  $\invi{k}(\e{v}) = \invi{k}(\e{w})$, and $\e{v}$ and $\e{w}$ belong to the same $\Wi{k}(\bm{\beta})$ for some $\bm{\beta}$.
Hence each set  $\{\e{v} \in \Wi{k}(\bm{\beta}) \mid \invi{k}(\e{v}) = t\}$ is a union of nonzero equivalence classes. To see that this is actually a nonzero equivalence class, observe that any word  in $\Wi{k}(\bm{\beta})$ can be transformed into a power of  $q$ times the word $\e{w(\bm{\beta})}$ by a sequence of applications of \eqref{e far commute} and \eqref{e q commute}.
Here,
\[\e{w(\bm{\beta}) := w^0w^1\cdots w^{k-1},}\]
where $\e{w^i}$ is the word obtained by reading the shifted contents of $\beta^{(i)}$ along antidiagonals, starting from the northwesternmost and reading each antidiagonal from southwest to northeast.
For the $\bm{\beta}$ in Example \ref{ex LLT}, $\e{w^0 = 3}, \e{w^1 = 417}, \e{w^2 = 285},$ and $\e{w(\bm{\beta}) = 3417285}$.
Therefore if $\e{x}, \e{y} \in \{\e{v} \in \Wi{k}(\bm{\beta}) \mid \invi{k}(\e{v}) = t\}$, then
\[\e{x} \equiv q^{\invi{k}(\e{x}) - \invi{k}(\e{w(\bm{\beta})})} \e{w}(\bm{\beta}) = q^{\invi{k}(\e{y}) - \invi{k}(\e{w(\bm{\beta})})} \e{w}(\bm{\beta}) \equiv \e{y} \, \bmod{\, \Jlamst{k}}.\]
Hence  $\e{x} \equiv \e{y} \, \bmod{\, \Ilamst{k}}.$
\end{proof}

\subsection{Assaf's LLT bijectivizations}
\label{ss Assaf's LLT bijectivizations}
We now relate Lam's \cite{LamRibbon}  algebraic approach to studying LLT polynomials to Assaf's \cite{Sami} approach using D~graphs.
To do this, we first define the following  bijectivizations of~\,$\U/\Irkst$ and related D$_0$~graphs.
Let  $\U/\Iassafst{k}$ denote the quotient of~\,$\U$ by the following relations (let $\Iassafst{k}$ denote the corresponding two-sided ideal of~\,$\U$):
\begin{alignat}{3}
&\e{bac} = \e{bca} \qquad &&\text{for $c-a > k$ and $a<b<c$,} \label{e Iassaf1} \\
&\e{acb} = \e{cab} \qquad &&\text{for $c-a > k$ and $a<b<c$,} \label{e Iassaf2}\\
&\e{bac} = \e{acb} \qquad &&\text{for $c-a \leq k$ and $a<b<c$,} \label{e Iassaf3}\\
&\e{bca} = \e{cab} \qquad &&\text{for $c-a \leq k$ and $a<b<c$,} \label{e Iassaf4}\\
&\e{w} = 0 \qquad &&\text{for words $\e{w}$ with a repeated letter.} \label{e Iassaf5}
\end{alignat}
It is easy to see that these relations imply the relations \eqref{e Irkst1}--\eqref{e Irkst2}, hence $\U/\Iassafst{k}$ is a bijectivization of~\,$\U/\Irkst$.
These algebras generalize the plactic algebra since  $\U/\Iassafst{1} = \U/\Iplacst$ (see Example \ref{ex plac}).

\begin{definition} \label{d Assaf D graph}
Define $\G_k^\stand$ to be the D$_0$~graph on the vertex set  $\W$ such that
\[\text{the type of the KR square $\e{v \stars w}$ of  $\G_k^\stand$ is }
\begin{cases}
\text{Knuth} & \text{if } c-a >k, \\
\text{rotation} & \text{if } c-a \le k.
\end{cases}
\]
We define an \emph{Assaf LLT$_k$ D~graph} to be a union of connected components of $\G_k^\stand$.
\end{definition}
By construction,  the nonzero equivalence classes of~\,$\U/\Iassafst{k}$ (see \textsection\ref{ss bijectivizations})
are the components of $\G_k^\stand$.

\begin{proposition} \label{p lam and assaf}
The algebra  $\U/\Ilamst{k}$ is a quotient of~\,$\U/\Iassafst{k}$.
\end{proposition}
\begin{proof}
Clearly, the relations \eqref{e Iassaf1} and \eqref{e Iassaf2} are implied by \eqref{e far commute}.
Next, we compute using \eqref{e q commute} twice,
\[\e{bac} \equiv q\e{abc} \equiv \e{acb} \ \bmod{\, \Jlamst{k}} \ \qquad \text{for $c-a \leq k$ and $a<b<c$.}\]
It follows that \eqref{e Iassaf3} holds in $\U/\Ilamst{k}$.
The relation \eqref{e Iassaf4} is similar.
\end{proof}

By the previous proposition, each nonzero equivalence class
$\{\e{v} \in \Wi{k}(\bm{\beta}) \mid \invi{k}(\e{v}) = t\}$  of~\,$\U/\Ilamst{k}$
is a union of nonzero equivalence classes of~\,$\U/\Iassafst{k}$.
This means that each set $\{\e{v} \in \Wi{k}(\bm{\beta}) \mid \invi{k}(\e{v}) = t\}$ with  $\bm{\beta}$ and  $t$ as in Proposition \ref{p Ilamst equiv classes}
is the vertex set of a union of connected components of $\G_k^\stand$;
define $\G_k^\stand(\bm{\beta}, t)$ to be this union of components.

We can now rephrase the connection between the graphs $\G_k^\stand$ and LLT polynomials from \cite[\textsection4.2]{Sami} in our language:
\begin{corollary}\label{c Sami LLT graphs}
Let $\bm{\beta}$ be as in \eqref{e beta condition}. There holds $\sum_{\e{v} \in \Wi{k}(\bm{\beta}),\, \invi{k}(\e{v}) = t} \e{v} \in (\Iassafst{k})^\perp$.
The expression for LLT polynomials from Proposition \ref{p llt formula} can be rewritten as
\[\mathcal{G}_{\bm{\beta}}(\mathbf{x};q)
= \sum_t q^t \, \Delta\Big( \sum_{\e{v} \in \Wi{k}(\bm{\beta}),\, \invi{k}(\e{v}) = t} \e{v} \Big) = \sum_t q^t\big( \text{gen. function of }\G_k^\stand(\bm{\beta}, t) \big). \]
\end{corollary}

For each  $k$-tuple $\bm{\beta}$ of skew shapes with contents satisfying \eqref{e beta condition},
the graph $\bigsqcup_t \G_k^\stand(\bm{\beta},t)$ is isomorphic to
the graph $\G^{(k)}_{c,D}$ defined in \cite[\textsection4.2]{Sami} as a signed, colored graph, where $c,D$ are the content vector and $k$-descent set corresponding to $\bm{\beta}$ as in \cite[Equation 4.2]{Sami}
(see also  \cite[Proposition 2.6]{BLamLLT}).

It is clear that the  $\G_k^\stand$ satisfy axiom 5, and it is shown in  \cite{Sami} that the graphs  $\G^{(k)}_{c,D}$ and hence the  $\G_k^\stand$ satisfy axiom $4'b$.
\begin{theorem}[{\cite[Theorem 5.38]{Sami}}]
The  D$_0$~graphs $\G_k^\stand$ are D~graphs.  Hence, as their name suggests,  the Assaf LLT$_k$ D~graphs are D~graphs.
\end{theorem}

\begin{example}
Figure \ref{f permutation D graph 32+311} depicts the graph $\bigsqcup_t \G_3^\stand(\bm{\beta},t)$ for
$\bm{\beta} = \left(\partition{&~\\&},\partition{&~\\&~},\partition{~&~\\&}\right) = (2/1,22/11,2)$.
Hence by Corollary \ref{c Sami LLT graphs}, 
\[\mathcal{G}_{\bm{\beta}}(\mathbf{x};q)=
q^2 s_{41} + q^3 (s_{32}+s_{311}) +  q^4 (s_{311}+s_{221}) + q^5 s_{2111}.
\]
\end{example}

\begin{figure}
\centerfloat
\begin{tikzpicture}[xscale = 2,yscale = 1]
\tikzstyle{vertex}=[inner sep=0pt, outer sep=4pt]
\tikzstyle{framedvertex}=[inner sep=3pt, outer sep=4pt, draw=gray]
\tikzstyle{aedge} = [draw, thin, ->,black]
\tikzstyle{edge} = [draw, thick, -,black]
\tikzstyle{doubleedge} = [draw, thick, double distance=1pt, -,black]
\tikzstyle{hiddenedge} = [draw=none, thick, double distance=1pt, -,black]
\tikzstyle{dashededge} = [draw, very thick, dashed, black]
\tikzstyle{LabelStyleH} = [text=black, anchor=south]
\tikzstyle{LabelStyleHn} = [text=black, anchor=north]
\tikzstyle{LabelStyleV} = [text=black, anchor=east]

\begin{scope} [xshift = 2.6cm]
\node[vertex] (v1) at (1,1){\footnotesize$\e{23451}$};
\node[vertex] (v2) at (2,1){\footnotesize$\e{23415}$};
\node[vertex] (v3) at (3,1){\footnotesize$\e{24135}$};
\node[vertex] (v4) at (4,1){\footnotesize$\e{41235}$};
\node[vertex] (inv2) at (6,1){\small $\invi{3} = 2$};
\end{scope}

\draw[edge] (v1) to node[LabelStyleH]{\Tiny$4 $} (v2);
\draw[edge] (v2) to node[LabelStyleH]{\Tiny$\tilde{3} $} (v3);
\draw[edge] (v3) to node[LabelStyleH]{\Tiny$\tilde{2} $} (v4);

\begin{scope} [yshift = -3.5cm]
\node[vertex] (v1) at (9,2){\footnotesize$\e{23541}$};
\node[vertex] (v2) at (8,2){\footnotesize$\e{24351}$};
\node[vertex] (v3) at (7,3){\footnotesize$\e{32451}$};
\node[vertex] (v4) at (1,2){\footnotesize$\e{24513}$};
\node[vertex] (v5) at (2,2){\footnotesize$\e{24153}$};
\node[vertex] (v6) at (3,2){\footnotesize$\e{41253}$};
\node[vertex] (v7) at (7,1){\footnotesize$\e{24315}$};
\node[vertex] (v8) at (6,2){\footnotesize$\e{32415}$};
\node[vertex] (v9) at (5,2){\footnotesize$\e{34125}$};
\node[vertex] (v10) at (4,2){\footnotesize$\e{41325}$};
\node[vertex] (v11) at (4,3){\footnotesize$\e{42135}$};
\node[vertex] (inv3) at (8.6,3){\small $\invi{3} = 3$};
\end{scope}
\draw[edge] (v1) to node[LabelStyleH]{\Tiny$\tilde{3} $} (v2);
\draw[edge] (v2) to node[LabelStyleH]{\Tiny$\tilde{2} $} (v3);
\draw[edge] (v2) to node[LabelStyleH]{\Tiny$4 $} (v7);
\draw[edge] (v3) to node[LabelStyleH]{\Tiny$4 $} (v8);
\draw[doubleedge] (v4) to node[LabelStyleH]{\Tiny$3$} (v5);
\draw[hiddenedge] (v4) to node[LabelStyleHn]{\Tiny$4$} (v5);
\draw[edge] (v5) to node[LabelStyleH]{\Tiny$\tilde{2} $} (v6);
\draw[edge] (v6) to node[LabelStyleH]{\Tiny$\tilde{4} $} (v10);
\draw[edge] (v7) to node[LabelStyleH]{\Tiny$\tilde{2} $} (v8);
\draw[edge] (v8) to node[LabelStyleH]{\Tiny$\tilde{3} $} (v9);
\draw[edge] (v9) to node[LabelStyleH]{\Tiny$\tilde{2} $} (v10);
\draw[edge] (v10) to node[LabelStyleV]{\Tiny$\tilde{3} $} (v11);

\begin{scope}[yshift = -7cm]
\node[vertex] (v1) at (4,3){\footnotesize$\e{24531}$};
\node[vertex] (v2) at (4,2){\footnotesize$\e{25341}$};
\node[vertex] (v3) at (5,2){\footnotesize$\e{32541}$};
\node[vertex] (v4) at (6,2){\footnotesize$\e{34251}$};
\node[vertex] (v5) at (7,1){\footnotesize$\e{42351}$};
\node[vertex] (v6) at (3,2){\footnotesize$\e{25413}$};
\node[vertex] (v7) at (2,2){\footnotesize$\e{42513}$};
\node[vertex] (v8) at (1,2){\footnotesize$\e{42153}$};
\node[vertex] (v9) at (7,3){\footnotesize$\e{34215}$};
\node[vertex] (v10) at (8,2){\footnotesize$\e{42315}$};
\node[vertex] (v11) at (9,2){\footnotesize$\e{43125}$};
\node[vertex] (inv4) at (8.6,3){\small $\invi{3} = 4$};
\end{scope}
\draw[edge] (v1) to node[LabelStyleV]{\Tiny$\tilde{3} $} (v2);
\draw[edge] (v2) to node[LabelStyleH]{\Tiny$\tilde{2} $} (v3);
\draw[edge] (v2) to node[LabelStyleH]{\Tiny$\tilde{4} $} (v6);
\draw[edge] (v3) to node[LabelStyleH]{\Tiny$\tilde{3} $} (v4);
\draw[edge] (v4) to node[LabelStyleH]{\Tiny$4 $} (v9);
\draw[edge] (v4) to node[LabelStyleH]{\Tiny$\tilde{2} $} (v5);
\draw[edge] (v5) to node[LabelStyleH]{\Tiny$4 $} (v10);
\draw[edge] (v6) to node[LabelStyleH]{\Tiny$\tilde{2} $} (v7);
\draw[doubleedge] (v7) to node[LabelStyleH]{\Tiny$3$} (v8);
\draw[hiddenedge] (v7) to node[LabelStyleHn]{\Tiny$4$} (v8);
\draw[edge] (v9) to node[LabelStyleH]{\Tiny$\tilde{2} $} (v10);
\draw[edge] (v10) to node[LabelStyleH]{\Tiny$\tilde{3} $} (v11);

\begin{scope} [yshift = -8.5cm, xshift = 2.6cm]
\node[vertex] (v1) at (1,1){\footnotesize$\e{25431}$};
\node[vertex] (v2) at (2,1){\footnotesize$\e{42531}$};
\node[vertex] (v3) at (3,1){\footnotesize$\e{43251}$};
\node[vertex] (v4) at (4,1){\footnotesize$\e{43215}$};
\node[vertex] (inv4) at (6,1){\small $\invi{3} = 5$};
\end{scope}
\draw[edge] (v1) to node[LabelStyleH]{\Tiny$\tilde{2} $} (v2);
\draw[edge] (v2) to node[LabelStyleH]{\Tiny$\tilde{3} $} (v3);
\draw[edge] (v3) to node[LabelStyleH]{\Tiny$4 $} (v4);

\end{tikzpicture}
\caption{\label{f permutation D graph 32+311} The Assaf LLT$_3$ D~graph $\bigsqcup_t \G_3^\stand(\bm{\beta},t)$ for
$\bm{\beta} = (2/1,22/11,2)$.
The components have generating functions (from top to bottom)  $s_{41}$, $s_{32}+s_{311}$, $s_{311}+s_{221}$, $s_{2111}$.
}
\end{figure}
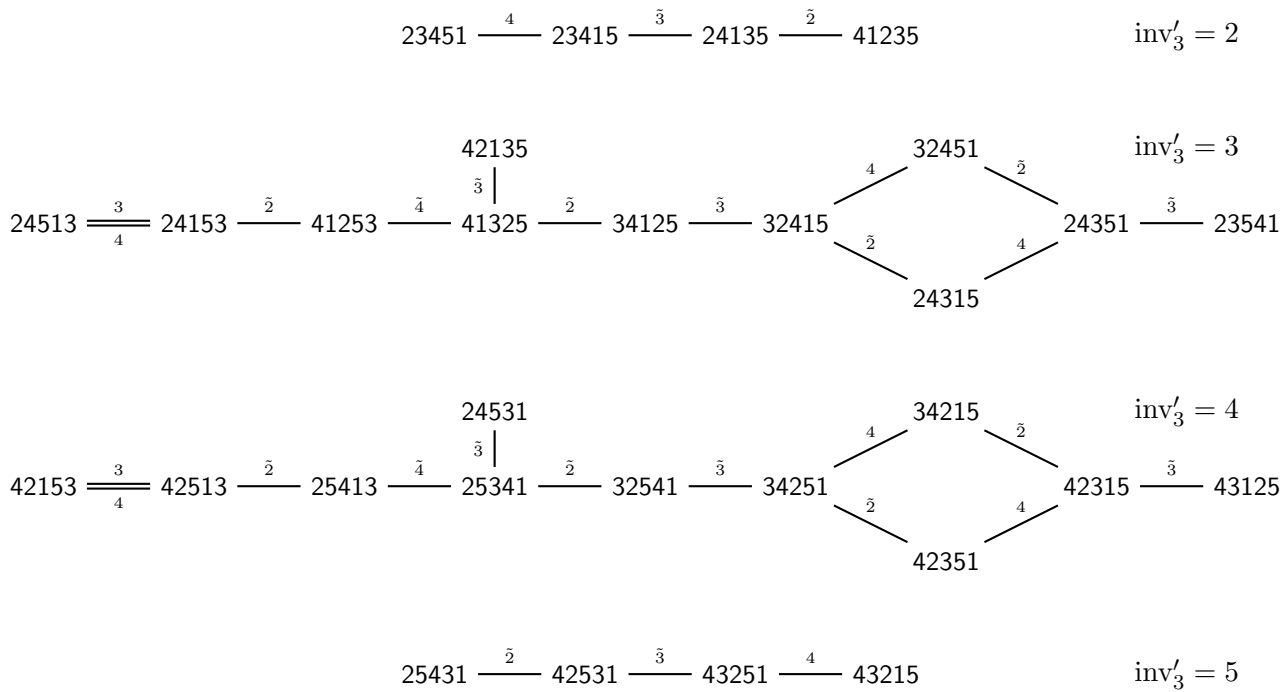

\section{Linear programming and Schur positivity}
\label{s Linear programming and Schur positivity}
Here we state and prove the full version of Theorem~\ref{t intro linear program 2} from the introduction, which
relates the monomial positivity of $\mathfrak{J}_\lambda(\mathbf{u})$ in a quotient of~\,$\U/\Irkst$
to the Schur positivity of generating functions of certain D$_0$~graphs.
Its proof uses linear programming  duality.

\subsection{Background on linear programming}
\label{ss Background on linear programming}
We review some terminology related to linear programs.
A good reference for this background is \cite[Chapter 7]{Schrijver}.

Let  $A$ be  a $t \times m$ real matrix and $\mathbf{b} \in \RR^t$, $\mathbf{c} \in \RR^m$ be column vectors.
The linear program $\lp$ associated to this data is the optimization problem
minimize $\mathbf{c}^T \mathbf{x}$ over all $\mathbf{x}\in \RR^m$ satisfying the linear constraints
\begin{align*}
A\mathbf{x} &\ge \mathbf{b}, \\
\mathbf{x} &\ge \mathbf{0}.
\end{align*}
The linear program dual to  $\lp$, denoted  $\lp^\vee$, is to maximize $\mathbf{b}^T \mathbf{y}$ over all $\mathbf{y}\in \RR^t$ satisfying the linear constraints
\begin{align*}
A^T \mathbf{y} & \le \mathbf{c},\\
\mathbf{y} & \ge \mathbf{0}.
\end{align*}

A \emph{solution} of a linear program is a vector satisfying the stated constraints
(e.g.  $\mathbf{x}\in \RR^m$ is a solution of  $\lp$ if and only if $A\mathbf{x} \ge \mathbf{b}, \ \mathbf{x}\ge \mathbf{0}$).
The \emph{feasible region} is the set of all solutions.
A linear program is \emph{feasible} if  the feasible region is nonempty and $\emph{infeasible}$ otherwise.
A linear program whose objective is to minimize $\mathbf{c}^T\mathbf{x}$ is \emph{bounded} if this quantity is bounded below as $\mathbf{x}$ ranges over the feasible region. Similarly, a linear program whose objective is to maximize $\mathbf{b}^T \mathbf{y}$ is \emph{bounded} if this quantity is bounded above as $\mathbf{y}$ ranges over the feasible region.
\begin{theorem}[Duality Theorem]
Exactly one of the following is true for the linear program  $\lp$ and its dual  $\lp^\vee$:
\begin{itemize}
\item  $\lp$ and  $\lp^\vee$ are feasible and bounded and
\[\min\{\mathbf{c}^T \mathbf{x} \mid A\mathbf{x} \ge \mathbf{b}, \ \mathbf{x}\ge \mathbf{0}\} = \max \{\mathbf{b}^T \mathbf{y} \mid A^T \mathbf{y} \le \mathbf{c}, \ \mathbf{y} \ge \mathbf{0}\}.\]
\item $\lp$ is  unbounded and $\lp^\vee$ is  infeasible.
\item $\lp^\vee$ is  unbounded and $\lp$ is  infeasible.
\item  $\lp$ and $\lp^\vee$ are infeasible.
\end{itemize}
\end{theorem}
If  the first bullet point holds above, then  $\min\{\mathbf{c}^T \mathbf{x} \mid A\mathbf{x} \ge \mathbf{b}, \ \mathbf{x}\ge \mathbf{0}\}$ is the \emph{optimal value}
of  $\lp$; any solution $\mathbf{x}^*$ of  $\lp$ such that  $\mathbf{c}^T \mathbf{x}^*$ equals the optimal value
is an \emph{optimal solution}.
Similarly, \emph{an optimal integer solution} of  $\lp$ is an integer vector $\mathbf{x}^*$ such that $\mathbf{c}^T\mathbf{x}^*$ is the minimum of $\mathbf{c}^T \mathbf{x}$ over all integer vectors $\mathbf{x}$ in the feasible region; this minimum value is the \emph{optimal integer value} of  $\lp$. We use similar terminology for the dual.

\subsection{Linear programs for monomial positivity of  $\mathfrak{J}_\lambda(\mathbf{u})$}
Below we define three linear programs $\lp_{\cap}$,  $\lp_{\cap,=}$, and  $\lp_\cap^\vee$,
and then in Theorem~\ref{t linear program 2} we relate them to the problem of expressing  $\mathfrak{J}_\lambda(\mathbf{u})$
as a positive sum of monomials in a quotient of~\,$\U/\Irkst$.  We must first fix some notation.

Throughout this subsection, fix a positive integer  $d$ and a partition $\lambda$ of  $d$.
Let  $\U_d$ denote the degree  $d$ homogeneous component of~\,$\U$, and for any  homogeneous ideal $J$ of~\,$\U$, let  $J_d$
denote the degree  $d$ homogeneous component of  $J$.
Let $\e{w^1,\ldots,w^m}$ be the words of~\,$\U$ of length  $d$.
We work in the vector space $\RR\U_d := \RR \tsr_\ZZ \U_d \cong \RR^m$ and identify any vector
$\mathbf{x} \in \RR^m$ with the element $\sum_{i} x_i\e{w^i} \in \RR\U_d$.
In  particular, the element $\mathfrak{J}_\lambda(\mathbf{u}) \in \U_d$
is  identified with a vector in  $\RR^m$.

Let $J^{(1)}, \ldots, J^{(l)}$ be ideals of~\,$\U$ such that for each  $j \in [l]$
\begin{itemize}
\item the quotient $\U/J^{(j)}$ is a bijectivization of~\,$\U/\Irkst$,
\item the union of the nonzero equivalence classes of~\,$\U/J^{(j)}$ of degree  $d$ is  $\W_d$.
\end{itemize}
For each $j \in [l]$, let $h_1^{(j)},\ldots, h_{t_j}^{(j)}$ be the  $\ZZ$-basis of $(J^{(j)}_d)^\perp$ described in \textsection\ref{ss bijectivizations};
each $h_i^{(j)}$ is a zero-one vector (in the basis  $\e{w^1,\dots,w^m}$) whose support is a nonzero  equivalence class of~\,$\U/J^{(j)}$.

We next define matrices and vectors for the linear programs.
Let  $A$ be the $(\sum t_j) \times m$ matrix given by
\[A = \mat{A_1 \\ A_2 \\ \vdots \\ A_l}, \quad \text{where} \
A_j = \mat{\rule[2.2pt]{1em}{0.4pt} \! \!  & h_1^{(j)} & \! \! \rule[2.2pt]{1em}{0.4pt}\\
\, & \vdots &  \\
\rule[2.2pt]{1em}{0.4pt} \! \!  &  h_{t_j}^{(j)} &  \! \! \rule[2.2pt]{1em}{0.4pt} }.
\]
Define the column vector  $\mathbf{b} \in \RR^{\sum t_j}$ by
\begin{align}\label{e b def}
\mathbf{b} = \mat{\mathbf{b}^1 \\ \mathbf{b}^2 \\ \vdots \\ \mathbf{b}^l}, \quad \text{where} \ \ \mathbf{b}^j = A_j \mathfrak{J}_\lambda(\mathbf{u}) = \mat{\langle \mathfrak{J}_\lambda(\mathbf{u}), h_1^{(j)}\rangle \\ \vdots \\ \langle \mathfrak{J}_\lambda(\mathbf{u}), h_{t_j}^{(j)} \rangle} \in \RR^{t_j}.
\end{align}

Let  $\mathbf{c} \in \RR^m$ be the zero-one vector whose support is $\W_d$, i.e.
\[c_i =
\begin{cases}
1 & \text{ if } \e{w^i} \in \W_d, \\
0 & \text{ if } \e{w^i} \notin \W_d.
\end{cases}\]
The linear program $\lp_\cap$ is to minimize $\mathbf{c}^T \mathbf{x}$ over all $\mathbf{x} \in \RR^m$ satisfying
\begin{align*}
A\mathbf{x} &\ge \mathbf{b}, \\
\mathbf{x} &\ge \mathbf{0}.
\end{align*}

Define the linear program  $\lp_{\cap, =}$ to be the same as  $\lp_\cap$ except with the constraint  $A\mathbf{x} \ge \mathbf{b}$ replaced by $ A\mathbf{x} = \mathbf{b}$.

The linear program dual to $\lp_{\cap}$, denoted $\lp_{\cap}^\vee$, is to maximize $\mathbf{b}^T \mathbf{y}$ over all $\mathbf{y} \in \RR^{\sum t_j}$ satisfying
\begin{align*}
A^T \mathbf{y} & \le \mathbf{c},\\
\mathbf{y} & \ge \mathbf{0}.
\end{align*}


\begin{theorem} \label{t linear program 2}
Maintain the notation above.  In particular, we are assuming that  $d$ and  $\lambda$ have been fixed and
for each  $j \in [l]$,
$\U/J^{(j)}$ is a bijectivization of~\,$\U/\Irkst$.
Let $J = \bigcap_j J^{(j)}$.
Recall that  $\e{w^1,\dots, w^m}$ are the words of~\,$\U_d$.
Let  $\RR_{\ge 0}^m$ denote the polyhedron $\{\sum_{i} x_i \e{w^i} \mid x_i \ge 0\} \subseteq \RR\U_d$.
Set $M = \mathbf{c}^T \mathfrak{J}_\lambda(\mathbf{u})$.
The following are equivalent:
\begin{list}{\emph{(\roman{ctr})}}{\usecounter{ctr} \setlength{\itemsep}{2pt} \setlength{\topsep}{3pt}}
\item $\mathfrak{J}_\lambda(\mathbf{u})$ lies in the image of  $\RR_{\ge 0}^m$ in  $\RR \U_d/\RR J_d$.
\item The linear program $\lp_{\cap, =}$ is feasible.
\item The linear program $\lp_\cap$ has optimal value $M$.
\item The linear program $\lp_\cap^\vee$ has optimal value $M$.
\item For every $\mathbf{g} \in \sum_j \big((\RR J^{(j)}_d)^\perp \cap \RR_{\ge 0}^m \big)$ such that  $\mathbf{g} \le \mathbf{c}$,
$\langle \mathfrak{J}_\lambda(\mathbf{u}), \mathbf{c}-\mathbf{g} \rangle \ge 0$.
\end{list}
\end{theorem}

Statement (i) is another way of saying that $\mathfrak{J}_{\lambda}(\mathbf{u})$ is a positive real sum of monomials in  $\RR \tsr_\ZZ \U/J$.
Note that $(\RR J_d)^\perp = \sum_j (\RR J^{(j)}_d)^\perp$ and  $\mathbf{c} \in J_d^\perp$, hence
(v) says that $\langle \mathfrak{J}_\lambda(\mathbf{u}), f \rangle \ge 0$
for certain  $f$ that are a positive real sum of monomials and lie in  $(\RR J_d)^\perp$.
Thus (i) certainly implies (v), but the converse is not obvious.
This theorem therefore
provides us with a deeper understanding of the key ingredient
required for the Fomin-Greene approach (see the discussion at the end of \textsection \ref{ss FG setup}).

\begin{proof}
First note that $J^{(j)}_d = \ker(A_j)$ i.e.  $J^{(j)}_d = \{\mathbf{x} \in \U_d \mid A_j \mathbf{x} = \mathbf{0}\}$.
Hence $J_d = \ker(A)$.  There also holds
$A \mathfrak{J}_{\lambda}(\mathbf{u}) = \mathbf{b}$ by \eqref{e b def}.

We therefore have, for  $\mathbf{x} \in \RR^m$,
\begin{align*}
&\mathbf{x} \text{ is a solution of } \lp_{\cap, =} \\
\iff& A(\mathfrak{J}_\lambda(\mathbf{u})-\mathbf{x}) = \mathbf{0} \text{ and } \mathbf{x} \ge \mathbf{0} \\
\iff& \mathfrak{J}_\lambda(\mathbf{u}) - \mathbf{x} \in \RR J_d \text{ and }\mathbf{x} \ge \mathbf{0}.
\end{align*}
This proves the equivalence of (i) and (ii).
%

To prove that (ii) and (iii) are equivalent, we first  establish some basic facts.
The rows of  $A_j$ are the nonzero equivalence classes of~\,$\U/J^{(j)}$ of degree  $d$ and a  $\ZZ$-basis of  $(J^{(j)}_d)^\perp$.
Hence
\begin{align}
&\mathbf{1}^T A_j =  \mathbf{c}^T,  \label{e lp fact1}\\
&\mathbf{1}^T \mathbf{b}^j = \mathbf{1}^T A_j \mathfrak{J}_{\lambda}(\mathbf{u}) = \mathbf{c}^T \mathfrak{J}_{\lambda}(\mathbf{u}) = M, \label{e lp fact1b} \\
&\mathbf{c} \in (J^{(j)}_d)^\perp \subseteq J_d^\perp, \label{e lp fact2}
\end{align}
where  $\mathbf{1} \in \RR^{t_j}$ denotes the all-ones vector.

We now prove (iii) implies (ii).
Suppose $\lp_\cap$ has an optimal solution $\mathbf{x}^*$ with optimal value $M = \mathbf{c}^T \mathbf{x}^*$.
We must show that $A_j \mathbf{x}^* = \mathbf{b}^j$.
By \eqref{e lp fact1b}, the inequality  $\mathbf{b}^j \le A_j\mathbf{x}^*$, and \eqref{e lp fact1} we obtain
\begin{align*}
M = \mathbf{1}^T\mathbf{b}^j \le \mathbf{1}^TA_j\mathbf{x}^* = \mathbf{c}^T\mathbf{x}^* = M,
\end{align*}
so the inequality is an equality.
It follows that $A_j \mathbf{x}^* = \mathbf{b}^{j}$, as desired.

To prove (ii) implies (iii),
suppose $\mathbf{x}$ is a solution of $\lp_{\cap,=}$.
By the proof of the equivalence of (i) and (ii),
$\mathfrak{J}_{\lambda}(\mathbf{u})-\mathbf{x} \in \RR J_d$.
Hence by \eqref{e lp fact2},
$\mathbf{c}^T (\mathfrak{J}_{\lambda}(\mathbf{u})-\mathbf{x}) = 0$, which gives
$\mathbf{c}^T \mathbf{x} = \mathbf{c}^T \mathfrak{J}_{\lambda}(\mathbf{u}) = M$.
Hence $\lp_\cap$ has optimal value $\le M$.
To see that  $\lp_\cap$ has optimal value $\ge M$, we compute using \eqref{e lp fact1b} and \eqref{e lp fact1} that for any solution  $\mathbf{x}$ of  $\lp_\cap$,
\begin{align*}
M = \mathbf{1}^T\mathbf{b}^j \le \mathbf{1}^TA_j\mathbf{x} = \mathbf{c}^T\mathbf{x}.
\end{align*}

Statements (iii) and (iv) are equivalent by the Duality Theorem since  $\lp_\cap$ and  $\lp_{\cap}^\vee$ are always feasible:
$\lp_\cap$ is always feasible because $A \mathbf{x} \ge \mathbf{b}$ whenever the $x_i$ are sufficiently large,
and  $\lp_\cap^\vee$ is always feasible because $\mathbf{y} = \mathbf{0}$ is always a solution.

To prove the equivalence of (iv) and (v), we first claim that
as $\mathbf{y}$ ranges over the feasible region of $\lp_\cap^\vee$, the vector $\mathbf{g} = A^T\mathbf{y} \in \RR^m$ ranges over
\[ \bigg( \sum_j \big((\RR J^{(j)}_d)^\perp \cap \RR_{\ge 0}^m \big) \bigg) \, \bigcap \, \big\{\mathbf{z} \in \RR^m \mid \mathbf{z} \le \mathbf{c} \big\}. \]
This follows from the fact that the rows of  $A_j$ are the nonzero equivalence classes of~\,$\U/J^{(j)}$ of degree  $d$
and hence are a monoid basis of $(\RR J^{(j)}_d)^\perp \cap \RR_{\ge 0}^m$.

Let $\mathbf{g} = A^T\mathbf{y}$. We compute using $A \mathfrak{J}_{\lambda}(\mathbf{u}) = \mathbf{b}$,
\begin{align*}
\langle \mathfrak{J}_{\lambda}(\mathbf{u}), \mathbf{g} \rangle = \mathbf{g}^T \mathfrak{J}_{\lambda}(\mathbf{u}) = \mathbf{y}^T A \mathfrak{J}_{\lambda}(\mathbf{u}) = \mathbf{y}^T \mathbf{b}.
\end{align*}
We conclude that
\begin{align}
\langle \mathfrak{J}_{\lambda}(\mathbf{u}), \mathbf{c}-\mathbf{g} \rangle = M-\mathbf{b}^T \mathbf{y} \ge 0 \iff \mathbf{b}^T \mathbf{y} \le M.  \label{e lp iv v}
\end{align}
Hence by \eqref{e lp iv v} and the claim above, every solution $\mathbf{y}$ of  $\lp_\cap^\vee$ has $\mathbf{b}^T \mathbf{y}  \le M$ if and only if (v) holds.
Moreover, if we let $\mathbf{y}$ be the vector with a 1 in its first  $t_1$ coordinates and 0's elsewhere, then  $\mathbf{g} = A^T \mathbf{y} = \mathbf{c}$ and $\mathbf{b}^T \mathbf{y} = M$, so  $\lp_\cap^\vee$ has optimal value  $\ge M$.
Hence (iv) is equivalent to the statement that $\lp_\cap^\vee$ has optimal value  $\le M$, and we just showed this to be equivalent to (v).
\end{proof}

\section{A D~graph whose generating function is not Schur positive}
\label{s A D graph whose generating function is not Schur positive}
For this section, we set   $N=8$ so that $\U = \ZZ\langle u_1,\ldots,u_8\rangle$.
Let  $\lltlp$ be the linear program $\lp_{\cap}^\vee$ from Theorem~\ref{t linear program 2} with  $d=8$, $\lambda = (2,2,2,2)$,
and the following bijectivizations
of~\,$\U/\Irkst$\,:
\[ \U/J^{(k)} =  \U/\Ilamst{k} \ \ \text{for $k = 1,2, \dots 8$,} \quad \text{and }\, \U/J^{(9)} = \U/\Iplacst. \]
The bijectivizations $\U/\Ilamst{k}$ and $\U/\Iplacst$ are defined in \textsection\ref{ss Lams algebra} and Example \ref{ex plac}, respectively.

We will see that an optimal integer solution of  $\lltlp$ yields a  D$_0$~graph whose generating function is not Schur positive.
We go on to construct a D~graph on the same vertex set as this D$_0$~graph.

\subsection{The linear program  $\lltlp$}
\label{ss linear program lltlp}
Here is some specific data about the linear program  $\lltlp$.
Since we are optimizing  $\mathbf{b}^T \mathbf{y}$, we may delete all entries of  $\mathbf{b}$ that are 0 and delete the corresponding rows of $A$.
Also, since all of the columns of  $A$ and entries of  $\mathbf{c}$ corresponding to $\e{w} \notin \W_d$ are zero, we may as well delete these columns and entries.
Let  $A'$, $A'_i$, $\mathbf{b}'$,  $\mathbf{c}'$,  etc. denote the matrices and vectors associated to this smaller linear program obtained by removing zeros, and let this linear program be  $\lltlp'$.
Note that  $\mathbf{c}'$ is the all-ones vector.
Also note that  $A'$ has columns indexed by  $\W_d = \S_8$;
assume we have fixed some ordering $\e{v^1}, \ldots, \e{v^{8!}}$ of $\S_8$ which labels these columns.

To give an idea of the size of this linear program, here are the  $\mathbf{b}^{\prime i}$\,:
{\small \[\begin{array}{l}
(\mathbf{b}^{\prime1})^T = \mat{ 1& 1& 1& 1& 1& 1& 1& 1& 1& 1& 1& 1& 1& 1 }\\[1mm]
(\mathbf{b}^{\prime2})^T = \mat{ 1& 1& 1& 1& 1& 1& 1& 1& 1& 1& 1& 1& 1& 1 }\\[1mm]
(\mathbf{b}^{\prime3})^T = \mat{ 1& 1& 1& 1& 1& 1& 1& 1& 1& 1& 1& 1& 1& 1 }\\[1mm]
(\mathbf{b}^{\prime4})^T = \mat{ 1& 1& 1& 1& 1& 1& 1& 1& 1& 1& 1& 1& 1& 1 }\\[1mm]
(\mathbf{b}^{\prime5})^T = \mat{ 1& 1& 1& 1& 1& 1& 1& 1& 1& 1& 1& 1& 1& 1 }\\[1mm]
(\mathbf{b}^{\prime6})^T = \mat{ 1& 1& 2& 1& 1& 1& 1& 1& 1& 1& 1& 1& 1 }\\[1mm]
(\mathbf{b}^{\prime7})^T = \mat{ 1& 1& 1& 1& 1& 1& 1& 1& 2& 1& 1& 1& 1 }\\[1mm]
(\mathbf{b}^{\prime8})^T = \mat{ 1& 1& 1& 2& 1& 2& 1& 1& 2& 1& 1 }\\[1mm]
(\mathbf{b}^{\prime9})^T  = \mat{ 1& 1& 1& 1& 1& 1& 1& 1& 1& 1& 1& 1& 1& 1 }
\end{array}\] }

The rows of  $A'_9$  correspond to the Knuth equivalence classes of shape  $(2,2,2,2)$ contained in $\S_8$.
More precisely, the rows of  $A'_9$ are in bijection with $\SYT(2,2,2,2)$ so that the row labeled by $T \in \SYT(2,2,2,2)$ is the zero-one vector of length $8!$ whose support
is $\{i \in \{1,2,\dots,8!\} \mid P(\e{v^i}) = T\}$.
Here, $\SYT(2,2,2,2)$ denotes the set of standard Young tableaux of shape $(2,2,2,2)$.
For  $k \le 8$, the rows of  $A'_k$ correspond to certain equivalence classes from Proposition \ref{p Ilamst equiv classes}.

\subsection{Solving  $\lltlp$}

A \emph{weighted packing linear program} is a linear program of the form:
maximize $\mathbf{b}^T \mathbf{y}$ over the vectors $\mathbf{y} \in \RR^t$ satisfying
\begin{align*}
A^T \mathbf{y} & \le \mathbf{1},\\
\mathbf{y} & \ge \mathbf{0},
\end{align*}
where $A$ has entries in  $\{0,1\}$ and has nonzero rows, and $\mathbf{b}$ has entries in  $\NN$.
Let  $S_i$ denote the support of the  $i$-th row of $A$. The  integer solutions of this linear program correspond to packings of the $S_i$, i.e., collections of $S_i$ such that no two intersect. The
optimization problem is to find a packing of maximum weight, where the weight of a packing  $\{S_i\}_{i \in I}$ ($I \subseteq [t]$) is  $\sum_{i \in I} b_i$.

Any weighted packing linear program  $\lp_\text{pack}$ can be converted into a
maximum independent set problem as follows:
let  $A, \mathbf{b}, S_i$ be the data associated to $\lp_\text{pack}$ as above.
Define $G(A, \mathbf{b})$ to be the graph with
\begin{itemize}
\item vertex set  $V := \bigsqcup_{i=1}^t V_i$, where  $V_i$ is a set of size $b_i$,
\item edges between all vertices in  $V_i$ and all vertices in  $V_j$ if $S_i \cap S_j \ne \varnothing$ and  $i \ne j$,
\item no edge between any vertex in  $V_i$ and any vertex in  $V_j$ if $S_i \cap S_j = \varnothing$,
\item no edge between any two vertices of  $V_i$.
\end{itemize}
Define a map $\Theta$ by
\[\displaystyle \Theta: \{0,1\}^t \to \text{ subsets of  $V$,} \quad \mathbf{y} \mapsto \bigcup_{\{i \in [t] \,\mid\, y_i = 1\}} V_i. \]

\begin{proposition}\label{p packing}
Maintain  the notation above.
The map $\Theta$ yields a bijection between integer solutions of  $\lp_\text{pack}$  and independent sets of  $G(A, \mathbf{b})$ that intersect each  $V_i$ in all of  $V_i$ or not at all.
Moreover, $\mathbf{y}^*$ is an optimal integer solution of  $\lp_\text{pack}$ if and only if $\Theta(\mathbf{y}^*)$ is a maximum independent set of  $G(A, \mathbf{b})$.
\end{proposition}
\begin{proof}
First note that any integer solution of $\lp_\text{pack}$ is a zero-one vector.
The first statement then follows from the following chain of equivalences. For any $\mathbf{y} \in \{0,1\}^t$,
\begin{align*}
     & A^T \mathbf{y} \le \mathbf{1}  \\
\iff & \text{$S_i \cap S_j = \varnothing$ for each pair $i, j$ of distinct indices in the support of $\mathbf{y}$}\\
\iff & \text{there is no edge between $\Theta(\mathbf{y}) \cap V_i$ and $\Theta(\mathbf{y}) \cap V_j$ for all $i,j$.
}
\end{align*}
The second statement then follows from  $\mathbf{b}^T\mathbf{y} = |\Theta(\mathbf{y})|$.
\end{proof}

For the linear program $\lltlp'$,  $\mathbf{b}' \in \NN^{121}$ and  $G(A', \mathbf{b}')$ has  $126$ vertices.  The maximum size of an independent set is 15.
This computation takes 1-2 seconds on a machine with a 32-core,
2.7 GHz processor and 25 GB of RAM
using the branch and bound algorithm \cite{BronKerbosch} built into Magma \cite{Magma}.
Hence by Proposition \ref{p packing}, $\lltlp'$ has optimal integer value 15 and its optimal integer solutions are in bijection with
independent sets of  $G(A',\mathbf{b}')$ of size 15.

\subsection{An optimal integral solution of  $\lltlp$}
\label{ss solution to lltlp}
There are many optimal integer solutions of $\lltlp'$.
We now describe one such solution  $\mathbf{y}^*$ in detail and its consequences for Schur positivity of  D$_0$~graphs.

Recall that the rows of  $A'$ and rows of  $\mathbf{y}^*$ are naturally labeled by certain nonzero equivalence classes of  the $\U/J^{(j)}$.
The solution $\mathbf{y}^*$ is  a zero-one vector whose support corresponds to the following equivalence classes:
13 of the 14 Knuth equivalence classes of shape $(2,2,2,2)$,
an equivalence class of~\,$\U/\Ilamst{3}$, and an equivalence class of~\,$\U/\Ilamst{7}$, with precise descriptions given by
\eqref{e 13equivclasses}, \eqref{e equivclassk3}, and \eqref{e equivclassk7}, respectively.
\begin{align}
 &\left\{\e{w} \in \S_8 \mid P(\e{w}) \text{ has shape } (2,2,2,2), \, P(\e{w}) \ne {\tiny \tableau{1&2\\3&4\\5&6\\7&8} }\right\}  &&\text{ (size $13 \cdot 14$)}, \label{e 13equivclasses}\\
 &\ \{\e{w} \in \S_8 \mid \e{w} \in \Wi{3}(\bm{\beta}_3), \, \invi{3}(\e{w}) = 9\}  &&\text{ (size 126)}, \label{e equivclassk3}\\
 &\ \{\e{w} \in \S_8\mid \e{w} \in \Wi{7}(\bm{\beta}_7), \, \invi{7}(\e{w}) = 18\}  &&\text{ (size 316)}, \label{e equivclassk7}
\end{align}
where $\bm{\beta}_3$ (resp. $\bm{\beta_7}$) is the 3-tuple (resp. 7-tuple) of skew shapes with contents given by
\begin{align*}
\bm{\beta}_3 &= (33,333,333) / (22, 222, 222). && \text{Its shifted contents are }
{\tiny\left(\tableau{& & 6\\& & 3 \\ & & },\tableau{&&7 \\& & 4 \\ & &1},\tableau{ &&8\\ & &5 \\ & &2 }\right)}.\\
{\tiny \bm{\beta}_7 }&{\tiny = (2, 2, 1, 1, 1, 1, 1) / (1, \varnothing, \varnothing, \varnothing, \varnothing, \varnothing, \varnothing). }&& \text{Its shifted contents are }
{\tiny\left(\tableau{& 7},\tableau{1 & 8},\tableau{2}, \tableau{3}, \tableau{4},\tableau{5}, \tableau{6}\right)}.
\end{align*}
For example, the equivalence class \eqref{e equivclassk3} contains $\e{78634521}$ and the equivalence class \eqref{e equivclassk7} contains $\e{75183642}$.

Let $W^* \subseteq \S_8$ be the union of the equivalence classes \eqref{e 13equivclasses}, \eqref{e equivclassk3}, and \eqref{e equivclassk7}.
Set $\bar{W}^* = \S_8 \setminus W^*$.
Let $\mathbf{g}^* = {A'}^T \mathbf{y}^*$, the zero-one vector (in the basis $\e{v^1}, \ldots, \e{v^{8!}}$) whose support is $W^*$,
which may be identified with the element $\sum_{\e{w} \in  W^*} \e{w} \in \U_8$.
Let  $f^* = \mathbf{c}' - \mathbf{g}^* = \sum_{\e{w} \in \bar{W}^*} \e{w} \in \U_8$.
Note that $f^*, \mathbf{g}^*\in J^\perp$ and since the $\U/J^{(j)}$ are bijectivizations of~\,$\U/\Irkst,$ $J^\perp \subseteq (\Irkst)^\perp$.
Hence $W^*$ and $\bar{W}^*$ are KR sets.

For the linear program  $\lltlp$, the  $M$ from Theorem~\ref{t linear program 2} is given by
$M = \mathbf{c}^T \mathfrak{J}_\lambda(\mathbf{u}) = |\SYT(2,2,2,2)| = 14$.
Since  $\lltlp$ has optimal value  $\ge 15$, (iv) from Theorem~\ref{t linear program 2} is false,
hence (i)--(v) are all false.
This has two important consequences, which we summarize in the corollaries below.
A further consequence will be discussed in Section \ref{s Concluding remarks}.

\begin{corollary}\label{c D0 not schur pos}
The generating function $\Delta(f^*)$ of $f^*$ is not Schur positive.
Moreover, by Proposition-Definition \ref{pd KR set},
there is a D$_0$~graph on the vertex set $\bar{W}^*$ whose generating function is not Schur positive.
\end{corollary}
To see in more detail why $\Delta(f^*)$ is not Schur positive, we compute as in the proof of Theorem~\ref{t linear program 2}:
\begin{align*}
\langle \mathfrak{J}_{(2,2,2,2)}(\mathbf{u}), f^* \rangle = M - \langle \mathfrak{J}_{(2,2,2,2)}(\mathbf{u}), \mathbf{g}^* \rangle =M-\mathbf{y}^{* T} \mathbf{b}' = M-15 = -1.
\end{align*}
Hence the coefficient of  $s_{(2,2,2,2)}(\mathbf{x})$ in  $\Delta(f^*)$ is  $-1$.
It may also be instructive to see the Schur expansions of $\Delta(\mathbf{g}^*)$ and $\Delta(f^*)$ in full:
\[\Delta(\mathbf{g}^*) =
s_{41111} + s_{3221} + 3s_{32111} + s_{311111} + 15s_{2222} + 2s_{22211} + 2s_{221111}.
\]
\begin{align*}
&\Delta(f^*) = \Delta\bigg(\sum_{\e{w} \in \S_8} \e{w} - \mathbf{g}^*\bigg)  =
\sum_{|\mu| = 8} |\SYT(\mu)| s_\mu(\mathbf{x}) - \Delta(\mathbf{g}^*)  =\\
&s_{8} + 7s_{71} + 20s_{62} + 21s_{611} + 28s_{53} + 64s_{521} + 35s_{5111} + 14s_{44} + 70s_{431}  + \\
&56s_{422} + 90s_{4211} +
34s_{41111} +
42s_{332} + 56s_{3311} + 69s_{3221} + 61s_{32111} + \\
&20s_{311111} - s_{2222} + 26s_{22211} + 18s_{221111} + 7s_{2111111} + s_{11111111}.
\end{align*}
%

\begin{corollary}
There is no way to write $\mathfrak{J}_{(2,2,2,2)}(\mathbf{u})$ as a positive sum of monomials (even over  $\RR$) that holds simultaneously in  $\U/\Ilamst{k}$,  $k \in [8]$, and  $\U/\Iplacst$.
Hence (by Proposition \ref{p lam and assaf}) there is also no way write $\mathfrak{J}_{(2,2,2,2)}(\mathbf{u})$ as a positive sum of monomials over  $\RR$ that holds simultaneously in
all the $\U/\Iassafst{k}$ for $k \in [8]$.
\end{corollary}

\subsection{Forcing axiom 5}
\label{ss forcing axiom 5}
In this subsection we construct a D$_5$ graph (a  D$_0$~graph satisfying axiom 5) whose generating function is
not Schur positive.
The basic idea is to
start with a partial  D$_0$~graph on the vertex set $\bar{W}^*$ from \textsection\ref{ss solution to lltlp} and repeatedly
add edges that are forced by axiom 5.
To fully understand this construction, we first need the following functions which add edges to force axiom 5.

Recall from \textsection\ref{ss D0 graphs} that the type of a KR square of a  partial  D$_0$~graph
is $\varnothing$, 0, K, or R.
Note that a partial D$_0$~graph $H$ is a D$_0$~graph if and only if the type of any KR square of $H$ is not 0.

\begin{myalgorithm}
\label{a forceaxiom5}
The functions ForceAxiom5 and ForceAxiom5Step below take as input a partial D$_0$~graph $H$ and output a partial D$_0$~graph containing $H$ or FAILURE. \;

\begin{algorithm}[H]
\DontPrintSemicolon
\SetKwProg{Fn}{function}{}{end function}
\SetAlgoNoLine
\Fn{ForceAxiom5Step$(H)$}{
    For every  $|j-i| \ge 3$ and  $i$-edge  $\{\e{v},\e{w}\}$ of $H$ such that  $\e{v}$ admits a  $j$-neighbor, ensure that the type  $t_\e{v}^j$ of the  KR$_j$ square containing $\e{v}$ is the same as the type  $t_\e{w}^j$ of the  KR$_j$ square containing $\e{w}$.  This means \;
\begin{itemize}
\item if  $t_\e{v}^j = t_\e{w}^j$, do nothing
\item if  $t_\e{v}^j \neq t_\e{w}^j$ and both types $t_\e{v}^j, t_\e{w}^j$ are not 0, then return FAILURE
\item if  $t_\e{v}^j \neq 0$ and  $t_\e{w}^j = 0$, set  $t_\e{w}^j := t_\e{v}^j$ and add the two corresponding edges to $H$
\item if  $t_\e{w}^j \neq 0$ and  $t_\e{v}^j = 0$, set  $t_\e{v}^j := t_\e{w}^j$ and add the two corresponding edges to $H.$
\end{itemize}
Return the resulting graph. \;
}
\end{algorithm}
\begin{algorithm}[H]
\DontPrintSemicolon
\SetKwProg{Fn}{function}{}{end function}
\SetAlgoNoLine
\Fn{ForceAxiom5$(H)$}{
    \Repeat{$H' = H$ or  $H=$ FAILURE}{
        $H' := H$\;
        $H$ := ForceAxiom5Step$(H)$\;
    }
    \Return $H$ \;
}
\end{algorithm}
\end{myalgorithm}

Note that ForceAxiom5Step uses description (iii) of axiom 5 from Proposition \ref{p axiom5 patterns} below.  Since
ForceAxiom5Step checks that there are no violations to axiom 5 and only adds edges that are forced by axiom 5, the following proposition is immediate.
\begin{proposition}
If ForceAxiom5 outputs a  D$_0$~graph, then this graph satisfies axiom 5.
If ForceAxiom5$(H)$ outputs FAILURE, then there is no  D$_5$ graph containing $H$.
\end{proposition}

\begin{myalgorithm} \label{a grow D5}
The following recursive function takes as input a partial  D$_0$~graph  $H$
and outputs either a D$_5$ graph containing  $H$ or FAILURE.

\begin{algorithm}[H]
\DontPrintSemicolon
\SetKwProg{Fn}{function}{}{end function}
\SetAlgoNoLine
\Fn{GrowD5Graph$(H)$}{
    \If{ there are no undetermined KR squares of $H$}{
        \Return  $H$ \;
    }
    Choose an undetermined KR square $X$ of $H$. Declare  the type of  $X$ to be Knuth, add the two corresponding edges to $H$, and let  $H'$ denote the resulting graph. \;
    $H''$ := ForceAxiom5$(H')$ \;
    \If{$H''$ = FAILURE}{
        \Return FAILURE \;
    }
    \Return GrowD5Graph$(H'')$ \;
}
\end{algorithm}
\end{myalgorithm}

\begin{definition}\label{d KR set to D0}
Given a KR set $W$, the \emph{minimal partial  D$_0$~graph on  $W$} is the partial  D$_0$~graph  $H$ on the vertex set  $W$ having
as many undetermined KR squares as possible.
Equivalently,  $H$ is the graph on the vertex set  $W$ such that  $\{\e{v},\e{w}\}$ is an $i$-edge if and only if
 $\{\e{v},\e{w}\} = X \cap W$ for some KR$_i$ square $X$.
\end{definition}

Although Algorithm \ref{a grow D5} comes nowhere close to searching the entire space of  D$_0$~graphs on a vertex set for one that satisfies axiom 5,
it is good enough to find a D$_5$ graph on the vertex set  $\bar{W}^*$ from \textsection\ref{ss solution to lltlp}.
Specifically, let  $H^*$ denote the minimal partial  D$_0$~graph on  $\bar{W}^*$.
Order the KR squares of $\S_8$ as follows:
\[\parbox{13cm}{associate to each KR$_i$ square $\e{v} \stars \e{w}$ the string $ivbacw$
and then order these strings lexicographically.}\]
If the smallest undetermined KR square of $H$ is always chosen in GrowD5Graph,
then GrowD5Graph$(H^*)$ returns a  D$_5$ graph on $\bar{W}^*$.
By Corollary \ref{c D0 not schur pos}, its generating function is not Schur positive.
The largest component of this graph, which has 5322 vertices, is also a  D$_5$ graph whose generating function is not Schur positive.

\subsection{Reformulations of axiom 5}
Here we reformulate axiom 5 in terms of KR squares and give a slightly smarter version of the function ForceAxiom5.

\begin{definition}
For a partial D$_0$~graph $H$ of degree  $n$, and indices $|j-i| \geq 3$,
a \emph{KR$_{i,j}$ hypercube}  $Z$ of $H$
consists of a matrix of KR squares of $H$ of the form
\begin{equation}
\label{e hypercube}
 \begin{array}{cc|cc}
\e{x\stars yedfz} & \e{x\stars yefdz} & \e{xbacy\starsdef z} & \e{xbcay\starsdef z }   \\
\e{x\stars ydfez }& \e{x\stars yfdez} & \e{xacby\starsdef z} & \e{xcaby\starsdef z}
\end{array}
\end{equation}
together with their types, called the \emph{pattern} of  $Z$, notated as
\begin{equation}
\label{e pattern}
t =
\matdnobrac{
t_1^i & t_2^i & t_1^j & t_2^j    \\
t_3^i & t_4^i & t_3^j & t_4^j
},
\end{equation}
where $a < b < c$ and $d<e<f$ are letters and $\e{x},\e{y},\e{z}$ are words. Also, $i = |\e{x}|+2$,  $j = |\e{x}|+|\e{y}|+5$ so that the four KR squares on the left (resp. right) of \eqref{e hypercube} are KR$_i$ (resp. KR$_j$) squares.
\end{definition}
Note that the pattern of any KR$_{i,j}$ hypercube of a D$_0$~graph only contains the types $\varnothing$, K, and R.

\begin{example}
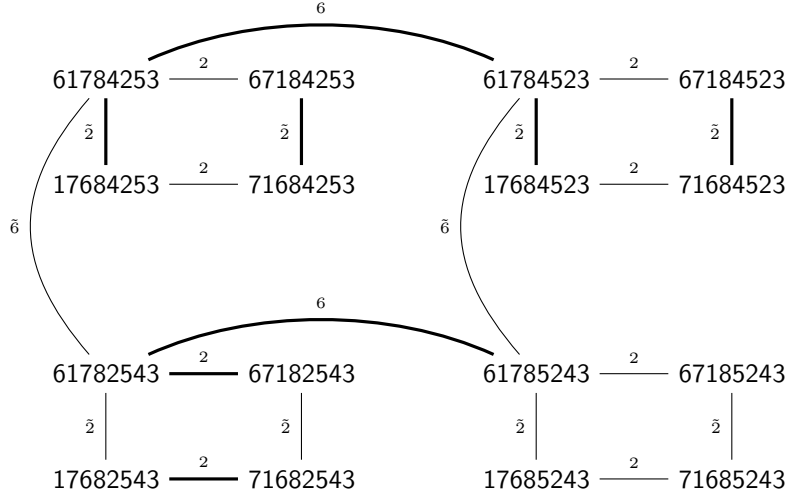
\begin{figure}
 \begin{tikzpicture}[xscale = 2.6,yscale = 1.4]
\tikzstyle{vertex}=[inner sep=0pt, outer sep=4pt]
\tikzstyle{aedge} = [draw, thin, ->,black]
\tikzstyle{edge} = [draw, -,black]
\tikzstyle{thickedge} = [draw, very thick, black]
\tikzstyle{LabelStyleH} = [text=black, anchor=south]
\tikzstyle{LabelStyleV} = [text=black, anchor=east]
\foreach \z in {2.8} {
\foreach \y in {2.2} {
\node[vertex] (v1) at (1,2+\z){\footnotesize$\e{61784253 }$};
\node[vertex] (v2) at (2,2+\z){\footnotesize$\e{67184253 }$};
\node[vertex] (v3) at (1,1+\z){\footnotesize$\e{17684253 }$};
\node[vertex] (v4) at (2,1+\z){\footnotesize$\e{71684253 }$};
\node[vertex] (v5) at (1+\y,2+\z){\footnotesize$\e{61784523 }$};
\node[vertex] (v6) at (2+\y,2+\z){\footnotesize$\e{67184523 }$};
\node[vertex] (v7) at (1+\y,1+\z){\footnotesize$\e{17684523 }$};
\node[vertex] (v8) at (2+\y,1+\z){\footnotesize$\e{71684523 }$};
\node[vertex] (v9) at (1,2){\footnotesize$\e{61782543 }$};
\node[vertex] (v10) at (2,2){\footnotesize$\e{67182543 }$};
\node[vertex] (v11) at (1,1){\footnotesize$\e{17682543 }$};
\node[vertex] (v12) at (2,1){\footnotesize$\e{71682543 }$};
\node[vertex] (v13) at (1+\y,2){\footnotesize$\e{61785243 }$};
\node[vertex] (v14) at (2+\y,2){\footnotesize$\e{67185243 }$};
\node[vertex] (v15) at (1+\y,1){\footnotesize$\e{17685243 }$};
\node[vertex] (v16) at (2+\y,1){\footnotesize$\e{71685243 }$};
\draw[edge] (v1) to node[LabelStyleH]{\Tiny$2 $} (v2);
\draw[edge] (v3) to node[LabelStyleH]{\Tiny$2 $} (v4);
\draw[thickedge] (v1) to node[LabelStyleV]{\Tiny$\tilde{2} $} (v3);
\draw[thickedge] (v2) to node[LabelStyleV]{\Tiny$\tilde{2} $} (v4);
\draw[edge] (v5) to node[LabelStyleH]{\Tiny$2 $} (v6);
\draw[edge] (v7) to node[LabelStyleH]{\Tiny$2 $} (v8);
\draw[thickedge] (v5) to node[LabelStyleV]{\Tiny$\tilde{2} $} (v7);
\draw[thickedge] (v6) to node[LabelStyleV]{\Tiny$\tilde{2} $} (v8);
\draw[thickedge] (v9) to node[LabelStyleH]{\Tiny$2 $} (v10);
\draw[thickedge] (v11) to node[LabelStyleH]{\Tiny$2 $} (v12);
\draw[edge] (v9) to node[LabelStyleV]{\Tiny$\tilde{2} $} (v11);
\draw[edge] (v10) to node[LabelStyleV]{\Tiny$\tilde{2} $} (v12);
\draw[edge] (v13) to node[LabelStyleH]{\Tiny$2 $} (v14);
\draw[edge] (v15) to node[LabelStyleH]{\Tiny$2 $} (v16);
\draw[edge] (v13) to node[LabelStyleV]{\Tiny$\tilde{2} $} (v15);
\draw[edge] (v14) to node[LabelStyleV]{\Tiny$\tilde{2} $} (v16);
\draw[thickedge, bend left=40] (v1) to node[LabelStyleH]{\Tiny$6 $} (v5);
\draw[thickedge, bend left=40] (v9) to node[LabelStyleH]{\Tiny$6 $} (v13);
\draw[edge, bend left=-25] (v1) to node[LabelStyleV]{\Tiny$\tilde{6} $} (v9);
\draw[edge, bend left=-25] (v5) to node[LabelStyleV]{\Tiny$\tilde{6} $} (v13);
}
}
\end{tikzpicture}
\caption{\label{f pattern}The vertices and some of the edges potentially involved in a KR$_{2,6}$ hypercube.}
\end{figure}

If  $H$ is a partial  D$_0$~graph containing all the vertices in Figure \ref{f pattern} and only the bold edges, then
\[
\text{the pattern of }\,
\begin{array}{cc|cc}
\e{[167]84253} & \e{[167]84523} & \e{6178[245]3} & \e{6718[245]3}    \\
\e{[167]82543} & \e{[167]85243} & \e{1768[245]3} & \e{7168[245]3}
\end{array}
\text{\, is \ \, }
\matdnobrac{
\text{R} & \text{R} & \text{K} & 0    \\
\text{K} & 0        & 0 & 0
}.
\]
For a partial  D$_0$~graph containing only the eight vertices on the left of the figure, and the four bold edges between these vertices together with the four rotation $6$-edges between these vertices, the pattern is
\ $\matdnobrac{
\text{R} & \varnothing & \text{R} & \text{R}    \\
\text{K} & \varnothing & \text{R} & \text{R}
}.$
\end{example}

A \emph{partial pattern} is a $2\times 4$-matrix as in \eqref{e pattern} whose entries are either unoccupied or one of the types $\varnothing$, 0, K, or R.
A pattern $t$ \emph{contains} a partial pattern $\hat{t}$ if $t$ and $\hat{t}$ agree on the occupied entries of $\hat{t}$.

If $T$ and $T'$ are sets of $2\times 2$ matrices, then let $T\dot{\times} T'$ denote the set of $2\times 4$ matrices
\[\{
\left[\begin{array}{@{}c|c@{}}
t & t'\end{array}\right]
 \mid t \in T, ~ t' \in T'\}. \]
For example, if
\[
T = \left\{
\matb{
\text{K} &      \\
\text{R} &
},
\matb{
 & \text{K}     \\
 & \text{R}
}\right\}
, \quad T' = \left\{
\matb{
\text{R} &      \\
  &
}
,
\matb{
 & \text{R}     \\
 &
}\right\}
,
\text{ then }
\]
\[T \dot{\times} T' = \left\{
\matd{
\text{K} &  & \text{R}&     \\
\text{R} &  &  &
},
\matd{
\text{K} &  &  &  \text{R}   \\
\text{R} &  &  &
},
\matd{
 & \text{K} & \text{R} &     \\
 & \text{R} & &
},
\matd{
 & \text{K}&  & \text{R}    \\
 & \text{R}&  &
}
\right\}.\]
It is convenient to reformulate axiom 5 in terms of KR square types, with no direct mention of edges.
\begin{proposition} \label{p axiom5 patterns}
The following properties of a D$_0$~graph $\G$ are equivalent:
\begin{list}{\emph{(\roman{ctr})}}{\usecounter{ctr} \setlength{\itemsep}{2pt} \setlength{\topsep}{3pt}}
\item $\G$ satisfies axiom 5.
\item For every  $|j-i| \geq 3$ and $i$-edge $\{\e{v},\e{w}\}$ such that $\e{v}$ admits a $j$-neighbor, the type of the  $j$-edge at $\e{v}$ is the same as the type of the $j$-edge at $\e{w}$.
\item For every  $|j-i| \geq 3$ and $i$-edge $\{\e{v},\e{w}\}$ such that $\e{v}$ admits a $j$-neighbor, the type of the KR$_j$ square containing $\e{v}$ is the same as the type of the KR$_j$ square containing $\e{w}$.
\item For every $j-i \geq 3$, the pattern of any KR$_{i,j}$ hypercube of $\G$ does not contain any of the partial patterns in the following set of size 64:
\begin{align}
 &
\left\{
\matb{
\text{\emph K} &      \\
  &
},
\matb{
  & \text{\emph K}     \\
  &
},
\matb{
  &      \\
\text{\emph K} &
},
\matb{
  &       \\
  & \text{\emph K}
}
\right\}
\dot{\times}
\left\{
\matb{
\text{\emph K} & \text{\emph R}      \\
  &
},
\matb{
  &      \\
\text{\emph K} & \text{\emph R}
},
\matb{
\text{\emph R} & \text{\emph K}     \\
  &
},
\matb{
  &       \\
\text{\emph R} & \text{\emph K}
}
\right\}  \label{e partial1}\\
 \bigsqcup
& \left\{
\matb{
\text{\emph R} &      \\
  &
},
\matb{
  & \text{\emph R}     \\
  &
},
\matb{
  &      \\
\text{\emph R} &
},
\matb{
  &       \\
  & \text{\emph R}
}
\right\} \dot{\times}
\left\{
\matb{
\text{\emph K} &      \\
\text{\emph R} &
},
\matb{
  & \text{\emph K}     \\
  & \text{\emph R}
},
\matb{
\text{\emph R} &      \\
\text{\emph K} &
},
\matb{
  & \text{\emph R}     \\
  & \text{\emph K}
}
\right\}
\label{e partial2}\\
 \bigsqcup
&\left\{
\matb{
\text{\emph K} & \text{\emph R}      \\
  &
},
\matb{
  &      \\
\text{\emph K} & \text{\emph R}
},
\matb{
\text{\emph R} & \text{\emph K}     \\
  &
},
\matb{
  &       \\
\text{\emph R} & \text{\emph K}
}
\right\} \dot{\times}
\left\{
\matb{
\text{\emph K} &      \\
  &
},
\matb{
  & \text{\emph K}     \\
  &
},
\matb{
  &      \\
\text{\emph K} &
},
\matb{
  &       \\
  & \text{\emph K}
}
\right\}
\label{e partial3}\\
\bigsqcup &\left\{
\matb{
\text{\emph K} &      \\
\text{\emph R} &
},
\matb{
  & \text{\emph K}     \\
  & \text{\emph R}
},
\matb{
\text{\emph R} &      \\
\text{\emph K} &
},
\matb{
  & \text{\emph R}     \\
  & \text{\emph K}
}
\right\} \dot{\times}
\left\{
\matb{
\text{\emph R} &      \\
  &
},
\matb{
  & \text{\emph R}     \\
  &
},
\matb{
  &      \\
\text{\emph R} &
},
\matb{
  &       \\
  & \text{\emph R}
}
\right\}.\label{e partial4}
\end{align}
\item For every $j- i \geq 3$, the pattern of any KR$_{i,j}$ hypercube of $\G$ appears below or is obtained from one of these by replacing some of the types K or R by $\varnothing$.
\begin{align*}
&\matdnobrac{
\text{\emph K} & \text{\emph K} & \text{\emph K} & \text{\emph K}    \\
\text{\emph K} & \text{\emph K} & \text{\emph K} & \text{\emph K}
} \quad \ \
\matdnobrac{
\text{\emph K} & \text{\emph K} & \text{\emph K} & \text{\emph K}    \\
\text{\emph K} & \text{\emph K} & \text{\emph R} & \text{\emph R}
}\quad \ \
\matdnobrac{
\text{\emph K} & \text{\emph K} & \text{\emph R} & \text{\emph R}    \\
\text{\emph K} & \text{\emph K} & \text{\emph K} & \text{\emph K}
}\quad \ \
\matdnobrac{
\text{\emph K} & \text{\emph K} & \text{\emph R} & \text{\emph R}    \\
\text{\emph K} & \text{\emph K} & \text{\emph R} & \text{\emph R}
}\\[2mm]
&\matdnobrac{
\text{\emph K} & \text{\emph K} & \text{\emph K} & \text{\emph K}    \\
\text{\emph R} & \text{\emph R} & \text{\emph K} & \text{\emph K}
}\quad \ \
\matdnobrac{
\text{\emph R} & \text{\emph R} & \text{\emph K} & \text{\emph K}    \\
\text{\emph K} & \text{\emph K} & \text{\emph K} & \text{\emph K}
}\quad \ \
\matdnobrac{
\text{\emph K} & \text{\emph R} & \text{\emph R} & \text{\emph R}    \\
\text{\emph K} & \text{\emph R} & \text{\emph R} & \text{\emph R}
} \quad \ \
\matdnobrac{
\text{\emph R} & \text{\emph K} & \text{\emph R} & \text{\emph R}    \\
\text{\emph R} & \text{\emph K} & \text{\emph R} & \text{\emph R}
}\\[2mm]
&\matdnobrac{
\text{\emph R} & \text{\emph R} & \text{\emph K} & \text{\emph K}   \\
\text{\emph R} & \text{\emph R} & \text{\emph K} & \text{\emph K}
}\quad \ \
\matdnobrac{
\text{\emph R} & \text{\emph R} & \text{\emph K} & \text{\emph R}    \\
\text{\emph R} & \text{\emph R} & \text{\emph K} & \text{\emph R}
} \quad \ \
\matdnobrac{
\text{\emph R} & \text{\emph R} & \text{\emph R} & \text{\emph K}    \\
\text{\emph R} & \text{\emph R} & \text{\emph R} & \text{\emph K}
} \quad \ \
\matdnobrac{
\text{\emph R} & \text{\emph R} & \text{\emph R} & \text{\emph R}    \\
\text{\emph R} & \text{\emph R} & \text{\emph R} & \text{\emph R}
}
\end{align*}
\end{list}
\end{proposition}

\begin{proof}
The equivalence of (i), (ii), and (iii) follows from the fact that in a D$_0$~graph, the type of an $i$-edge is equal to the type of the KR$_i$ square containing it.

To prove (iv) implies (iii), suppose  $\{\e{v},\e{w}\}$ is an  $i$-edge such that $\e{v}$ admits a $j$-neighbor.
Assume  $j > i$, the case  $j < i$ being similar.
Let
\[ t =
\matdnobrac{
t_1^i & t_2^i & t_1^j & t_2^j    \\
t_3^i & t_4^i & t_3^j & t_4^j
}
\]
be the pattern of the KR$_{i,j}$ hypercube of $\G$ that contains the KR$_i$ square containing $\e{v}$ and $\e{w}$.
Note that the KR$_j$ squares containing $\e{v}$ and $\e{w}$ cannot have type $\varnothing$ or 0.
Hence if $\{\e{v},\e{w}\}$ has type K (resp. R), then  $t$ not containing any of the partial patterns in \eqref{e partial1} (resp. \eqref{e partial2}) implies $t_1^j = t_2^j$ and $t_3^j = t_4^j$ (resp. $t_1^j = t_3^j$ and $t_2^j = t_4^j$).
It follows that the type of the KR$_j$ square containing $\e{v}$ is the same as the type of the KR$_j$ square containing $\e{w}$.

To help prove (iii) implies (iv), we
first claim that (iii) implies
\begin{equation}
\label{e same t}
\parbox{14cm}{the left $2\times 2$ submatrix and the right $2\times 2$ submatrix in the pattern of any  KR$_{i,j}$ hypercube of $\G$ does not contain
$\matb{\text{R} & \text{K} \\ \text{K} & }$ and $\matb{\text{K} & \text{R} \\ \text{R} &}$ and their rotations.
}
\end{equation}
To prove this claim, consider a KR$_{i,j}$ hypercube of  $\G$ as in \eqref{e hypercube} with pattern  $t$ as in \eqref{e pattern}.  Since $\G$ is a D$_0$~graph,
if $t_1^i \in \{\text{K}, \text{R}\},$ then at least two of the vertices in  $\e{x\stars yedfz}$ are vertices of  $\G$ and admit  $j$-neighbors.
So (iii) implies that at least one of $t_2^i, t_3^i$ is the same as $t_1^i$.  This proves that the the left  $2 \times 2$ submatrix of $t$ does not
contain $\matb{\text{R} & \text{K} \\ \text{K} & }$ and $\matb{\text{K} & \text{R} \\ \text{R} &}$. The other conclusions of the claim are proved in a similar way.

Now assume (iii) holds.
Suppose for a contradiction that (iv) does not hold, i.e.
there is a pattern
\[ t =
\matdnobrac{
t_1^i & t_2^i & t_1^j & t_2^j    \\
t_3^i & t_4^i & t_3^j & t_4^j
}
\]
of a KR$_{i,j}$ hypercube of $\G$ as in \eqref{e hypercube} such that
 $t$ contains one the partial patterns in (iv).
All 64 possibilities are handled in a similar way, so assume
$t_1^i = \text{R}$, $t_2^j = \text{R},$ $t_4^j = \text{K}$.
By \eqref{e same t}, we must have $t_1^j \in \{\text{R}, \varnothing\}$ and $t_3^j \in \{\text{K}, \varnothing\}$.
Since $t_1^i = \text{R}$, at least one of the sets $\{\e{xbacyedfz, xacbyedfz}\}$, $\{\e{xbcayedfz, xcabyedfz}\}$ is an $i$-edge of $\G$.
But by (iii), the former being an $i$-edge implies $t_1^j = t_3^j \in \{\text{K}, \text{R}\}$ and the latter being an $i$-edge implies $t_2^j = t_4^j \in \{\text{K}, \text{R}\}$, contradiction.

To prove that (v) implies (iv), it suffices to check each of the 12 patterns appearing in (v) does not contain any of the partial patterns in (iv), which is straightforward.

We now prove (iv) implies (v). Let $t$ be the pattern of any KR$_{i,j}$ hypercube of $\G$.
Let $t^i$ (resp. $t^j$) be the left (resp. right) $2\times 2$ submatrix of $t$.
Note that (iv) implies \eqref{e same t} since (iv) implies (iii).
It follows that
\begin{equation}
\label{e left right forced}
\parbox{14.5cm}{$t^i$ (resp. $t^j$) belongs to
\[
\matb{\text{K} & \text{K} \\ \text{K} & \text{K} } \ \ \matb{\text{K} & \text{K} \\ \text{R} & \text{R} } \ \ \matb{\text{R} & \text{R} \\ \text{K} & \text{K} } \ \
\matb{\text{K} & \text{R} \\ \text{K} & \text{R} } \ \ \matb{\text{R} & \text{K} \\ \text{R} & \text{K} } \ \ \matb{\text{R} & \text{R} \\ \text{R} & \text{R} },
\]
or is obtained from one of these by replacing some of the types K or R by $\varnothing$.}
\end{equation}
If $t^i$ contains $\matb{\text{K} & \text{R} \\ &}$, then by (iv), line \eqref{e partial3}, $t^j$ equals $\matb{\text{R} & \text{R} \\ \text{R} & \text{R}}$ or is obtained from this by replacing some of the types by $\varnothing$. It follows from \eqref{e left right forced} that
$t = \matd{\text{K} & \text{R} & \text{R} & \text{R} \\ \text{K} & \text{R} & \text{R} & \text{R}}$
or is obtained from this by replacing some of the types by $\varnothing$.
This pattern belongs to the list in (v), as desired.
If $t^i$ contains a rotation of $\matb{\text{K} & \text{R} \\ &}$ or $\matb{\text{R} & \text{K} \\ &}$, then a similar argument applies.
If $t^i$ contains only K or $\varnothing$, then by \eqref{e left right forced} and (iv) (line \eqref{e partial1}), $t$ is equal to one of the patterns in the top row of the list in (v) or is obtained from one of these by replacing some of the types by $\varnothing$.
The case that $t^i$ contains only R or $\varnothing$ is similar.
This exhausts all possibilities for  $t^i$ because by the definition of a D$_0$~graph, if $t$ does not consist entirely of $\varnothing$'s, then $t^i$ contains two entries that are not $\varnothing$ in the same row or column.
\end{proof}

Say that a pattern $t'$ of a KR$_{i,j}$ hypercube  $Z$ \emph{covers} another pattern  $t$ of  $Z$
if  $t'$ is obtained from  $t$ by replacing some of its 0's and $\varnothing$'s  by K or R.

\begin{myalgorithm}
\label{a forceaxiom5 advanced}
The functions below take as input a partial D$_0$~graph $H$ and output a partial D$_0$~graph containing $H$ or FAILURE.\vspace{1mm} \;

\begin{algorithm}[H]
\DontPrintSemicolon
\SetKwProg{Fn}{function}{}{end function}
\SetAlgoNoLine
\Fn{\emph{AdvancedForceAxiom5Step}$(H)$}{
For each $j-i \ge 3$ and  KR$_{i,j}$ hypercube of  $H$ with pattern
\[ t =
\begin{array}{cc|cc}
t_{11} & t_{12} & t_{13} & t_{14}    \\
t_{21} & t_{22} & t_{23} & t_{24}
\end{array},
\]
and for each $t_{rc}$ equal to 0, set $t_{rc} = \text{K}$ (resp.  $t_{rc} = \text{R}$) and add the two corresponding edges to $H$
if each of the 12 patterns of Proposition \ref{p axiom5 patterns} (v) that covers $t$ has $r,c$-entry equal to K (resp. R). Let $H'$ denote the resulting graph. If none of the 12 patterns covers $t$, return FAILURE, otherwise return $H'$.\;
}
\end{algorithm}
\begin{algorithm}[H]
\DontPrintSemicolon
\SetKwProg{Fn}{function}{}{end function}
\SetAlgoNoLine
\Fn{\emph{AdvancedForceAxiom5}$(H)$}{
\Repeat{$H' = H$ or  $H =$ FAILURE}{
\text{$H' := H$}\;
$H$ := AdvancedForceAxiom5Step($H$)\;
}
\Return $H$ \;
}
\end{algorithm}
\end{myalgorithm}

\vspace{-5pt}
To illustrate AdvancedForceAxiom5Step, the pattern
$t = ~
\matdnobrac{
\text{R} & 0 & 0 & 0    \\
\text{K} & 0 & 0 & 0
}~$
can be completed to
$~
\matdnobrac{
\text{R} & \text{R} & \text{K} & \text{K}    \\
\text{K} & \text{K} & \text{K} & \text{K}
}~~
$
since this is the only one of the 12 patterns covering $t$.

The following proposition is immediate from Proposition \ref{p axiom5 patterns}.
\begin{proposition}
If AdvancedForceAxiom5 outputs a  D$_0$~graph, then this graph satisfies axiom 5.
If AdvancedForceAxiom5$(H)$ outputs FAILURE, then there is no  D$_5$ graph containing $H$.
\end{proposition}

Recall that in Example \ref{ex triple biject} we defined triples bijectivizations of~\,$\U/\Irkst$. Also of interest are graphs associated to these bijectivizations. These will be discussed in the next example and Section \ref{s Concluding remarks}.
\begin{definition}
\label{d triple D0}
For each assignment of Knuth or rotation to the subsets of size 3 of $[N]$ (as in Example \ref{ex triple biject}), define a D$_0$~graph on the vertex set $\W$ by requiring:
\[\text{the type of the KR square $\e{v\stars w}$ is }
\begin{cases}
\text{Knuth} & \text{if $\{\e{a,b,c}\}$ is a Knuth triple,}\\
\text{rotation} & \text{if $\{\e{a,b,c}\}$ is a rotation triple.}
\end{cases} \]
We say that a \emph{triples D$_0$~graph} is a union of connected components of such a graph.
\end{definition}

\begin{example}
To better understand the subset of D$_5$ graphs inside the set of all D$_0$~graphs, the following enumerative results are instructive.
\[
\begin{array}{l|l}
\text{Number of KR squares of $\S_n$} & \frac{(n-2)n!}{3!}\\[2mm]
\text{Number of D$_0$~graphs on $\S_n$} & 2^{\frac{(n-2)n!}{3!}}\\
\text{Number of triples D$_0$~graphs on $\S_n$} & 2^{\binom{n}{3} }\\[1mm]
\text{Number of D$_0$~graphs on $\S_6$} & 2^{24 \cdot 20}\\[1mm]
\text{Number of D$_5$ graphs on $\S_6$} & 2^{16\cdot 20}\cdot 12^{20}
\end{array}
\]

The first line comes from the fact that KR${}_i$ squares of  $\S_n$ are in bijection with words of length $n-3$ having no repeated letter, in the alphabet $\e{1,2,\dots,n}$.
A D$_0$~graph on  $\S_n$ is determined by choosing the type of each KR square of  $\S_n$ independently, hence the second line.
The number of triples D$_0$~graphs on  $\S_n$ is clear.

The number of D$_5$ graphs on $\S_n$ seems difficult to compute for $n > 6$, however the $n=6$ case is particularly simple because each KR square is contained in at most one
KR$_{i,j}$ hypercube.
There are $24\cdot 20$ KR squares of  $\S_6$. Axiom 5 imposes no conditions on the KR${}_3$ and KR${}_4$ squares. The same goes for the KR${}_2$ squares of the form $\e{\stars w}$, where $\e{w}$ is an increasing or decreasing word and for the KR${}_5$ squares of the form $\e{v\stars }$, where $\e{v}$ is an increasing or decreasing word.
This is a total of $16 \cdot 20$ KR squares.
There are $\binom{6}{3} = 20$  KR$_{2,5}$ hypercubes of $\S_6$
corresponding to the subsets of  $\{\e{1,2,3,4,5,6}\}$ of size 3, and these partition the
remaining $8\cdot 20$ KR squares into 20 sets of size 8.
By Proposition \ref{p axiom5 patterns} (v), there are 12 possibilities for the patterns of each of these hypercubes.
\end{example}

\subsection{Forcing axioms $4'b$ and 5}
\label{ss forcing 45}
We sketch the algorithm used to find a D~graph whose generating function is not Schur positive.

Let  $H$ be a partial  D$_0$~graph.
For each  $4 \le i \le n-2$, define $B_i(H) \subseteq \ver(H)$ to be the set of $\e{w} \in \ver(H)$ such that
\begin{itemize}
\item $E_{i-2}E_iE_{i-2}(\e{w}), E_iE_{i-2}(\e{w}),E_{i-2}(\e{w}), \e{w}, E_i(\e{w}), E_{i-2}E_i(\e{w})$ are well defined and distinct,
\item $\e{w}$ has $(i+1)$-type  W and  $E_{i-2}(\e{w})$ does not have  $(i+1)$-type W.
\end{itemize}
Define the following statistic on partial D$_0$~graphs $H$ taking values in  $\NN \sqcup \{\infty\}$:
\[\statb(H) :=
\begin{cases}
\sum_{i=4}^{n-2} |B_i(H)| &  \text{if $H$ satisfies axiom $4'b$,}\\
\infty &  \text{otherwise.}
\end{cases}
 \]

Given a partial D$_0$~graph $H$, let $\Q(H)$ be the graph with
\begin{itemize}
\item vertex set the undetermined KR squares of  $H$,
\item an edge for every pair of undetermined KR squares that appear together in a KR$_{i,j}$ hypercube.
\end{itemize}

The space of all  D$_5$ graphs containing a given partial  D$_0$~graph  $H$ is much too big to explore completely for one that satisfies axiom $4'b$.
To explore this space intelligently, we test the possibilities K or R for the undetermined KR squares of  $H$ while not violating axioms  $4'b$ and 5.  We work on one component of  $\Q(H)$ at a time, and for small components, we explore both possibilities K or R for the undetermined KR squares, using the statistic  $\statb$ to decide which ones to keep.
For concreteness, let us say that a subgraph of  $\Q(H)$ is  \emph{small} if it has $\leq 10$ vertices (see Remark \ref{r heuristics}).

\begin{myalgorithm}
\label{a force4'b etc}
In the functions below, the input $H$ is a partial  D$_0$~graph, $C$ and $\Q$ are induced subgraphs of the graph  $\Q(H)$, and $C$ is assumed to be connected.
Both functions return a partial D$_0$~graph containing  $H$ together with a statistic in $\NN \sqcup \{\infty\}$.
Moreover, if GrowDGraphOneQComponent (resp. GrowDgraph) returns  $\G, z$ with $z < \infty$, then  $\G$ is a partial D$_0$~graph that satisfies axioms  $4'b$ and  5 and has no undetermined KR squares belonging to the vertex set of $C$ (resp.  $\Q$).
In particular, if GrowDgraph($H, \Q(H)$) returns $\G, z$ with $z < \infty$, then  $\G$ is a D$_0$~graph and a D~graph.
The function GrowDGraphOneQComponent calls AdvancedForceAxiom5 from Algorithm \ref{a forceaxiom5 advanced}.

{\small
\begin{algorithm}
\DontPrintSemicolon
\RestyleAlgo{boxed}
\SetKwProg{Fn}{function}{}{end function}
\SetAlgoNoLine
\Fn{\emph{GrowDGraphOneQComponent}$(H, C)$}{
Choose a vertex  $X$ of  $C$, which is an undetermined KR square of  $H$.\;
Let $H^\text{K}$ (resp. $H^\text{R}$) be the graph obtained from  $H$ by setting the type of $X$ to be Knuth (resp. rotation) and adding the two corresponding edges to $H$.\;
$G^\text{K} := \text{AdvancedForceAxiom5}(H^\text{K})$\;
\If{$G^\text{K} =$ FAILURE}{
    \Return $H,\, \infty$ \;
}
Let $\Q^\text{K}$ be the graph obtained from $C$ by removing all the vertices which are no longer undetermined KR squares in $G^\text{K}$ ($\Q^\text{K}$ is an induced subgraph
of  $\Q(G^\text{K})$). \;
$G^\text{K},\, \stat^\text{K} := \text{GrowDGraph}(G^\text{K}, \Q^\text{K})$ \;
\eIf{$C$ is small}{
    $G^\text{R} := \text{AdvancedForceAxiom5}(H^\text{R})$\;
    \If{$G^\text{R} =$ FAILURE}{
        \Return $H,\, \infty$ \;
    }
    Let $\Q^\text{R}$ be the graph obtained from $C$ by removing all the vertices which are no longer undetermined KR squares in $G^\text{R}$. \;
    $G^\text{R},\, \stat^\text{R} := \text{GrowDGraph}(G^\text{R}, \Q^\text{R})$ \;
    \eIf{$\stat^\text{K} \leq \stat^\text{R}$}{
        \Return  $G^\text{K},\, \stat^\text{K}$\;}{
        \Return $G^\text{R},\, \stat^\text{R}$ \;}
}{
    \eIf{$\stat^\text{K} < \infty$}{
        \Return $G^\text{K},\, \stat^\text{K}$ \;
    }{
    \Return GrowDGraph$(G^\text{R}, \Q^\text{R})$ \;
    }
}
}
\end{algorithm} }

{\small
\begin{algorithm}
\DontPrintSemicolon
\RestyleAlgo{boxed}
\SetKwProg{Fn}{function}{}{end function}
\SetAlgoNoLine
\Fn{\emph{GrowDGraph}$(H, \Q)$}{
Let  $C_1, \dots, C_t$ be the components of  $\Q$, in increasing order of size ($\Q$ may be empty in which case $t := 0$).\;
\For{$i = 1\text{ to } t$}{
  $H,\, \stat := $ GrowDGraphOneQComponent$(H, C_i)$\;
   \If{$\stat = \infty$}{
    \Return $H,\, \infty$ \;
    }
}
    \Return $H,\, \statb(H)$ \;
}
\end{algorithm} }
\end{myalgorithm}

\begin{remark}\label{r heuristics}
In practice, we used complicated heuristics to determine the condition that $C$ is small and to choose a vertex of $C$ in the algorithm above.
Also, we used a slightly simpler and faster version of AdvancedForceAxiom5.
\end{remark}

\begin{remark}\label{r fancy KR set to D0}
Instead of calling GrowDGraph with the minimal partial D$_0$~graph  $H^*$ on  $\bar{W}^*$, it is better to call it with a smaller partial  D$_0$~graph constructed as follows:
let  $H'$ be the  D$_0$~graph obtained by setting
each undetermined KR square of  $H^*$ to type Knuth.
Let  $V$ be the vertex set of the largest component of $H'$ (there is a unique largest component); the size of  $V$ is 6026, whereas the size of  $\bar{W}^*$ is 39696.
Let $H^*_{6026}$ be the minimal partial  D$_0$-graph on  $V$ (by Proposition \ref{p component KR set} (i) and Proposition-Definition \ref{pd KR set}, $V$ is a KR set).
The generating function of this graph is not Schur positive.
\end{remark}

With $H^*_{6026}$ as in Remark \ref{r fancy KR set to D0}, the function call
GrowDGraph$(H^*_{6026}, \Q(H^*_{6026}))$, or rather its modification with the heuristics and caveats of Remark \ref{r heuristics},
outputs a D~graph.
The largest component  $\G^*$ of this graph
is a D~graph whose generating function is not Schur positive.
The graph  $\G^*$ is described in the data files accompanying this paper.
See \textsection\ref{ss accompanying data files} for a guide to these files and some basic properties of $\G^*$.


\subsection{A strengthening of axiom $4'b$}
\label{ss axiom 4''b}
We also considered a stronger version of axiom $4'b$, which we call axiom $4''b$.
\begin{definition}
Let $\G$ be a signed, colored graph of degree  $n$ satisfying axiom 0.  For
$4 \le i < n$, a \emph{weak flat $i$-chain} of  $\G$ is a sequence $(x^1, x^2,\ldots, x^{2h-1}, x^{2h})$ of distinct vertices admitting $i-2$-neighbors
such that
\begin{itemize}
\item $x^{2j-1} = E_i(x^{2j})$,
\item $x^{2j+1}$ is any vertex on the $(i-2)$-$(i-1)$-string containing $x^{2j}$
such that $E_{i-2}(x^{2j+1})$ does not have $(i-1)$-type W.
\end{itemize}
\end{definition}

\begin{remark}
The vertices $x^{2j}$ and $x^{2j+1}$ in a weak flat  $i$-chain are related just as in Remark \ref{r flat chain}, except that if $H$ is a path  $(y^1,\dots,y^5)$ with 4 edges as in that remark, then exactly one of the following possibilities must hold:
\begin{itemize}
\item $x^{2j} = y^2$, $x^{2j+1} = y^5$,
\item $x^{2j} = y^4$, $x^{2j+1} = y^5$,
\item $x^{2j} = y^5$, $x^{2j+1} = y^4$,
\item $x^{2j} = y^2$, $x^{2j+1} = y^4$,
\item $x^{2j} = y^3$, $x^{2j+1} = y^4$,
\item $x^{2j} = y^3$, $x^{2j+1} = y^5$.
\end{itemize}
\end{remark}

Axiom $4''b$ is the same as axiom $4'b$ except with weak flat  $i$-chain in place of flat $i$-chain:
\begin{list}{}{\usecounter{ctr} \setlength{\itemsep}{2pt} \setlength{\topsep}{3pt}}
\item[(axiom $4''b$)] if $(x^1,x^2,\ldots,x^{2h-1},x^{2h})$ is a weak flat $i$-chain and $x^j$ has $(i+1)$-type W for some $2 < j < 2h-1$, then either $x^1, \dots, x^j$
have $(i+1)$-type W or  $x^j, \dots, x^{2h}$ have  $(i+1)$-type W (or both).
\end{list}

Any flat $i$-chain is a weak flat $i$-chain, hence any signed, colored graph satisfying axiom $4''b$ also satisfies axiom $4'b$. The D$_0$~graph $\G^*$ constructed in \textsection\ref{ss forcing 45} satisfies axiom $4''b$ in addition to axioms 5 and $4'b$.

\subsection{Accompanying data files}
\label{ss accompanying data files}
The graph  $\G^*$ constructed in \textsection\ref{ss forcing 45} has 4950 vertices and its generating function is
\[6 s_{4 2 1 1}
+ 12 s_{4 1 1 1 1}
+ s_{3 3 2}
+ 3 s_{3 3 1 1}
+ 13 s_{3 2 2 1}
+ 37 s_{3 2 1 1 1}
+ 12 s_{3 1 1 1 1 1}
- s_{2 2 2 2}
+ 8 s_{2 2 2 1 1}
+ 2 s_{2 2 1 1 1 1}.\]
Figure \ref{f flat chain} depicts a flat 4-chain of  $\G^*$ of length 10 and some of its neighboring vertices.

Let $V^*$ be the vertex set of $\G^*$.
The accompanying files present $\G^*$ in several different formats and contain Maple code to check that $\G^*$ satisfies axioms $4'b$ and 5.
\begin{list}{(\alph{ctr})}{\usecounter{ctr} \setlength{\itemsep}{2pt} \setlength{\topsep}{3pt}}
\item {\tt vertexset.txt} gives the vertex set $V^* \subseteq \S_8$ as a list of lists.
\item {\tt involutions.txt} gives the edges of $\G^*$ as involutions $E_{2},\ldots, E_{7}$ of $\{1,\ldots, 4950\} \cong V^*$, where the isomorphism is given by the ordering of vertices in (a).
\item {\tt involutionswithzeros.txt} is the same as (b) except $E_i(\e{v})$ is set to 0 if $\e{v}$ does not admit an $i$-neighbor.
\item {\tt maplegraph.txt} presents $\G^*$ as a Maple graph (no edge colors or types included).
\item {\tt checkaxioms.txt} contains Maple code to check that $\G^*$ is a D$_0$~graph and satisfies axioms $4'b$ and 5. Hence it is a D~graph by Theorem~\ref{t D0 satisfy 123etc}.
We also provide code to verify that the coefficient of  $s_{(2,2,2,2)}(\mathbf{x})$ in the generating function of $\G^*$ is  $-1$. This file contains the data in {\tt vertexset.txt} and {\tt involutionswithzeros.txt}.
\end{list}

\begin{remark}
Checking that $\G^*$ is a D~graph only depends on the code in
{\tt checkaxioms.txt} and the material in Sections~{\ref{s Noncommutative symmetric functions}--\ref{s Basic properties of D0 graphs}}
(more specifically, Theorem~\ref{t D0 satisfy 123etc} and the definition of a D~graph in \textsection\ref{ss Preliminary definitions for D graphs}--\ref{ss D graphs}).  (The material in Sections \ref{s Linear programming and Schur positivity}--\ref{s A D graph whose generating function is not Schur positive} is important for understanding how  $\G^*$ was found, but is not needed for checking it is a D~graph.)
\end{remark}

\begin{figure}
        \centerfloat
\begin{tikzpicture}[xscale = 2.2,yscale = 1.5]
\tikzstyle{vertex}=[inner sep=0pt, outer sep=4pt]
\tikzstyle{framedvertex}=[inner sep=2pt, outer sep=3pt, draw=gray]
\tikzstyle{aedge} = [draw, thin, ->,black]
\tikzstyle{edge} = [draw, thick, -,black]
\tikzstyle{doubleedge} = [draw, thick, double distance=1pt, -,black]
\tikzstyle{hiddenedge} = [draw=none, thick, double distance=1pt, -,black]
\tikzstyle{dashededge} = [draw, very thick, dashed, black]
\tikzstyle{LabelStyleH} = [text=black, anchor=south]
\tikzstyle{LabelStyleHn} = [text=black, anchor=north]
\tikzstyle{LabelStyleD} = [text=black, inner sep=2pt, anchor=south east]
\tikzstyle{LabelStyleDn} = [text=black, inner sep=2pt, anchor=north west]
\tikzstyle{LabelStyleV} = [text=black, anchor=east]
\tikzstyle{LabelStyleVn} = [text=black, anchor=west]

\node[framedvertex] (v4) at (3,3){\footnotesize$\e{18576243} $};
\node[framedvertex] (v5) at (3,4){\footnotesize$\e{18657243} $};
\node[framedvertex] (v11) at (2,4){\footnotesize$\e{81657243} $};
\node[framedvertex] (v12) at (4,3){\footnotesize$\e{81576243} $};
\node[framedvertex] (v15) at (2,5){\footnotesize$\e{81675243} $};
\node[framedvertex] (v18) at (4,2){\footnotesize$\e{81756243} $};
\node[framedvertex] (v25) at (1,5){\footnotesize$\e{68175243} $};
\node[framedvertex] (v29) at (7,2){\footnotesize$\e{78516243} $};
\node[framedvertex] (v35) at (1,6){\footnotesize$\e{68715243} $};
\node[framedvertex] (v37) at (7,1){\footnotesize$\e{78561243} $};

\node[vertex] (v2) at (4,5){\footnotesize$\e{18652743} $};
\node[vertex] (v3) at (3,2){\footnotesize$\e{15876243} $};
\node[vertex] (v8) at (3,5){\footnotesize$\e{81652743} $};
\node[vertex] (v14) at (2,3){\footnotesize$\e{86157243} $};
\node[vertex] (v19) at (1,4){\footnotesize$\e{67815243} $};
\node[vertex] (v21) at (6,2){\footnotesize$\e{75816243} $};
\node[vertex] (v23) at (5,3){\footnotesize$\e{81762543} $};
\node[vertex] (v24) at (5,2){\footnotesize$\e{78156243} $};
\node[vertex] (v30) at (4,1){\footnotesize$\e{87156243} $};
\node[vertex] (v34) at (8,1){\footnotesize$\e{75861243} $};
\node[vertex] (v36) at (0,6){\footnotesize$\e{86715243} $};
\node[vertex] (v38) at (2,6){\footnotesize$\e{68751243} $};
\node[vertex] (v40) at (0,7){\footnotesize$\e{86175243} $};

\draw[edge] (v2) to node[LabelStyleH]{\Tiny$2 $} (v8);
\draw[doubleedge] (v2) to node[LabelStyleD]{\Tiny$5 $} (v5);
\draw[hiddenedge] (v2) to node[LabelStyleDn]{\Tiny$6 $} (v5);
\draw[edge] (v3) to node[LabelStyleV]{\Tiny$3 $} (v4);
\draw[edge] (v4) to node[LabelStyleV]{\Tiny$\tilde{4} $} (v5);
\draw[edge] (v4) to node[LabelStyleH]{\Tiny$2 $} (v12);
\draw[edge] (v5) to node[LabelStyleH]{\Tiny$2 $} (v11);
\draw[doubleedge] (v8) to node[LabelStyleD]{\Tiny$5 $} (v11);
\draw[hiddenedge] (v8) to node[LabelStyleDn]{\Tiny$6 $} (v11);
\draw[edge] (v11) to node[LabelStyleV]{\Tiny$3 $} (v14);
\draw[edge] (v11) to node[LabelStyleV]{\Tiny$4 $} (v15);
\draw[edge] (v12) to node[LabelStyleV]{\Tiny$4 $} (v18);
\draw[edge] (v15) to node[LabelStyleH]{\Tiny$\tilde{2} $} (v25);
\draw[edge] (v18) to node[LabelStyleH]{\Tiny$\tilde{2} $} (v24);
\draw[edge] (v18) to node[LabelStyleV]{\Tiny$3 $} (v30);
\draw[edge] (v18) to node[LabelStyleH]{\Tiny$\tilde{5} $} (v23);
\draw[edge] (v19) to node[LabelStyleV]{\Tiny$\tilde{3} $} (v25);
\draw[edge] (v21) to node[LabelStyleH]{\Tiny$2 $} (v29);
\draw[edge] (v21) to node[LabelStyleH]{\Tiny$\tilde{3} $} (v24);
\draw[edge] (v25) to node[LabelStyleV]{\Tiny$4 $} (v35);
\draw[doubleedge] (v29) to node[LabelStyleV]{\Tiny$4 $} (v37);
\draw[hiddenedge] (v29) to node[LabelStyleVn]{\Tiny$5 $} (v37);
\draw[doubleedge] (v34) to node[LabelStyleH]{\Tiny$2 $} (v37);
\draw[hiddenedge] (v34) to node[LabelStyleHn]{\Tiny$3 $} (v37);
\draw[edge] (v35) to node[LabelStyleH]{\Tiny$5 $} (v38);
\draw[edge] (v35) to node[LabelStyleH]{\Tiny$2 $} (v36);
\draw[edge] (v36) to node[LabelStyleV]{\Tiny$4 $} (v40);
\end{tikzpicture}
\caption{\label{f flat chain} A flat 4-chain of $\G^*$ of length 10. The words that belong to the chain are outlined. All of the 3-neighbors and 5-neighbors of the vertices in the chain are also shown as well as the vertices along the 2-3-strings connecting consecutive vertices in the chain. Only the last two vertices of the chain, $\e{78516243}$ and $\e{78561243}$, have 5-type W, hence this does not violate axiom $4'b$.}
\end{figure}

\section{Concluding remarks}
\label{s Concluding remarks}
The following table summarizes our current knowledge of Schur positivity for the generating functions of various classes of D$_0$~graphs.

\[\begin{array}{ll}
\text{Class of D$_0$~graphs}                          & \text{Schur positive} \\
\hline \text{All D$_0$~graphs}                        & \text{No} \\[.5mm]
\text{D$_0$~graphs satisfying axiom 5}                & \text{No} \\[.5mm]
\text{D$_0$~graphs satisfying axioms $4'b$ and 5}       & \text{No} \\[.5mm]
\text{D$_0$~graphs satisfying axioms $4''b$ and 5}      & \text{No} \\[.5mm]
\text{All triples D$_0$~graphs (Definition \ref{d triple D0})}         & \text{Unknown}\\[.5mm]
\text{The Assaf LLT$_k$ D~graphs (Definition \ref{d Assaf D graph})}    & \text{Unknown}\\[.5mm]
\text{D$_0$~graphs $\G$ such that $\sum_{\e{w} \in \ver(\G)}\e{w} \in (\Ilamst{\le 3})^\perp$} & \text{Yes} \\[.5mm]
\text{The graphs $\G_k^\stand(\bm{\beta}, t)$ (see \textsection\ref{ss Assaf's LLT bijectivizations})} & \text{Yes}\\[.5mm]
\end{array}\]

Here, $\Ilamst{\le k}$ is the two-sided ideal of~\,$\U$ corresponding to the relations \eqref{e Irkst2} together with
\begin{alignat*}{2}
\e{ac} &= \e{ca} \qquad &&\text{for $c-a > k$,} \\
\e{b(ac-ca)} &= \e{(ac-ca)b} \qquad &&\text{for  $a < b < c$ and $c-a \le k$.}
\end{alignat*}
Then $\Ilamst{\le k} \supseteq \Irkst$.
In \cite{BLamLLT}, we show that
 $\mathfrak{J}_\lambda(\mathbf{u})$ is a positive sum of monomials in  $\U/\Ilamst{\le k} $,
which explains the second-to-last entry of the table. Also, the last entry of the table follows from the Schur positivity of LLT polynomials.

Despite the negative results of this paper, there is substantial evidence that many D$_0$~graphs have Schur positive generating functions.
This includes computer experimentation of many researchers, the ``Yes'' entries in the table above, and to some extent the ``Unknown'' entries because substantial computer searches support Schur positivity in these cases.
In addition, Theorem~\ref{t fat hook} can be generalized to show that the coefficient~of~$s_{\lambda}(\mathbf{x})$ in the generating function of any
D$_0$~graph is nonnegative whenever $\lambda $ is a hook shape or of the form $(a,2,1^b)$.
It remains a major open question to give some uniform explanation of these observations.
Before formulating some specific questions along these lines, let us see exactly how the negative results of this paper make such an explanation difficult.

For the sake of concreteness, let us define a \emph{local property} of a D$_0$~graph  $\G$ to be one that only depends on $\myRes_K \G$ for  $|K| \le 6$. Then axioms 3, $4'a$, $4'b$, $4''b$ and  $\LSP_4$ and  $\LSP_5$ are local in this sense.
Understanding the Schur expansions of  D$_0$~graphs using the theory of noncommutative Schur functions  is well adapted to studying such properties.
For example, if $I$ is  a two-sided ideal of~\,$\U$ generated in degree  $\le 6$ and  $I \supseteq \Irkst$, then the property that a D$_0$~graph $\G$ satisfies  $\sum_{\e{w} \in \ver(\G)} \e{w} \in I^\perp$ is a local property.
These are the natural local properties from this perspective.

However, \textsection\ref{ss solution to lltlp} shows  the limitations of such properties for proving Schur positivity.
If we want to prove that all Assaf LLT$_k$ D~graphs have Schur positive generating functions using a local property corresponding to an ideal $I$ as above, then we must prove that
\begin{list}{(\Alph{ctr})}{\usecounter{ctr} \setlength{\itemsep}{2pt} \setlength{\topsep}{3pt}}
\item for all $k$ and all Assaf LLT$_k$ D~graphs $\G$, there holds $\sum_{\e{w} \in \ver(\G)} \e{w} \in I^\perp$,
\item $\Delta(\sum_{\e{w} \in W} \e{w})$ is Schur positive for every KR set $W$ such that $\sum_{\e{w} \in W} \e{w} \in I^\perp$.
\end{list}
Then we must have  $I^\perp \supseteq (\Iassafst{k})^\perp$ for all  $k$, hence  $I^\perp \supseteq J^\perp$ for  $J = \bigcap_j J^{(j)}$ with  $J^{(j)}$ as in \textsection\ref{ss solution to lltlp}.  But then (B) is false. Hence we must turn to nonlocal properties, properties that depend on edges and not just vertices, or some ``non-uniform'' explanation of Schur positivity that handles each $k$ separately.

Here are some precise questions to guide us towards a better understanding of what makes a  D$_0$~graph Schur positive:
\begin{itemize}
\item In which quotients $\U/I$ of~\,$\U/\Irkst$ can $\mathfrak{J}_\lambda(\mathbf{u})$ be written as a positive sum of monomials? We believe the case $I$ is generated in degree 3 to be particularly important.
\item For which quotients $\U/I$ of~\,$\U/\Irkst$ do all D$_5$ graphs $\G$ satisfying $\sum_{\e{w} \in \ver(\G)} \e{w} \in I^\perp$ have Schur positive generating functions?
\end{itemize}
The negative results of this paper suggest that answers to these questions will not be simple and that there may be no unique, most general condition that guarantees Schur positivity.

\section*{Acknowledgments}
I am extremely grateful to Sergey Fomin for his valuable insights and guidance on this project and to John Stembridge for his generous advice and many detailed discussions.  I thank Sami Assaf, Sara Billey, and Jennifer Morse for valuable discussions and Elaine So and Xun Zhu for help typing and typesetting figures.

\bibliographystyle{plain}
\def\cprime{$'$} \def\cprime{$'$} \def\cprime{$'$}

\end{document}